\documentclass[10pt,twoside]{article}
\usepackage{amsmath,amssymb,amsthm,mathptmx,hyperref,nicefrac}
\usepackage{titlesec}
\usepackage{indentfirst}
\usepackage[inner=2.4cm,outer=2.65cm,bottom=3cm]{geometry}
\titleformat*{\section}{\normalsize\bfseries}
\titleformat*{\subsection}{\normalsize\itshape}
\numberwithin{equation}{section}
\theoremstyle{remark}
\newtheorem{remark}{Remark}[section]
\theoremstyle{definition}
\newtheorem{definition}{Definition}[section]

\newtheorem{proposition}{Proposition}[section]

\newtheorem{corollary}{Corollary}[section]
\newtheorem{theorem}{Theorem}[section]
\newcommand{\f}[1]{\pmb{#1}}
\DeclareMathOperator{\N}{\mathbb{N}}
\DeclareMathOperator{\R}{\mathbb{R}}
\DeclareMathOperator{\C}{\mathcal{C}}
\DeclareMathOperator{\F}{\mathcal{F}}
\DeclareMathOperator{\AC}{\mathcal{AC}}

\DeclareMathOperator{\V}{\f H^1_{0,\sigma}}
\DeclareMathOperator{\Vd}{(\f H^{1}_{0,\sigma})^*}

\DeclareMathOperator{\Ha}{\f L^2_{\sigma}}
\DeclareMathOperator{\Hb}{\f{H}^1_0}
\DeclareMathOperator{\Hc}{\f H^2}
\DeclareMathOperator{\Hd}{\f H^{-1}}
\DeclareMathOperator{\He}{\f{H}^1}
\DeclareMathOperator{\Le}{\f{L}^2}
\DeclareMathOperator{\Hf}{\f{H}^3}
\DeclareMathOperator{\Hg}{\f{H}^4}
\DeclareMathOperator{\Hi}{\f{H}^2}

\newcommand{\Hrand}[1]{{\f H^{\nicefrac{#1}{2}}(\partial\Omega)}}
\newcommand{\rot}[1]{[ #1 ]_{\f X}}
\newcommand{\rott}[1]{[ #1 ]_{-\f X}}

\DeclareMathOperator{\Sr}{\mathbb{E}}

\DeclareMathOperator{\Se}{\mathbb{S}}

\DeclareMathOperator{\M}{\mathcal{M}}

\DeclareMathOperator{\ra}{\rightarrow}
\DeclareMathOperator{\de}{\text{d}}
\DeclareMathOperator{\tr}{tr}

\DeclareMathOperator{\spa}{span}

\DeclareMathOperator{\sgn}{sgn}

\DeclareMathOperator{\esssup}{ess\,sup}

\newcommand{\dreidots}{\text{\,\multiput(0,-2)(0,2){3}{$\cdot$}}\,\,\,\,}
\newcommand{\dreidotkom}{\text{\,\multiput(0,0)(0,2){2}{$\cdot$}\put(0,0){,}}\,\,\,\,}

\newcommand{\ft}[1]{{\tilde{\mathbf{ #1}}}}

\newcommand{\pat}[2]{\frac{\partial #1}{\partial #2}} 
\DeclareMathOperator{\di}{\nabla \cdot}

\newcommand{\ov}[1]{\overline{#1}}

\DeclareMathOperator{\curl}{\nabla \times}

\newcommand{\intt}[1]{\int_{0}^T\left({ #1}\right) \de t}
\newcommand{\intte}[1]{\int_{0}^T{ #1} \de t}

\newcommand{\intet}[1]{\int_{\Omega}{ #1} \de \f x}
\renewcommand{\ll}[1]{\langle\hspace{-0.75mm}\langle{#1}\rangle\hspace{-0.75mm}\rangle}
\newcommand{\ftn}[1]{\tilde{\f {{#1}}}_{n,\delta}}
\newcommand{\fnt}[1]{\tilde{\f {{#1}}}_{n,\delta}}

\newcommand{\fdk}[1]{\f {{#1}}_{\delta_k}}

\newtheorem{rem}{Remark}
\newtheorem{defi}{Definition}

\DeclareMathOperator{\sym}{{sym}}
\DeclareMathOperator{\Lap}{\Delta}
\DeclareMathOperator{\Lapp}{\Delta^2}
\DeclareMathOperator{\skw}{skw}
\newcommand{\sy}[1]{(\nabla \f #1)_{{\sym}}}
\newcommand{\sk}[1]{(\nabla \f #1)_{\skw}}
\renewcommand{\t}{\partial_t  }
\newcommand{\syn}[1]{(\nabla \fn {#1})_{\sym}}
\newcommand{\skn}[1]{(\nabla \fn {#1})_{\skw}}
\newcommand{\syd}[1]{(\nabla \fd {#1})_{\sym}}
\newcommand{\skd}[1]{(\nabla \fd {#1})_{\skw}}
\newcommand{\fn}[1]{\f {{#1}}_{n,\delta}}
\newcommand{\fd}[1]{\f {{#1}}_{\delta}}

\author{%
\textsc{ Robert Lasarzik}\thanks{%
        Technische Universit\"{a}t Berlin,
        Institut f\"{u}r Mathematik,
        Stra{\ss}e des 17.~Juni 136,
        10623 Berlin, Germany
        \newline{\tt lasarzik@math.tu-berlin.de}
        }}
\title{
\begin{Large}
\textbf{Measure-valued solutions to the Ericksen--Leslie model equipped with the Oseen--Frank energy}\end{Large}\footnote{This work was funded by CRC 901 {\em Control of self-organizing nonlinear systems: Theoretical methods and concepts of application} (Project A8)\/.
}}
\begin{document}
\markboth{Measure-valued solutions to the Ericksen--Leslie model}{R.~Lasarzik}
\date{Version \today}
\maketitle
\begin{abstract}
\normalsize
In this article, we prove the existence of measure-valued solutions to the Ericksen--Leslie system equipped with the Oseen--Frank energy. We introduce the concept of generalized gradient Young measures. Via a Galerkin approximation, we show the existence of weak solutions to a regularized system and attain measure-valued solutions for vanishing regularization. 
Additionally, it is shown that the measure-valued solution fulfills an energy inequality.
\newline
\newline
{\em Keywords:
Liquid crystal,
Ericksen--Leslie equation,
Existence,
Measure-valued solution,
generalized Young measure
}
\newline
{\em MSC (2010): 35Q35, 35K65, 76A15
}
\end{abstract}
\tableofcontents
%%%%%%%%%%%%%%%%%%%%%%%%%%%%%%%
%% Introduction
%%%%%%%%%%%%%%%%%%%%%%%%%%%%%%%
\section{Introduction}\label{sec:int}
Nonlinear partial differential equations require generalized solution concepts. In this context, the concept of \textit{Young measure-valued solutions} was first introduced by Tartar~\cite{tartar}. Later on, the concept of \textit{generalized Young measures} was used by DiPerna and Majda~\cite{DiPernaMajda} to define generalized solutions to the Euler equations.  	
These generalized Young measures capture oscillation and concentration effects for sequences bounded in $L^1$. Another step in the analysis of such sequences and their limits under nonlinear functions has been achieved by Alibert and Bouchitt{\'e}~\cite{alibert} who observed that concentrations can only occur almost everywhere.
In the article at hand, we further generalize these concepts to prove global existence of \textit{measure-valued solutions} to the \textit{Ericksen--Leslie system} describing \textit{nematic liquid crystal flow}.

Nematic liquid crystals are anisotropic fluids. They consist of rod-like molecules that build or are dispersed in a fluid and are directionally ordered. This ordering and its direction heavily influences the properties of the material such as light scattering or flow behavior. This gives rise to many applications, among which \textit{liquid crystal displays} are only the most prominent one. The Ericksen--Leslie model is a generally accepted model to describe nematic liquid crystals. The direction of the aligned molecules is modeled by a unit-vector field and the fluid flow by a 
velocity field. Since this model has been proposed  in the 60s by Ericksen~\cite{Erick2} and Leslie~\cite{leslie}, it has been extensively studied. Nevertheless, the global mathematical existence theory is restricted to simple quadratic free energies. 

In this article, we propose a remedy by introducing a new concept of solutions, the so-called {measure-valued solutions}. This is a rather weak notion of solutions, but in~\cite{weakstrong}, we show that the presented solutions enjoy the \textit{weak-strong uniqueness} property. They coincide with the local strong solution as long as the latter exists. Thus, the concept of measure-valued solutions is a natural generalization of the classical strong solutions. 
%%%%%%%%%%%%%%%%%%%%%%%%%%
%% Review 
%%%%%%%%%%%%%%%%%%%%%%%%%%%%

The first mathematical analysis of  a simplified Ericksen--Leslie model is due to Lin and Liu~\cite{linliu1}. They show global existence of weak solutions and local existence of strong solutions. Additionally, they manage to generalize these results to a more realistic model~\cite{linliu3}. They also show partial regularity of weak solutions to the considered system~\cite{linliu2}. Following this work, there have been many articles considering slightly more complicated models, for example~\cite{prohl},~\cite{allgemein},~or~\cite{isothermal}. Nevertheless to the best of the author's knowledge, the only generalization with respect to the free energy potential is   performed by Emmrich and the author in~\cite{unsere}. 

There are also results on the \textit{local existence of solutions} to the full Ericksen--Leslie model,  see~\cite{localin3d},~\cite{recent} or~\cite{Pruess2}. Especially, local strong solutions are known to exist to different simplifications of the system considered in this article.
The full (thermodynamically consistent) Ericksen--Leslie system equipped with the Dirichlet energy
is considered in~\cite{Pruess2}, whereas the simplified Ericksen--Leslie system with the full Oseen--Frank energy is studied in~\cite{localin3d} as well as in~\cite{localoseen}. 
Since finite time singularities in nematic liquid crystals have been observed  experimentally~\cite{singul2} and analytically~\cite{blow}, it seems appropriate to investigate a weakened solution concept such as measure-valued solutions.

We also want to mention the article by Brenier, De Lellis and Sz{\'e}kelyhidi~\cite{weakstrongeuler} showing the weak-strong uniqueness of measure-valued solutions to the Euler equation, because the techniques introduced there can be transferred to the setting presented here to show additional properties of the limiting measures, as well as the weak-strong uniqueness in~\cite{weakstrong}.

%%%%%%%%%%%%%%%%%%%%%%%%%%%
%% Outline
%%%%%%%%%%%%%%%%%%%%%%%%%%%%
\subsection{Outline of the paper}
%%%
In this paper, we study the Ericksen--Leslie model in three dimensions equipped with the Oseen--Frank free energy. 
This energy is not convex and the existence theory is non-standard and involves generalized gradient Young measures.
Already Leslie
suggests to equip the model with the Oseen--Frank energy. It can be seen as the physically most relevant free energy function.

The paper is organized as follows: In Section~\ref{sec:not}, we collect some notation. Section~\ref{sec:model} contains the model, the definition of generalized solutions, and the main results.
In Section~\ref{sec:young}, we introduce the concept of generalized gradient Young measures and prove the associated main theorem. While Section~\ref{sec:weak} is devoted to the proof of existence of weak solutions to the regularized system, Section~\ref{sec:convmeas} shows the convergence of these weak solutions to measure-valued solutions for vanishing regularization. In the last section (Section~\ref{sec:add}), we  show additional properties of the measure-valued solutions such as additional strong convergences of the norm of the director as well as an energy inequality. The energy inequality is a necessary tool to obtain the weak-strong uniqueness of solutions.

%%%%%%%%%%%%%%%%%%%%%%%
%%%%%%%%%%%%%%%%%%%%%%%%%
%%%%%%%%%%%%%%%%%%%%%%%%%
%%%%%%%%%%%%%%%%%%%%%%%%%%
%%%%%%%%%%%%%%%%%%%%%%%%%%%

%%%%%%%%%%%%%%%%%%%%%%%%%%%
%% Mathematics
%%%%%%%%%%%%%%%%%%%%%%%%%%%
\subsection{Notation\label{sec:not}}
Vectors of $\R^3$ are denoted by bold small Latin letters. Matrices of $\R^{3\times 3}$ are denoted by bold capital Latin letters. We also use tensors of higher order, which are denoted by bold capital Greek letters.
Moreover, numbers are denoted be small Latin or Greek letters, and capital Latin letters are reserved for potentials.
The euclidean scalar product in $\R^3$ is denoted by a dot, $ \f a \cdot \f b : = \f a ^T \f b = \sum_{i=1}^3 \f a_i \f b_i$  for $ \f a, \f b \in \R^3$ and the Frobenius product in $\R^{3\times 3}$ by a colon $ \,\f A: \f B:= \tr ( \f A^T \f B)= \sum_{i,j=1}^3 \f A_{ij} \f B_{ij}$ for $\f A , \f B \in \R^{3\times 3}$.
Additionally, the scalar product in the space of tensors of order three is denoted by three dots, 
\begin{align*}
\f \Upsilon \dreidots\, \f \Gamma : =\left [ \sum_{j,k,l=1} ^3 \f \Upsilon_{jkl} \f \Gamma_{jkl}\right ], \quad    \f \Upsilon \in \R^{3\times 3 \times 3 },  \, \f \Gamma \in \R^{3\times 3 \times 3}  .
\end{align*}
The associated norms are all denoted by $| \cdot |$, as well as the norms of tensors of higher order, 
\begin{align*}
| \f \Lambda|^2 := \sum_{i,j,k,l=1}^3 
\f \Lambda_{ijkl}^2\,\quad\text{for }\f \Lambda \in \R^{3^4} \quad 
\text{and }\quad| \f \Theta |^2  := \sum_{i,j,k,l,m,n=1}^3 \f \Theta ^2_{ijklmn}\,\quad\text{for } 
\f \Theta \in \R^{3^6}\,
\end{align*}
respectively.
Similar, we define the products of tensors of different order.
The product of a tensor of third order and a matrix and a  vector is defined by
\begin{align*}
\f \Gamma : \f A := \left [ \sum_{j,k=1}^3 \f \Gamma_{ijk}\f A_{jk}\right ]_{i=1}^3\, ,  \,   \f \Gamma \cdot \f A := \left [ \sum_{k=1}^3 \f \Gamma_{ijk}\f A_{kl}\right ]_{i,j,l=1}^3\, ,  \,    \f \Gamma \cdot \f a  := \left [ \sum_{k=1}^3 \f \Gamma_{ijk}\f a_k\right ]_{i,j=1}^3 \, , \, \f \Gamma\in \R^{3 \times 3\times 3 } , \, \f A \in \R^{3\times 3},\,\f a \in \R^3 .
\end{align*}
%\rand{Hier noch produkte vierter Stufe mit zweiter und erster Stufe...}
The product of a tensor of fourth order with a matrix and a vector is defined by
\begin{align*}
\f \Lambda : \f A : =\left [ \sum_{k,l=1} ^3 \f \Lambda_{ijkl} \f A_{kl}\right ]_{i,j=1}^3\,, \, 
\f \Lambda : \f a : =\left [ \sum_{l=1} ^3 \f \Lambda_{ijkl} \f a_{l}\right ]_{i,j,k=1}^3\,, 
 \,    \f \Lambda \in \R^{3^4 },  \,  \f A \in \R^{3\times 3 } \, 
 \f a 
 \in 
 \R^3 .
\end{align*}
%%%%%%%%%%%
The product of tensors of fourth and third order is given by
\begin{align*}
\f \Lambda  : \f \Gamma : =\left [ \sum_{k,l=1} ^3 \f \Lambda_{ijkl} \f \Gamma _ {klm}\right ]_{i,j,m=1}^3 , \, \f \Lambda  \dreidots \f \Gamma : =\left [ \sum_{j,k,l=1} ^3 \f \Lambda_{ijkl} \f \Gamma_{jkl}\right ]_{i=1}^3, \,    \f \Lambda \in \R^{3^4},  \, \f  \Gamma \in \R^{3\times 3\times 3}  .
\end{align*}
The product of a tensor of fourth order and a matrix or a tensor of third order is defined via
\begin{align*}
 \f A : \f \Theta : ={}& \left [ \sum_{i,j=1} ^3 \f A_{ij} \f \Theta_{ijklmn}  \right ]_{k,l,m,n=1}^3 , \, \f \Theta \dreidots \f \Gamma : = \left [ \sum_{l,m,n=1} ^3 \f \Theta_{ijklmn} \f \Gamma_{lmn}\right ]_{i,j,k=1}^3 ,  \,  
  \f \Theta \in \R^{3^6},\f A \in \R^{3\times 3} ,\f \Gamma \in \R^{3\times 3 \times 3}\,.
\end{align*}
The product of a vector and a tensor of fourth order is defined differently. The definition is adjusted to the cases of this work:
 \begin{align*}
\f a \cdot \f \Theta :={}& \left [ \sum_{k=1} ^3 \f a_{k} \f \Theta_{ijklmn}  \right ]_{i,j,l,m,n=1}^3,
\, 
  \f \Theta \in \R^{3^6},\f a \in \R^{3} \,.
\end{align*}
The standard matrix and matrix-vector multiplication is written without an extra sign for bre\-vi\-ty,
$$\f A \f B =\left [ \sum _{j=1}^3 \f A_{ij}\f B_{jk} \right ]_{i,k=1}^3 \,, \quad  \f A \f a = \left [ \sum _{j=1}^3 \f A_{ij}\f a_j \right ]_{i=1}^3\, , \quad  \f A \in \R^{3\times 3},\,\f B \in \R^{3\times3} ,\, \f a \in \R^3 .$$
The outer vector product is given by
 $\f a \otimes \f b := \f a \f b^T = \left [ \f a_i  \f b_j\right ]_{i,j=1}^3$ for two vectors $\f a , \f b \in \R^3$ and by $ \f A \otimes \f a := \f A \f a ^T = \left [ \f A_{ij}  \f a_k\right ]_{i,j,k=1}^3 $ for a matrix $ \f A \in \R^{3\times 3} $ and a vector $ \f a \in \R^3$. 
The symmetric and skew-symmetric parts of a matrix are given by 
$\f A_{\sym}: = \frac{1}{2} (\f A + \f A^T)$ and 
$\f A _{\skw} : = \frac{1}{2}( \f A - \f A^T)$, respectively ($\f A \in \R^{3\times  3}$).
For the product of two matrices $\f A, \f B \in \R^{3\times 3 }$, we observe
 \begin{align*}
 \f A: \f B = \f A : \f B_{\sym}\,, \quad \text{if } \f A^T= \f A\quad \text{and}\quad
  \f A: \f B = \f A : \f B_{\skw}\,, \quad \text{if } \f A^T= -\f A\, .
 \end{align*}
Furthermore, it holds $\f A^T\f B : \f C = \f B : \f A \f C$ for
$\f A, \f B, \f C \in \R^{3\times 3}$ and
$ \f a\otimes \f b : \f A = \f a \cdot \f A \f b$ for
$\f a, \f b \in \R^3$, $\f A \in \R^{3\times 3 }$  and hence $ \f a \otimes \f a : \f A = \f a \cdot \f A \f a =  \f a \cdot \f A_{\sym} \f a$.
%%%%%%%%%%%%%%%%%%%%%%%%%%%%%%%%
%%%%%%%%%%%%%%%%%%%%%%%%%%%%%%%%%
%%%%%%%%%%%%%%%%%%%%%%%%%%%%%%%%%

%%%%%%%%% Nabla operator
We use  the Nabla symbol $\nabla $  for real-valued functions $f : \R^3 \to \R$, vector-valued functions $ \f f : \R^3 \to \R^3$ as well as matrix-valued functions $\f A : \R^3 \to \R^{3\times 3}$ denoting
\begin{align*}
\nabla f := \left [ \pat{f}{\f x_i} \right ] _{i=1}^3\, ,\quad
\nabla \f f  := \left [ \pat{\f f _i}{ \f x_j} \right ] _{i,j=1}^3 \, ,\quad
\nabla \f A  := \left [ \pat{\f A _{ij}}{ \f x_k} \right ] _{i,j,k=1}^3\, .
\end{align*}
 The divergence of a vector-valued and a matrix-valued function is defined by
\begin{align*}
\di \f f := \sum_{i=1}^3 \pat{\f f _i}{\f x_i} = \tr ( \nabla \f f)\, , \quad  \di \f A := \left [\sum_{j=1}^3 \pat{\f A_{ij}}{\f x_j}\right] _{i=1}^3\, .
\end{align*}

Throughout this paper, let $\Omega \subset \R^3$ be a bounded domain of class $\C^{3,1}$.
We rely on the usual notation for spaces of continuous functions, Lebesgue and Sobolev spaces. Spaces of vector-valued functions are  emphasized by bold letters, for example
$
\f L^p(\Omega) := L^p(\Omega; \R^3)$,
$\f W^{k,p}(\Omega) := W^{k,p}(\Omega; \R^3)$.
The standard inner product in $L^2 ( \Omega; \R^3)$ is just denoted by
$ (\cdot \, , \cdot )$, in $L^2 ( \Omega ; \R^{3\times 3 })$
by $(\cdot ; \cdot )$, and in $L^2 ( \Omega ; \R^{3\times 3\times 3 })$ by   $(\cdot \dreidotkom \cdot )$.

The space of smooth solenoidal functions with compact support is denoted by $\mathcal{C}_{c,\sigma}^\infty(\Omega;\R^3)$. By $\f L^p_{\sigma}( \Omega) $, $\V(\Omega)$,  and $ \f W^{1,p}_{0,\sigma}( \Omega)$, we denote the closure of $\mathcal{C}_{c,\sigma}^\infty(\Omega;\R^3)$ with respect to the norm of $\f L^p(\Omega) $, $ \f H^1( \Omega) $, and $ \f W^{1,p}(\Omega)$, respectively.
We denote the Dirichlet-trace by $\f \gamma_0$.

The dual space of a Banach space $V$ is always denoted by $ V^*$ and equipped with the standard norm; the duality pairing is denoted by $\langle\cdot, \cdot \rangle$. The duality pairing between $\f L^p(\Omega)$ and $\f L^q(\Omega)$ (with $1/p+1/q=1$), however, is denoted by $(\cdot , \cdot )$, $( \cdot ; \cdot )$, or $( \cdot \dreidotkom \cdot )$. The dual of $\f H^1_0$ is denoted by $\f H^{-1}$.

The unit ball in d dimensions is denoted by $B_d:= \{ \f x \in \R^d ; | \f x | \leq 1\}$ and the sphere in $d$ dimensions by  $\Se^{d-1}:= \{ \f x \in \R^d ; | \f d |=1  \}$.
We also use the sphere with radius $\nicefrac{1}{2}$, $\Se^{d-1}_{\nicefrac{1}{2}}$.

For $Q\subset \R^d$, the Radon measures are denoted by $\mathcal{M}(Q)$, the positive Radon measures by $\mathcal{M}^+(Q)$, and probability measures by $\mathcal{P}(Q)$. We recall that the Radon measures equipped with the total variation are a Banach space and  for compact sets Q, it can be characterized by~$\mathcal{M}(Q) = ( \C(Q))^*$ (see~\cite[Theorem~4.10.1]{edwards}). $\C_b(Q)$ are all bounded continuous functions on the set $Q$.
The integration of a function $f\in \C(Q)$ with respect to a measure $\mu\in \mathcal{M}(Q)$ is denoted by $ \int_Qf(\f h ) \mu(\de \f h)\,.$ In case of the Lebesgue measure we just write 
$ \int_Qf(\f h ) \de \f h\,.$

The cross product of two vectors is denoted by $\times $. We introduce the notation $ \rot{\cdot}$, which is defined via
\begin{align}
\rot{\cdot } : \R^d \ra \R^{d\times d}\, , \quad \rot{ \f h} := \begin{pmatrix}
0& - \f h_3 &\f h_2\\
\f h_3 & 0 & - \f h_1 \\
- \f h_2 & \f h_1 & 0
\end{pmatrix}\, .
\end{align}
The mapping $\rot{\cdot}$ has some nice properties, for instance
\begin{align*}
\rot{\f a}\f b = \f a \times \f b \, ,\quad \rot{\f a} ^T \rot{\f b} = (\f a \cdot \f b) I - \f b \otimes \f a\, 
\end{align*}
for all $\f a$, $\f b \in \R^3$, where $I$ denotes the identity matrix in $\R^{3\times 3}$, or 
\begin{align*}
 \quad \rot{\f a} : \nabla \f b = \rot{\f a} : \sk b = \f a \cdot \curl \f b \, , \quad \di \rot{ \f a} = - \curl \f a \, , \quad \frac{1}{2} \rot{\curl \f a} = \sk a \,
\end{align*}
for all $ \f a, \f b \in \C^1(\ov \Omega)$.

Additionally, we define $ \rott{\cdot } : \R^{3 \times 3} \ra \R^3 $, which is the left inverse of $\rot{\cdot} $ and given by
\begin{align*}
\rott{\f A} : = \begin{pmatrix}
\f A_{3,2} \\ \f A_{1,3} \\ \f A_{2,1}
\end{pmatrix}\,\quad \text{for all } \f A \in \R^{3\times 3}\,.%\text{ und damit } \rot{\rott{\cdot }} = I\,.
\end{align*} 
It holds $ \rott{\rot{\f a}} = \f a $ and hence $ 2 \rott{\sk{ a}}= \curl \f a$ for all $ \f a \in \C^1(\ov{ \Omega }; \R^3)$. 

We also use the Levi--Civita tensor $\f \Upsilon$. Let $\mathfrak{S}_3$ be the symmetric group of all permutations of $(1,2,3)$. The sign of  a given permutation  $\sigma \in \mathfrak{S}_3$ is denoted by $\sgn \sigma $.
The Tensor~$\f \Upsilon$ is defined via
\begin{align*}
\f \Upsilon_{ijk}:= \begin{cases}
\sgn{\sigma},  &  ( i ,j,k) = \sigma( 1,2,3)\text{ with } \sigma\in \mathfrak{S}_3 ,\\ 
%-1,  &  ( i ,j,k) = \sigma( 1,2,3)\text{ with } \sigma\in \mathfrak{S}_3 \text{ and } \sgn{\sigma}=-1,\\ 
0, & \text{ else}\, .
\end{cases}
\end{align*}
This tensor allows it two write the cross product as 
\begin{align*}
(\f a \times \f b)_i = \left (\f \Upsilon :(\f a \otimes \f b)\right )_i =\f \Upsilon _{ijk} \f a_j \f b _k\, \quad \text{for all }\f a , \f b \in \R^d \, 
\end{align*}
and the curl via
\begin{align*}
(\curl \f d)_i = \f\Upsilon_{ijk} \partial_j \f d_k \, \quad\text{for all }\f d \in \C^1 ( \Omega)\, .
\end{align*}

For a given Banach space $ V$, Bochner--Lebesgue spaces are denoted  by $ L^p(0,T; V)$. Moreover,  $W^{1,p}(0,T; V)$ denotes the Banach space of abstract functions in $ L^p(0,T; V)$ whose weak time derivative exists and is again in $ L^p(0,T; V)$ (see also
Diestel and Uhl~\cite[Section~II.2]{diestel} or
Roub\'i\v{c}ek~\cite[Section~1.5]{roubicek} for more details).
We often omit the time interval $(0,T)$ and the domain $\Omega$ and just write, e.g., $L^p(\f W^{k,p})$ for brevity.

Finally, by $c>0$, we denote a generic positive constant.

\section{Model and main results\label{sec:model}}
\subsection{Governing equations}
%%%
We consider the Ericksen--Leslie model as introduced in~\cite{unsere} with the constant $\gamma $  set to one. 
Additionally, the evolution equation of the director is restricted onto the unit sphere by taking the whole equation in the cross product with the director itself (compare~\cite{recent}). 
 The governing equations  read as
\begin{subequations}\label{eq:strong}
\begin{align}
\t {\f v}  + ( \f v \cdot \nabla ) \f v + \nabla p + \di \f T^E- \di  \f T^L&= \f g, \label{nav}\\
\f d \times \left (\t {\f d }+ ( \f v \cdot \nabla ) \f d -\sk{v}\f d + \lambda \sy{v} \f d + \f q\right ) & =0,\label{dir}\\
\di \f v & = 0,
\\
| \f d |&=1.
\end{align}%

We recall that $\f v : \ov{\Omega}\times [0,T] \ra \R^3$ denotes the velocity  of the fluid, $\f d:\ov{\Omega}\times[0,T]\ra \R^3$ represents the orientation of the rod-like molecules, and $p:\ov{\Omega}\times [0,T] \ra\R$ denotes the pressure.
The Helmholtz free energy potential~$F$, which is described rigorously in the next section, is assumed to depend on the director and its gradient, $F= F( \f d, \nabla \f d)$.
The free energy functional~$\mathcal{F}$  is defined by
\begin{align*}
\mathcal{F}: \f H^{\nicefrac{5}{4}} \ra \R , \quad \mathcal{F}(\f d)= \int_{\Omega} F( \f d, \nabla \f d) \de \f x \,,
\end{align*}
and $\f q$ is its variational derivative (see Furihata and Matsuo~\cite[Section 2.1]{furihata}),
\begin{align}\label{qdefq}
\f q =\frac{\delta \mathcal{F}}{\delta \f d}(\f d) =  \pat{F}{\f d}(\f d , \nabla\f d)-\di \pat{F}{\nabla \f d}(\f d, \nabla \f d)\, .
\end{align}
The Ericksen stress tensor $\f T^E$ is given by
\begin{equation}
\f T^E = \nabla \f d^T \pat{F}{\nabla \f d}( \f d , \nabla\f d ) \, ,\label{Erik}
\end{equation}
and the Leslie stress tensor by
\begin{align}
\begin{split}
\f T^L ={}&  \mu_1 (\f d \cdot \sy{v}\f d )\f d \otimes \f d +\mu_4 \sy{v}
 + {(\mu_5+\mu_6)} \left (  \f d \otimes\sy{v}\f d \right )_{\sym}
\\
& +{(\mu_2+\mu_3)} \left (\f d \otimes \f e  \right )_{\sym}
 +\lambda \left ( \f d \otimes \sy{v}\f d  \right )_{\skw} + \left (\f d \otimes \f e  \right )_{\skw}\, ,
\end{split}\label{Leslie}
\end{align}
where
\begin{align}
\f e : = \t {\f d} + ( \f v \cdot \nabla ) \f d - \sk v\f d\, .\label{e}
\end{align}
To ensure the dissipative character of the system, we assume that
\begin{align}
\begin{gathered}
\mu_1  > 0, \quad \mu_4 > 0, \quad (\mu_5+\mu_6)- \lambda (\mu_2+\mu_3)>0    \, ,
\\
4  \big( (\mu_5+\mu_6)- \lambda (\mu_2+\mu_3)\big)>
\big((\mu_2+\mu_3) -\lambda\big)^2\,.
\end{gathered}\label{con}
\end{align}
The case $\mu_1 =0$ simplifies the system and can thus be handled similar, but somehow simpler.
If Parodi's relation
\begin{equation}
\lambda = \mu_2+\mu_3\label{parodi}
\end{equation}
 is assumed to hold, the second line of~\eqref{con} is trivially fulfilled.
It can be derived from the Onsager reciprocal relation.  
This relation is only needed in this article to show that a certain energy inequality holds for the measure-valued solution. It is not needed for the existence of measure-valued solutions. 
Nevertheless, the announced weak-strong uniqueness result only holds for solutions fulfilling the energy inequality.

Finally, we impose boundary and initial conditions as follows:
\begin{align}
\f v(\f x, 0) &= \f v_0 (\f x) &\text{for } \f x \in \Omega ,
&&\qquad\f v (  \f x, t ) &= \f 0  &\text{for }( t,  \f x ) \in [0,T] \times \partial \Omega ,\notag \\
\f d (  \f x, 0 ) & = \f d_0 ( \f x) &\text{for } \f x \in \Omega ,
&&\qquad\f d (  \f x ,t ) & = \f d_1 ( \f x )  &\text{for }( t,  \f x ) \in [0,T] \times \partial \Omega .
\label{anfang}
\end{align}
\end{subequations}
We shall later assume that $\f d_1= \f d_0$ on $\partial \Omega$, which is a compatibility condition providing regularity.
%%%%%%%%%%%%%%%%%%%%%
%%%%%%%%%%%%%%%%%%%%A pair $(\f v , \f d )$ is
%%%%%%%%%%%%%%%%%%%%%
%%%%%%%%%%%%%%%%%%%%%
%%%%%%%%%%%%%%%%%%%%%%%

\subsection{The general Oseen--Frank energy}
The aim of this article is to provide a global solution concept for the Ericksen--Leslie model equipped with the Oseen--Frank energy, where the emphasis lies on the latter part. The Oseen--Frank energy was already considered by Leslie~\cite{leslie} and can be seen as the energy with the most physical relevance. Nevertheless, there is to the best of the author's knowledge no global mathematical solution concept available for this energy.

The \textit{Oseen--Frank} free energy potential is given by~(see~Leslie~\cite{leslie}) 
\begin{align*}
F(\f d , \nabla \f d) := \frac{K_1}{2} (\di \f d )^2 +\frac{K_2}{2}( \f d \cdot \curl \f d )^2  + \frac{K_3}{2} |\f d \times \curl \f d|^2 \,,
\end{align*}
where $K_1,K_2,K_3>0$.
This energy can be reformulated using the norm one restriction to
\begin{align}
\begin{split}
2 F( \f d , \nabla \f d)&:= k_1 ( \di \f d) ^2 +  k_2 | \curl \f d |^2 + k_3 | \f d |^2 ( \di \f d )^2 +   k_4 ( \f d\cdot \curl \f d )^2  +  k_5 | \f d \times \curl \f d |^2 \, ,
\end{split} \label{frei}
\end{align} 
where $ k_1=k_3=K_1/2$, $k_2={\min\{K_2,K_3\}}/{2}$, $k_4 = K_2-k_2$, and $ k_5 =K_3-k_2$ are again positive constants.
We remark that $| \f d |^2| \curl \f d |^2 = ( \f d \cdot \curl \f d )^2 + | \f d \times \curl \f d |^2 $. 

In Section~\ref{sec:add}, we use another reformulation. Setting $k:= \min\{ K_1/2,K_2/2,K_3/2\}$, $k_3=K_1-k$, $k_4=K_2-k$, as well as $k_5=K_3-k$, we get the formulation~\eqref{frei} with $k_1=k_2=k$. 
With some vector analysis one gets $ | \nabla \f d |^2 = ( \di \f d)^2 + | \curl \f d |^2 + \tr (\nabla \f d^2) - (\di \f d )^2 $, where the last two terms can be written in divergence form 
$$ \tr (\nabla \f d^2) - (\di \f d )^2 = \di ( \nabla \f d \f d - ( \di \f d )\f d ) $$
and hence this term is prescribed by the boundary values. 
This motivates to consider the Dirichlet energy 
$$ F_D(\nabla \f d ) = \frac{K}{2}| \nabla \f d|^2 \,,$$ 
which  is also called one-constant approximation. 
Most of the previous work concerning global solution concepts to the Ericksen--Leslie model consider this one constant approximation.

We introduce short notations for the derivatives of the free energy~\eqref{frei} with respect to $\nabla \f d$ and $ \f d$. The free energy~\eqref{frei} can be seen as a function $F: \R^d \times \R^{d\times d}$ where we replace $\f d$ in definition~\eqref{frei} by $\f h\in \R^d$ and $\nabla \f d $ by $ \f S\in \R^{d\times d }$.  Some vector calculus gives
\begin{align*}
2 F( \f h, \f S)  &=  k_1 \tr (\f S) ^2 + k_2 | ( \f S)_{\skw}|^2  + k_3 | \f h |^2 \tr ( \f S)^2  + k_4 ( \rot{ \f h } : ( \f S) _{\skw})^2  \\&\quad +4k_5| (\f S)_{\skw} \f  h|^2 \, 
\end{align*}
(see Section~\ref{sec:not} for the definition of the matrix $\rot{\cdot}$).

We abbreviate the derivative of $F$ with respect to $\f h$ by $F_{\f h}$ and the derivative with respect to $\f S$ by $F_{\f S}$ where
\begin{align*}
F_{\f S} : \R^d \times \R^{d\times d } \ra \R^{d\times d }  \quad \text{and } \quad F_{\f h} : \R^d \times \R^{d\times d } \ra \R^d \, ,
\end{align*}
these derivatives are given by
\begin{align}
\begin{split}
F_{\f S}(\f h ,\f S) & =   k _1 \tr (\f S) I + k_2 (\f S)_{\skw}  + k_3 \tr( \f S) | \f h|^2 I  +  k_4  \rot{ \f h}( \rot{\f h} :(\f S)_{\skw})\\& \quad  + 4 k_5 ((\f S)_{\skw}\f h \otimes \f h ) _{\skw} \\
F_{\f h} ( \f h , \f S) &= k_3 \tr(\f S)^2 \f h + 2 k_4 ( \rot{\f h} :(\f S)_{\skw}) \rott{( \f S)_{\skw}} + k_5 ( \f S)^T_{\skw}(\f S) _{\skw} \f h 
\, ,
\end{split}\label{FSFh}
\end{align}
(see Section~\ref{sec:not} for the definition of $\rott{\cdot}$).

To abbreviate, we define the tensor of order four $\f \Lambda \in \R^{d^4}$, and the tensor of order six $ \f \Theta \in \R^{d^6}$ via
\begin{align}
\f \Lambda_{ijkl} &: ={} k_1 \f \delta_{ij} \f \delta_{kl} + k_2 ( \f \delta_{ik}\f \delta_{jl}-\f\delta_{il}\f\delta_{jk})\,,\label{Lambda}
\intertext{and}
\f \Theta_{ijklmn} &:={}k_3 \f \delta_{ij}\f \delta_{lm}\f \delta_{kn} 
%+ k_4 \f \delta_{kn} ( \f \delta_{il}\f \delta_{jm} - \f \delta_{im}\f \delta_{jl})\notag \\& + ( k_5 -k_4) \left (  \f \delta_{il}\f \delta_{mn}\f \delta_{jk} - \f \delta_{mi}\f \delta_{ln}\f \delta_{jk} - \f \delta_{lj}\f \delta_{mn}\f \delta_{ik} + \f \delta_{jm}\f \delta_{ln}\f \delta_{ik}  \right )
%
% \\
%
%&
+ k_5  \left (  \f \delta_{il}\f \delta_{mn}\f \delta_{jk} - \f \delta_{mi}\f \delta_{ln}\f \delta_{jk} - \f \delta_{lj}\f \delta_{mn}\f \delta_{ik} + \f \delta_{jm}\f \delta_{ln}\f \delta_{ik}  \right )\notag
\\
& 
+k_4 \left (  \f \delta_{kn}\f \delta_{jm}\f \delta_{il} + \f \delta_{km}\f \delta_{jl}\f \delta_{in} + \f \delta_{kl}\f \delta_{jn}\f \delta_{im} - \f \delta_{kn}\f \delta_{jl}\f \delta_{im}- \f \delta_{km}\f \delta_{jn}\f \delta_{il} - \f \delta_{kl}\f \delta_{jm}\f \delta_{in}  \right )
\,, \notag
\end{align}
respectively.
Therewith, the free energy can be written as 
\begin{align}
2 F(\f d, \nabla \f d ) = \nabla \f d : \f \Lambda : \nabla \f d +  \nabla \f d \otimes  \f d  \dreidots \f \Theta \dreidots  \nabla \f d \otimes \f d \,. \label{tensoren}
\end{align}
The tensor $\f \Lambda$ is strongly elliptic, i.e.~there is an $\eta>0$ such that $ \f a \otimes  \f b : \f \Lambda : \f a \otimes \f b   \geq \eta | \f a|^2 | \f b|^2 $ for all $\f a, \f b \in \R^3$. Indeed, it holds 
\begin{align}\label{ellip}
 \f a \otimes  \f b : \f \Lambda : \f a \otimes \f b  = k_1 (\f a \cdot \f b)^2 + k_2 ( | \f a |^2 | \f b|^2-( \f a \cdot \f b )^2 ) \geq \min\{k_1,k_2\} | \f a |^2 | \f b|^2\,.
\end{align}

%%%%%%%%%%%%%%%%%%%%%%%%%%%%%%%%%%%%%
%% Regularisiertes System
%%%%%%%%%%%%%%%%%%%%%%%%%%%%%%%%%%%%%%%
\subsection{Regularized system}
Before, we show the existence of measure-valued solutions, we consider a regularized system and show the existence of weak solutions to this system.
A regularizing and a penalizing term  are added to the free energy potential and the system is adapted accordingly.
The regularized free energy potential is given by
\begin{align}
{F}_\delta(\f d , \nabla \f d , \nabla^2 \f d) : = \frac{\delta}2 | \Lap \f d |^2 + F(\f d , \nabla \f d )+ \frac{ 1}{4 \varepsilon(\delta) }( | \f d|^2 -1 )^2 \, ,\label{regfrei}
\end{align} 
where $\delta>0$ and $F$ is given by~\eqref{frei}. We define $\varepsilon(\delta)=\delta$. This is just a linear connection of the regularization parameter $\delta$ and the penalization parameter $\varepsilon$. Later on, we are going to choose another connection to be able to prove better estimates (see Section~\ref{sec:add}).
Like beforehand, if $\f d$, $\nabla \f d $, and $\nabla^2 \f d $ are replaced by $\f h$, $\f S$, and $\f \Gamma$, respectively, the regularized free energy potential can be written as
\begin{align*}
{F}_{\delta}(\f h ,  \f S , \f \Gamma ) = \frac{\delta}{2}  \left |\f   \Gamma :I \right |^2 + F( \f h, \f S) + \frac{ 1}{4 \varepsilon }( | \f h|^2 -1 )^2  \,.
\end{align*}
Thus, the free energy is given by $\F_{\delta}(\f d):=\int_{\Omega} F_{\delta}(\f d , \nabla \f d , \nabla ^2 \f d )\de \f x$ and the variational derivative of this free energy by
\begin{align}
{\f q}_{\delta} = \frac{\delta {\F_{\delta}}}{\delta \f d} = \pat{F_{\delta}}{\f h} - \di \pat{{F}_{\delta}}{ \f S } + \nabla ^2 : \pat{ {F}_{\delta}}{\f \Gamma }  =  \f q + \delta \Lapp \f d + \frac{1}{\varepsilon}( | \f d|^2-1)\f d  \, .\label{qtilde}
\end{align}

Additionally, we have to adapt the  Ericksen stress $\f T^E$  for the regularized system,
\begin{align}
\f T^E_{\delta} : = \f T^E + \delta  \Lap \f d \cdot \nabla^2 \f d  - \delta \nabla \f d ^T  \nabla \Lap \f d \, .\label{Erikreg}
\end{align}
\begin{rem}
This adaptation is necessary in order to show the energy equality~\eqref{entro1} for the discretized system, which is essential for all a priori estimates.
\end{rem}
First, we recall the important relation between the Ericksen stress and the gradient of the director multiplied with the variational derivative (see~\cite{unsere})
\begin{align*}
( \f T^E ; \nabla \f w ) - ( \nabla \f d^T {\f q}, \f w) = 0\quad \text{ for all } \f w \in \V \, .
\end{align*} 
A similar identity  holds for the regularized system. Let again be $\f w \in \V$, then we have
\begin{align}
\begin{split}
({\f T}_{\delta}^E ; \nabla \f w ) - ( \nabla \f d^T {\f q_{\delta}} , \f w )&=( {\f T}^E ; \nabla \f w ) - ( \nabla \f d^T \f q , \f w ) \\
&\quad  + \delta  (  \Lap \f d \cdot\nabla^2 \f d ; \nabla \f w) - \delta (\nabla \f d ^T  \nabla \Lap \f d  ; \nabla \f w ) 
\\
&\quad 
 - \delta ( \nabla \f d ^T\Lapp \f d , \f w )- \frac{1}{\varepsilon }\left ( \nabla \f d ^T \f d ( | \f d|^2-1) , \f w  \right )\\
&=   \delta  (  \Lap \f d \cdot\nabla^2 \f d ; \nabla \f w) - \delta (\nabla \f d ^T  \nabla \Lap \f d  ; \nabla \f w ) \\
& \quad 
+ \delta ( \nabla \f d^T  \nabla \Lap \f d ; \nabla \f w ) + \delta ( \nabla (\nabla \f d)^T :  \nabla \Lap \f d , \f w) 
\\
&\quad - \frac{1}{2\varepsilon } ( \nabla |\f d |^2 ( | \f d|^2-1) , \f w )\\
&= \delta  (  \Lap \f d  \nabla^2 \f d ; \nabla \f w)  - \delta  (  \Lap \f d \cdot\nabla^2 \f d ; \nabla \f w) 
\\
&\quad - \delta ( \nabla \Lap \f d  \cdot \Lap \f d , \f w) - \frac{1}{4\varepsilon } (  \nabla ( | \f d|^2-1)^2 , \f w )
 \\
& =- \intet{( \f w \cdot \nabla ) \left (\frac{\delta}{2}| \Lap \f d|^2+\frac{1}{4\varepsilon }   ( | \f d|^2-1)^2 \right ) } =0 \, .
\end{split}\label{Erikseniden}
\end{align}
We remark, that we have to equip the regularized system with another boundary condition, since the regularizing term is of higher order.
We regularize with the square of the operator $\Lap$ and thus get the additional boundary condition $\Lap \f d = 0 $ on $\partial \Omega$.

%%%%%%%%%%%%%%%%%%%%%%%%%%%%%%%%%%%%%%%%%%
%% Definition of a weak solution to the regularized system
%%%%%%%%%%%%%%%%%%%%%%%%%%%%%%%%%%%%%%%%%%%

\begin{definition}[Weak solution to the regularized system]\label{defi:weak}
A pair $(\fd v , \fd d )$ is said to be a solution to the  regularized Ericksen--Leslie system if
\begin{align}
\begin{split}
\fd v &\in L^\infty(0,T;\Ha)\cap  L^2(0,T;\V)
\cap W^{1,2}(0,T; (\f H^2 \cap \V)^*),
\\ \fd d& \in L^\infty(0,T;\Hc)\cap  L^2(0,T;\Hg) \cap W^{1,2}  (0,T;   \f L^{\nicefrac{3}{2}} ),
\end{split}\label{weakreg}
\end{align}
and if
\begin{subequations}\label{weak}
\begin{align}
\begin{split}
-\int_0^T (\fd v(t), \f \varphi'(t)) \de t &+ \int_0^T ((\fd v(t)\cdot \nabla) \fd v(t), \f \varphi(t)) \de t  - \intte{\left (\nabla
\fd d(t)^T  {\f T}^E_{\delta}(t)  ;\nabla   \f \varphi(t)
\right )}\\&+ \intte{(\f T^L_{\delta}(t): \nabla \f \varphi(t) ) } =\intte{ \left \langle \f g (t),\f \varphi(t)\right \rangle } ,\quad
\end{split}\label{eq:velo}\\\begin{split}
-\intte{( \fd d(t), \f \psi'(t) ) } &+ \intte{((\fd v(t)\cdot \nabla ) \fd d(t), \f \psi(t))} - \intte{\left( (\nabla \fd v(t))_{\skw} \fd d(t) , \f \psi(t)\right )}\\&
+\lambda\intte{\left( (\nabla \fd v(t))_{\sym} \fd d(t) , \f \psi(t)\right)}
+ \intte{\left(\f {q}_{\delta}(t) , \f \psi(t)\right)}=0 \quad
\end{split}
\label{eq:dir}
\end{align}%
\end{subequations}
for all solenoidal  $ \f \varphi \in \mathcal{C}_c^\infty( \Omega \times ( 0,T);\R^3))$ and $ \f \psi \in\mathcal{C}_c^\infty( \Omega \times ( 0,T);\R^3))$.
Additionally, the 
initial conditions  shall be fulfilled, i.e.~$(\fd v (0) , \fd d(0)) \rightharpoonup ( \f v_0, \f d_{0}) $ in $  \Ha \times \f H^{2}$  and the boundary values shall be fulfilled in the sense of the trace operator.

\end{definition}

\begin{theorem}[Existence of solutions to the regularized system]\label{thm:weak}
Let $\Omega$ be a bounded domain of class $\C^{3,1}$.
For given initial data $ \f v_0 \in \Ha $ and $ \f d_{0}\in  \Hc $ with $ | \f d_0|=1$ for a.e.~$\f x \in \Omega$,  boundary data $\f d_1\in \f H^{\nicefrac{7}{2}}( \partial \Omega)$ fulfilling the compatibility condition $ \f \gamma_0 (\f d_0) = \f d_1$, and right-hand side $ \f g\in L^2 ( 0,T; (\V)^*)$, there exists a global-in-time  solution to the Ericksen--Leslie system~\eqref{eq:strong} equipped with the regularized free energy~\eqref{regfrei}
in the sense of Definition~\ref{defi:weak}.
The solution additionally fulfills the intrinsic boundary condition $\f \gamma_0(\Lap \fd d ) =0$. 
\end{theorem}

\subsection{Measure-valued solutions}
\begin{defi}[Measure-valued solutions]\label{def:meas}
The tupel $( ( \f v ,\f d ) , ( \nu^o,m , \nu^\infty) , ( \mu , \nu^\mu))$ consisting of the pair $(\f v , \f d)$ of velocity field $\f v$ and director field  $\f d$, the generalized gradient Young measure $(\mu, \nu^\mu)$, and the defect measure $(\mu , \nu^\mu)$ (see Section~\ref{sec:young})  is  said to be a measure-valued solution to~\eqref{eq:strong} if
%%%%%%%%%%%%%%%%%%%%%%%%%%%%%%%%%%%%%%%%%
%%%%%%%%%%%%%%%%%%%%%%%%%%%%%%%%%%%%%%%%%%%%
\begin{align}
\begin{split}
\f v &\in L^\infty(0,T;\Ha)\cap  L^2(0,T;\V)
\cap W^{1,2}(0,T; ( \f W^{1,3}_{0,\sigma}(\Omega))^*),
\\ \f d& \in L^\infty(0,T;\He)\cap   W^{1,2}  (0,T;   \f  L^{\nicefrac{3}{2}} ),\\
\{\nu^o _{(\f x,t)}\}&  \subset \mathcal{P} ( \R^{d\times d})\, \text{ a.e.~in $\Omega\times [0,T]$} \, ,\\
 \{m_t\} &\subset \mathcal{M}^+(\ov\Omega)\,\text{ a.e.~in $ [0,T]$} \, ,  \\
 \{\nu^\infty _{(\f x,t)}\} &\subset \mathcal{P}(\ov B_d\times \Se^{d^2-1})\,\text{ $m_t $-a.e.~in }\Omega \text{ and a.e.~in }   [0,T]\, ,\\
 \{\mu_t\} &\subset \mathcal{M}^+(\ov\Omega)\,\text{ a.e.~in $ [0,T]$} \, ,  \\
\{ \nu^{\mu}_{(\f x ,t)}\}&\subset \mathcal{P}(\Se^{d^3-1})\, \text{ $\mu_t$-a.e.~in }\Omega \text{ and a.e.~in }   [0,T]\, 
\end{split}\label{measreg}
\end{align}
and if
\begin{subequations}\label{meas}
\begin{align}
\begin{split}
&\int_0^T (\t \f v(t), \f \varphi(t)) \de t + \int_0^T ((\f v(t)\cdot \nabla) \f v(t), \f \varphi(t)) \de t  - \intte{ \ll{\nu_t,\f S^T F_{\f S}(\f h, \f S):\nabla \f \varphi(t)   }  }
%\\ 
%& 
%+ \intte{ ( \langle \f S^T F_{\f S}^{OF}(\f h, \f S) , \nu^\infty \rangle ; \nabla \f \varphi ) _m  } 
\\&-2 \int_0^T \ll{\mu_t, \f \Gamma \dreidots ( \f \Gamma\cdot \nabla\f \varphi(t))}+ \intte{(\f T^L(t): \nabla \f \varphi(t) ) } =\intte{ \left \langle \f g (t),\f \varphi(t)\right \rangle }\, \quad
\end{split}\notag 
\intertext{as well as}
\begin{split}
&\intte{\left ( \f d (t) \times\left ( \t \f d(t)+( \f v(t)\cdot \nabla ) \f d(t) -  \sk{v(t)}\f d(t)\right ) , \f \psi(t)\right )}\\&
+\lambda\intte{\left(\f d (t) \times  \sy{v(t)} \f d(t) , \f \psi(t)\right)}
+ \intte{(\rot{\f d(t)}  F_{ \f S} ( \f d(t) , \nabla \f d (t)) ; \nabla \f \psi(t) )} \\ &  +\intte{ \ll{\nu_t, \left (\f \Upsilon :\left (\f  S  (F_{\f S}(\f h, \f S))^T\right )\right ) \cdot\f \psi(t)   }}+\intte{ \ll{\nu_t, \left (\f h \times F_{\f h}(\f h, \f S)\right ) \cdot\f \psi(t)   }}=0\, \quad
\end{split}
\label{eq:mdir}
\end{align}%
\end{subequations}
hold for all $ \f \varphi \in \mathcal{C}_c^\infty(\Omega \times ( 0,T);\R^3))$ with $ \di \f \varphi =0$ and $ \f \psi \in  \mathcal{C}_c^\infty(\Omega \times ( 0,T);\R^3))$, respectively.
Additionally, the norm restriction of the director holds, i.\,e. $|\f d (\f x,t)|=1$ for a.e.~$(\f x, t)\in \Omega\times (0,T)$, the oscillation measure of the identity is the gradient of the director
\begin{align*}
\int_{\R^{3\times 3}} \f S  \nu_{(\f x,t)}^o(\de \f S)   = \nabla \f d(\f x,t)\, ,
\end{align*}
for a.\,e. $( \f x,t) \in \Omega\times (0,T)$ 
and the
initial conditions $( \f v_0, \f d_0)\in \Ha \times \Hc$ with $ \f d_0 \in \Hrand{7}$ shall be fulfilled in the weak sense.
The dual pairings are defined as
\begin{align*}
 \ll{\mu_t,f} :={}& \int_{\ov \Omega} \int_{\Se ^{d^3-1}}  \sum_{i,j=1}^3f(\f \Gamma) \nu^\mu_{(\f x ,t)} ( \de \f \Gamma) \mu_t(\de \f x )\,\intertext{ for $f \in \C(\Se^{3^3-1};\R) $  and}
\ll{\nu_t, f } :={}& 
%\int_{\Omega} \langle \nu^o_{(\f x, t)},  f(\f x, t , \f d(\f x,t), \f S )  \rangle\de \f x + \int_{\ov\Omega}\langle \nu_{(\f x, t)}^\infty, \tilde{f}( \f x,t , \f h , \f S)  \rangle m_t (\de \f x)\\
%={}&
 \int_{\Omega} \int_{\R^{d\times d} } f(\f x, \f d(\f x, t), \f S)  \nu^o_{(\f x, t)} ( \de \f S)\de \f x \\&+ \int_{\ov\Omega}\int_{\Se^{d^2-1} \times \ov B_d} \tilde{f} (\f x,  \tilde{\f h} , \tilde{\f S}) \nu_{(\f x, t)}^\infty ( \de \tilde{\f S}, \de \tilde{\f h}) m_t (\de \f x)
\end{align*}
for $f\in \mathcal{R}$ (see~\eqref{R} below).
\end{defi}
We refer to the section~\ref{sec:not} for the definition of the tensor $\f \Upsilon$ and to \eqref{transi} for the definition of the transformed function $\tilde{f}$.
\begin{rem}
We often abuse  the notation by writing $ \ll{\nu_t , f(\f h , \f S)}$. Thereby, we mean the generalized Young measure applied to the continuous function $(\f h, \f S)\mapsto f (\f h ,\f S)$. 
\end{rem}

\begin{theorem}[Existence of measure-valued solutions]\label{thm:meas}
Let $\Omega$ be a bounded domain of class $\C^{3,1}$.
For given initial data $ \f v_0 \in \Ha $ and $ \f d_0\in  \Hc $
with $ | \f d_0|=1$ for a.e.~$\f x \in \Omega$, boundary data $\f d_1\in \f H^{\nicefrac{7}{2}}( \partial \Omega)$ fulfilling the compatibility condition $ \f \gamma_0 (\f d_0) = \f d_1$, and right-hand side $ \f g\in L^2 ( 0,T; (\V)^*)$, there exists a measure-valued  solution to the Ericksen--Leslie system~\eqref{eq:strong}  with
the Oseen--Frank free energy~\eqref{frei}
in the sense of Definition~\ref{def:meas}.
\end{theorem}
\begin{rem}
This is a global but very weak solution concept. Nonlinear occurring gradients of the director, i.\,e.~the Ericksen-stress and parts of the variational derivative, are represented by the associated generalized gradient Young measure. Additionally, a defect measure appears due to the regularization in the Ericksen-stress. In an upcoming article, we are going to show that this measure-valued solutions fulfill the weak-strong uniqueness property. As long as a local strong solution exists to this model, it coincides with the measure-valued solution. 
Local strong solutions are known to exist for similar models, see for instance~\cite{localin3d},~\cite{recent}, or~\cite{Pruess2}. 
\end{rem}

\begin{rem}
When we choose $\varepsilon = \delta^{\nicefrac{7}{3}}$, it can be shown that the support of the defect angle measure $\nu^\infty$ is $\Se^{d^2-1} \times \Se^{d-1}_{\nicefrac{1}{2}} $ instead of $\Se^{d^2-1} \times B_d$ (see Proposition~\ref{prop:apri2an}). 

\end{rem}
\begin{rem}
We postulate that the defect measure $\mu$ vanishes almost everywhere in $\Omega\times [0,T]$.
In the future, we additionally want to investigate
 whether the oscillation measure $\nu^o$ coincides with the point measure~$\delta_{\nabla \f d}$. However, such analysis relies on local energy methods (see for instance~\cite{Dim3}) which are very different to the global techniques used in this paper.

\end{rem}

%%%%%%%%%%%%%%%%%%%%%%%%%%%%%%%%%%%%%%
%%%% Young Masse
%%%%%%%%%%%%%%%%%%%%%%%%%%%%%%%%%%%%%%%%%
\section{Generalized gradient Young measures\label{sec:young}}
This section introduces the concept of generalized gradient Young measures and the sense of convergence that is used to prove Theorem~\ref{thm:meas}.

\subsection{Definitions and main theorem for generalized gradient Young measures}
Consider a sequence of functions $\{ \f d_{\delta}\} \subset L^\infty(0,T;\Hc)$ with $\|\nabla \f d_\delta |\f d_\delta|\|_{L^\infty(0,T;\Le)}\leq c $.  We want to study the limit of sequences of the form 
\begin{equation}
f( \cdot , \f d_{\delta}( \cdot ) , \nabla \fd d(\cdot)): Q \ra \R \,\label{limiti}
\end{equation} 
for continuous functions $f$ with appropriate growth conditions. 

We abbreviate $Q : = \Omega \times (0,T)$ and for a given function $f \in \C ( \ov Q \times \R^d \times \R^{d\times d }  ) $, we define its transform $ \tilde{f}\in \C ( \ov Q \times {B}_d \times {B}_{d\times d } ) $ by 
\begin{align}
\tilde{f} ( {\f y }, \tilde{\f h} ,\tilde{\f  S} ) : = 
f \left ( \f y , \frac{\tilde{\f h}}{\sqrt{1-|\tilde{\f h}|^2}}, \frac{\tilde{\f S}}{\sqrt{1-|\tilde{\f S}|^2}}\right )  ( 1-|\tilde{ \f h}|^2  )( 1- | \tilde{\f S}|^2) \, 
 .\label{transi}
\end{align}
The set of functions for which we are going to identify the limit of~\eqref{limiti} are the functions $f \in \C(Q\times \R^d \times \R^{d\times d })$ whose transform (see~\eqref{transi}) admits a continuous extension onto the closure of the domain. We thus define the following set of functions
\begin{align}
\mathcal{R}:= \left \{ f \in \C ( \ov Q \times \R^d \times \R^{d\times d }  ) | \exists \tilde{ g}\in \C(\ov Q \times \ov{B_d}\times \ov{B_{d\times d}})\, ; \, \tilde{f}= \tilde{g}\text{ on } \ov Q \times B_d \times B_{d\times d}  	\right \}\,.\label{R}
\end{align} 
The initial idea for the representation of limits of sequences like~\eqref{limiti} for functions $f \in \mathcal{R}$ is due to DiPerna and Majda~\cite{DiPernaMajda} and relies heavily on the fact that $\mathcal{R}$ is isometrically isomorphic to $\C( \ov Q\times \ov{B}_d \times \ov B_{d\times d})$ when $\mathcal{R}$ is equipped with an appropriate norm. Thus, it is possible to represent  the limit of~\eqref{limiti} by a measure $\tilde{\nu } \in \M( \ov Q\times  \ov B_d \times \ov B_{d\times d})= \C (\ov Q\times  \ov B_d \times \ov B_{d\times d})^*$.

A \textbf{generalized gradient Young measure} on $\ov \Omega \times [0,T]$ with values in $\R^d \times \R^{d\times d}$ is a triple $( \nu_y^o , m_t , \nu_y^\infty) $ consisting of  
\begin{itemize}
\item a parametrized family of probability measures $\{\nu_y^o\}_{y\in Q} \in \mathcal{P}(\R^{d\times d})$ for a.e.~$y \in Q$,
\item a positive measure $m_t\in \M^+(\ov \Omega )$ for a.e.~$t\in (0,T)$ and
\item a parametrized family of probability measures $\{ \nu_y^\infty\}_{y\in \ov Q} \in \mathcal{P}( \ov B_d \times \Se^{d^2-1})$ for $m_t$-a.e.~$\f x  \in \ov \Omega $ and a.e.~$t\in(0,T)$. 
\end{itemize}
As in~\cite[page 552]{rindler}, we call $\nu^o$ \textit{oscillation measure}, $m_t$ \textit{concentration measure} and $\nu^\infty$ \textit{concentration angle measure}. 

A \textbf{defect measure} on $\ov \Omega \times (0,T)$ with values in $ \R^{d\times d\times d }$
is a pair $(\mu_t, \nu^\mu )$ consisting of 
\begin{itemize}
\item a positive measure $\mu_t\in\M^+(\ov\Omega)$ for a.e.~$t\in (0,T)$ and
\item a parametrized family of probability measures $\{ \nu_y^\mu\}_{y\in \ov Q} \in \mathcal{P}(  \Se^{d^3-1})$ for $\mu_t$-a.e.~$\f x \in \ov \Omega$ and a.e.~$t\in(0,T)$. 
\end{itemize}
We are now able to state the important theorem for generalized gradient Young measures.

\begin{theorem}\label{thm:young}
Let $\{\f d_{\delta}\}_{\delta\in(0,1)}$ be a family of functions bounded in $L^\infty(0,T;\He)$ with $$ \sup_{\delta\in(0,1)} \left \| \nabla \fd d | \fd d |\right \|_{L^\infty(\Le)}< \infty$$
and $\{ \fd d \}$ is relatively compact in $L^2(0,T\Le)$. 
 Then there exists a subsequence $\{\delta_n\}$ and a generalized gradient Young measure such that for all $f\in\mathcal{R}$, we have
\begin{align*}
\int_{Q} f( \f y , \f d_{\delta_n} ( \f y ) , \nabla \f d_{\delta_n} ( \f y) ) \de \f y \ra \ll{f , \nu_t }\, 
\end{align*}
for $\delta_n \ra 0$,  where 
the dual paring $\ll{\cdot,\cdot}$ is defined for a function $f \in \mathcal{R}$ by
\begin{align*}
\ll{f,\nu_t} &:= \int_{\Omega} \langle f(\f x, \f d(\f x,t), \cdot) , \nu^o_{(\f x, t)} \rangle\de \f x + \int_{\ov\Omega}\langle \tilde{f}( \f x, \cdot, \cdot) , \nu_{(\f x, t)}^\infty \rangle m_t (\de \f x)\\
&= \int_{\Omega} \int_{\R^{d\times d} } f(\f x, \f d(\f x, t), \f S)  \nu_{(\f x, t)}^o ( \de \f S)\de \f x + \int_{\ov\Omega}\int_{\Se^{d^2-1} \times \ov B_d} \tilde{f} (\f x,  \tilde{\f h} , \tilde{\f S}) \nu_{(\f x, t)}^\infty ( \de \tilde{\f S}, \de \tilde{\f h}) m_t (\de \f x)\,.
\end{align*}
Additionally, $\nu^o $ is a classical gradient Young measure, i.e. 
\begin{align}
\langle \nu^o _{\f y}, I   \rangle= \int_{\R^{d\times d } }\f S \nu_{\f y} ( \de \f S) = \nabla \f d(\f y)\, \label{gradmeas}
\end{align}
for a.\,e.~$\f y\in Q$.
The function $\tilde{f}$ is the recession function similar to~\eqref{transi} defined by
\begin{align*}
\tilde{f}(\f y ,\tilde{\f h} , \tilde{\f S}) : = \lim_{\bar{\f y}\ra \f y  }
\lim_{\genfrac{}{}{0pt}{}{\bar{\f S}\ra \tilde{\f S}, | \bar{\f S}|<1}{\bar{\f h}\ra \tilde{\f h}, | \bar{\f h}|<1}}  f\left ( \bar{\f y} , \frac{\tilde{\f h}}{\sqrt{1-| \tilde{\f h}|^2}}, \frac{\tilde{\f S}}{\sqrt{1-| \tilde{\f S}|^2}} \right ) ( 1- | \tilde{\f h}|^2 ) ( 1- | \tilde{\f S}|^2 )
% = 
%\begin{cases}
%\lim_{s \ra \infty} \frac{f\left (\bar{\f y}, \frac{\tilde{\f h}}{\sqrt{1-| \tilde{\f h}|^2}}, s \tilde{\f S}\right )}{ s^2 }\\
%\lim_{s \ra \infty} \frac{f(\f x, s\tilde{\f h}, s \tilde{\f S})}{ s^4 }\\
%\end{cases}
\,,
\end{align*}
with $(\f y, \ft h, \ft S)\in \ov Q \times \ov B_d \times \ov B_{d\times d}$. 
\end{theorem}
The \textit{proof} of Theorem~\eqref{thm:young} is split in two propositions, Proposition~\ref{lem:meas} and Propositions~\ref{lem:timemeas}.

%%%%%%%%%%%%%%%%%%%%%%%%%%%%%%%%%%%%%%%%%%%
%%%%%%%%%%%%%%%%%%%%%%%%%%%%%%%%%%%

%%%%%%%%%%%%%%%%%%%%%%%%%%%%%%%%%
%% Masswertige Loesungen
%%%%%%%%%%%%%%%%%%%%%%%%%%%%%%%%%%%
%%%%%%%%%%%
%%
%%
%%
%%

\begin{proposition}\label{lem:meas}
Let $\{\fd d\}$ be a sequence with $$\sup_{\delta\in (0,1)} \left (\|\nabla \fd d\|_{\Le(Q)} + \left \| \nabla \fd d|\fd d|\right \|_{\Le(Q)} \right ) <\infty$$ and $\{\fd d\}$ is relatively compact in $L^2( 0,T; \f L^2)$.  Additionally, we assume that $f \in \mathcal{R}$. Then there exists a measure $m\in\M(\ov Q)$, two families of measures $\{\nu^o_{\f y}\}_{\f y\in Q} $ and $ \{\nu^\infty_{\f y} \}_{\f y\in Q} $ such that $ \nu^o_{ \f y} \in \mathcal{P}(\R^{d\times d})$ and $\nu^{\infty}_{\f y}\in \mathcal{P}(B_d \times \Se^{d^2-1})$ and 
\begin{align}
f( \f y, \fd d , \nabla \fd d ) \stackrel{*}{\rightharpoonup} \int_{\R^{d\times d} } f (\f d( \f y),\f S){\nu}^o_{\f y}( \de \f S) 
 + \int_{\Se^{ d^2 -1} \times \overline{B}_d} \tilde{f}(\ft h, \ft S) \nu^\infty_{\f y} (\de \ft h,  \de \ft S)m
 % +
% \int_{\Se^{d-1}}\int_{B_{d\times d} } \tilde{f}(\ft h, \ft S) \ov{\nu}^\infty_{\f y} (\de \ft h, \de \ft S)\ov{m} 
\quad \text{in }\M(\ov \Omega) \, .\label{convinS}
\end{align}
The measure $\nu^o$ fulfils~\eqref{gradmeas} almost everywhere.
\end{proposition}

\begin{rem}\label{rem:trs}
The transformation~\eqref{transi} does not change functions with quadratic growth in $\f S$ times $\f h$. Indeed, let $ g:\Se^{d-1}\times \Se^{d^2-1} \ra \R $ be continuous and  $f: \R^d \times \R^{d\times d} \ra \R$ be defined via $ f( \f h , \f S ) : = g( \f h / | \f h | , \f S /|\f S|) | \f h |  ^2 | \f S|^2 $. Then we get
\begin{align*}
\tilde{f}(\tilde{\f h}, \tilde{\f  S}) 
&= f \left ( \frac{\ft h }{\sqrt{1-|\ft h|^2}}, \frac{\tilde{\f S}}{\sqrt{1-|\ft S|^2}}\right )( 1-|\tilde{ \f h}|^2  ) ( 1- | \tilde{\f S}|^2) \\
&= g \left (\frac{\tilde{\f h}}{|\tilde{\f h}|}, \frac{\tilde{\f S}}{|\tilde{\f S}|}\right ) \frac{|\tilde{\f h}|^2}{1-|\tilde{\f h|^2}}\frac{|\tilde{\f S}|^2}{1-|\tilde{\f S}|^2} ( 1-|\tilde{ \f h}|^2  ) ( 1- | \tilde{\f S}|^2)\\
&= g \left (\frac{\tilde{\f h}}{|\tilde{\f h}|}, \frac{\tilde{\f S}}{|\tilde{\f S}|}\right ) | \tilde{\f h}|^2  | \tilde{\f S}|^{2} = f \left (\tilde{\f h}, \tilde{\f S}\right ) \,.
\end{align*}
%This property is important to us, since we want to consider terms with this .
Most of the appearing terms in  Definition~\ref{def:meas} have the above  growth behaviour. This implies that the transformation of $\f h \times F_{\f h}(\f h ,\f S)$ remains the function itself. Only the linear terms in $F_{\f S}$ are changed by multiplying them with $1-|\tilde{\f h}|^2$,
 such that for example
\begin{align*}
\widetilde{\f S ^T F_{\f S}}(\tilde{\f h}, \tilde{\f S})= \tilde{\f S}^T F_{\f S}(\tilde{\f h}, \tilde{\f S})- k_1 | \tilde{ \f h }|^2 \tr{(\tilde{\f S})}\tilde{\f S}^T - k_2| \tilde{\f h}|^2 \tilde{\f S}^T ( \tilde{\f S})_{\skw}\,.
\end{align*}
\end{rem}
\begin{proof}[Proof of Proposition~\ref{lem:meas}]
We define the family of measures $\{L_\delta	\}_\delta \subset  \mathcal{M}( \ov{Q}\times \overline{B}_d \times \overline{B}_{d\times d })$ via 
\begin{align}
\langle L_\delta, g \rangle := \int_{Q} g\left ( \f y ,\frac{\fd d (\f y)}{\sqrt{1+|\fd d (\f y)|^2}}, \frac{\nabla \fd d (\f y)}{\sqrt{1+|\nabla \fd d(\f y)|^2}}  \right ) (1+  | \fd d|^2)(1+| \nabla \fd d|^2 ) \mu (\de \f y) \, ,
\end{align}
where  $g \in \C_b( \ov{Q}\times \overline{B}_d \times \overline{B}_{d\times d })$. 
Due to our a priori estimates for the approximate solutions, we see that for all $g \in \C_b(\ov{Q}\times \overline{B}_d \times \overline{B}_{d\times d })$ with $\|g\|_{ \C(\ov {Q}\times \overline{B}_d \times \overline{B}_{d\times d })}\leq 1$, we have
\begin{align*}
\sup_{\delta\in(0,1)}\langle L_\delta ,g \rangle <\infty\, . 
\end{align*}
 Via standard arguments, we first extract a sequence $\{\delta_k\}$ such that $\delta_k \ra 0$ and then a  weakly$^*$ converging subsequence $\{\delta_n \}\subset \{\delta_k\}$ with   
\begin{align*}
L_{\delta_n} \stackrel{*}{\rightharpoonup}  L  \quad\text{in  } \mathcal{M}(\ov {Q}\times \overline{B}_d \times \overline{B}_{d\times d })\, .
\end{align*}
In the following, the subsequences are not relabled any more.
The canonical projection of $L$ onto $\ov Q$ will be called $\tilde{m}$, i.e.~$
\tilde{m}(E) :=  L ( E \times \overline{B}_d \times \overline{B}_{d\times d})
$
for all Borel sets $E\subset\ov Q$.
The classical desintegration argument for measures (see Evans~\cite[Theorem~10.]{evans} or Fonseca~\cite[Proposition 3.2.]{Fonseca}) provides the existence of a probability measure $\tilde{\nu}_{\f y}\in \mathcal{P}(  \overline{B}_d \times \overline{B}_{d\times d} , \tilde{m} ) $ such that
\begin{align}
\langle L,g \rangle = \int_{\overline{Q}} \int_{ \overline{B}_d \times \overline{B}_{d\times d}  }g( \f y , \f h , \f S) \tilde{\nu}_{\f y} ( \de \f h, \de S) \tilde{m}( \de \f y)\label{convmeasures} \,. 
\end{align}

Since $\tilde{m}$ is a measure on $\ov{Q}$, we now consider its {Radon--Nikod\'{y}m--Lebesgue}-de\-com\-po\-si\-tion (see Evans and Gariepy~\cite[section 1.6.2]{evans2} or Halmos~\cite[Section 32, Theorem C]{halmos})  with respect to the Lebesgue measure. There exists a function $p\in L^1(Q)$ and a measure $m_s \in \mathcal{M} ( \ov Q) $ such that
\begin{align*}
\tilde{m}( \de \f y ) = p( \f y)  \de \f y + m_s ( \de \f y) \, .
\end{align*}
The measure $m_s$ and the Lebesgue measure are then mutually singular. Remark that $\de $ without specifying the measure always means integration with respect to the Lebesgue measure.

Applying the desintegration theorem a second time (see Evans and Gariepy~\cite[section 1.6.2]{evans2}), we get
$$ \tilde{\nu}_{\f y } = \nu_{\f y,\ft S}^{\f d} \otimes   \bar{\nu}_{\f y} \, . $$
Here, $\nu_{\f y,\ft S}^{\f d} $ and $  \bar{\nu}_{\f y}$ are both probability measures with respect to $ \bar{\nu}_{\f y}$ and $ \tilde{m}$, respectively. 
%%%%%%%%%%%%%%% braucht man das hier?

Now taking $\tilde{f}$ as the constant function $1$,  $ \tilde{f}\equiv 1$, one gets the convergence
\begin{align*}
(1 + |\fd d|^2)(1 + |\nabla \fd d|^2)   \stackrel{*}{\rightharpoonup} \tilde{m}
\end{align*}
weakly$^*$ in $ \M(\ov Q)$.  This implies $p(\f y)\geq 1$ almost everywhere in $Q$ and $\tilde{m}\in \M^+(\ov Q)$.

Recall that the relative compactness of $\fd d$  in $\f L^2(Q)$ implies the strong convergence of a (not relabled) subsequence $\fd d$ to $\f d$ in $\f L^2(Q)$ and consequently the point-wise convergence of $\fd d(\f y)$ to $ \f d ( \f y)$ a.e.~in $Q$ as well as the existence of a dominating function in $\f L^2(Q)$.

Consider the function $f(\f h , \f S)=1+|\f h|^2$ the associated transformed function (see~\eqref{transi}) is given by $ \tilde{f}( \f y ,\tilde{ \f h} , \tilde{\f S}) = ( 1 - |\tilde{\f S}|^2)$. Inserting this function into  \eqref{convmeasures} yields
\begin{align*}
1+ | \f d( \f y) |^2 = \int_{\overline{B}_{d\times d } }\int_{\overline{B}_{d } } ( 1 - |\tilde{\f S}|^2) \tilde{\nu}_{\f y }(\de \ft h,\de \ft S) (p( \f y)  + m_s ) \,.
\end{align*}
The  function $( 1 - |\tilde{\f S}|^2)$  only vanishes on the set where the norm of  $\tilde{S}$ is equal to one, i.e.~on the set $\overline{B}_d \times \Se^{d^2-1}$. The measure $m_s$ was mutually singular, which now shows that 
\begin{align*}
 \int_{{B}_{d\times d } }  ( 1 - |\tilde{\f S}|^2)  \bar{\nu}_{\f y }(\de \tilde{\f S}) p( \f y) & =1+ |\f d(\f y)|^2\quad \text{ a.e.~with respect to the Lebesgue measure,} \\
 \int_{\overline{B}_{d\times d } } ( 1 - |\tilde{\f S}|^2) \bar{\nu}_{\f y }(\de \tilde{\f S})   & =0 \quad \text{ a.e.~with respect to } m_s\, .
\end{align*}
This now allows us to assign $p$ as 
\begin{align}
p (\f y ) := \left (\int_{{B}_{d\times d } }  ( 1 - |\tilde{\f S}|^2)  \bar{\nu}_{\f y }(\de \tilde{\f S})\right )^{-1}(1+| \f d(\f y)|^2)\,\label{p}
\end{align}
and to deduce that $ \tilde{\nu}_{\f y }$ is supported on $\overline{B}_d \times \Se^{d^2-1}$ $m_s$ a.e.~on $\ov Q$.
%%%%%%%%%%%%%%%%%%%%%%%%%%%%%%%%%%%%%%%%%%%%%%%%%%%%%%%%%%%%%%%%%%%%%%%%%%%%%%%%%%%%%%%%%%%%%%%%%%%%%%

For $\phi \in \C_{b}(\ov Q \times \R^d) $ we consider the test function $ f ( \f y, \f h, \f S) : = \phi (\f y, \f h ) ( 1+ | \f h|^2)$. 
On the one hand, due to the point-wise strong convergence of $\fd d$  to $\f d$ in $Q$ (see~\eqref{sr:d}) and the dominating function in $\f L^2(Q)$ we get that 
\begin{align*}
\lim_{\delta \ra 0} \int_Q f ( \f y , \fd d ( \f y ) , \nabla \fd d( \f y) )\de \f y  = \lim_{\delta \ra 0}\int_Q\phi(\f y, \fd d ) ( 1 + | \fd d(\f y)|^2)\de \f y = \int_Q\phi(\f y, \f d ) ( 1 + | \f d(\f y)|^2)\de \f y \,.
\end{align*}
On the other hand, the convergence result~\eqref{convmeasures} implies
\begin{align*}
\lim_{\delta \ra 0} f ( \f y , \fd d ( \f y ) , \nabla \fd d( \f y) )  \rightharpoonup ^* 
\int_{\overline{B}_{d\times d } } \int_{\overline{B}_d }
\phi\left  (\frac{\ft h}{\sqrt{1-| \ft h|^2}}\right ) 
 \nu_{\f y, \ft S}^{\f d}(\de \ft h ) (1-|\ft S|^2) \bar{\nu}_{\f y }( \de \ft S) \tilde{m}\,.
\end{align*}
Using~\eqref{p}, the definition of the measure $\tilde{m}$, and since $\nu_{\f y, \ft S}^{\f d}$ is a probability measure, we get
\begin{align*}
0 =\int_{\Omega}\int_{\overline{B}_{d\times d } } \int_{\overline{B}_d }\left ( \phi(\f y, \f d (\f y)) - 
\phi\left  (\f y, \frac{\ft h}{\sqrt{1-| \ft h|^2}}\right ) \right ) 
 \nu_{\f y, \ft S}^{\f d}(\de \ft h ) (1-|\ft S|^2) \bar{\nu}_{\f y }( \de \ft S) \tilde{m}(\de \f y)\,.
\end{align*}
We see that the function vanishes for all values of $( \ft h, \ft S)$ with $| \ft S|<1$. This means that the measure $ \nu_{\f y, \ft S}^{\f d}$ is concentrated on $\f d(\f y)/\sqrt{1+|\f d (\f y)|^2}$ for 
$ \bar{\nu}_{\f y }$ a.\,e.~$\f S \in B_{d\times d}$ and $\tilde{m}$ a.e.~
$  \f y\in \ov Q $. 

With the additional properties of $\tilde{\nu}_{\f y}$ we now define the projections of this measure onto the interior and the boundary of $ \overline{B}_{d\times d }$. For a continuous bounded function  $\varphi\in \C_0(\R^{d\times d}) $ we define the measure ${\nu}^o_{\f y}\in \mathcal{P}( \R^{d\times d})$ via 
\begin{align*}
\int_{\R^{d\times d} } \varphi (\f S){\nu}^o_{\f y}( \de \f S) &: = \frac{1}{1+ | \f d( \f y)|^2}\int_{B_{d\times d} }\varphi \left ( \frac{\ft S}{\sqrt{1- | \ft S|^2}} ) \right ) ( 1 - |\ft S|)^{2}\bar{\nu}_{\f y} (\de \ft S) p( \f y) %\\
%& = \int_{B_{d}}\int_{B_{d\times d} }\varphi \left ( \Phi^{-1}( \ft h, \ft S) ) \right )( 1 - |\ft h|)^{2} ( 1 - |\ft S|)^{2}\nu_{\f y, \ft S}^{\f d}(\de \ft h ) \bar{\nu}_{\f y }( \de \ft S) p( \f y) 
\, .
\end{align*}
With the considerations above, we see that the following identity holds for all functions~$\varphi\in \C_0(\R^d\times\R^{d\times d})$:
 \begin{align*}
 \int_{\R^{d\times d} }& \varphi (\f d(\f y),\f S){\nu}^o_{\f y}( \de \f S)\\ &= \frac{1}{1+ | \f d( \f y)|^2}\int_{B_{d\times d} }\varphi \left (\f d ( \f y), \frac{\ft S}{\sqrt{1- | \ft S|^2}} ) \right ) ( 1 - |\ft S|)^{2}\bar{\nu}_{\f y} (\de \ft S) p( \f y) \\
 &= \int_{B_{d\times d} }\int_{\overline{B}_{d}}\varphi \left ( \frac{\ft h}{\sqrt{1-|\ft h|^2}}, \frac{\ft S}{\sqrt{1-|\ft S|^2}}\right ) ( 1- | \ft h |^2)( 1 - |\ft S|^{2})\nu_{\f y, \ft S}^{\f d}(\de \ft h ) \bar{\nu}_{\f y} (\de \ft S) p( \f y)\\
 &= \int_{B_{d\times d} }\int_{\overline{B}_{d}}\tilde{\varphi}  (\ft h, \ft S)  \tilde{\nu}_{\f y} (\de \ft h , \de \ft S) p( \f y)\,.
 \end{align*}
 
 Additionally, we basically take the remaining part of the measure $\tilde{m}$, which is supported on $\overline{B}_d \times \Se^{d^2-1}$, and define the measure $m$ 
via 
\begin{align*}
 m &:= (p( \f y) + m_s) \tilde{\nu}_{\f y}(\overline{B}_d \times \Se^{d^2-1}) % \\ 
%\ov m &:= (p( \f y)\mu + m_s) \tilde{\nu}_{\f y}(  \Se^{d-1} \times B_{d\times d}) 
\, .
\end{align*}
 The probability measure $ \nu^{\infty}_{ \f y} $ on $\overline{B}_d \times\Se^{d^2-1}$ is defined for every continuous bounded function  $\varphi \in \C_b( \overline{B}_d \times\Se^{d^2-1})$ via
\begin{align*}
\int_{\overline{B}_d}\int_{\Se^{ d^2 -1} } \varphi(\ft h, \ft S) \nu^\infty_{\f y} (\de \ft h, \de \ft S) := \frac{ 1 }{\tilde{\nu }_{\f y} ( \overline{B}_d \times \Se^{d^2-1})} \int_{\overline{B}_d}\int_{S^{d^2-1}} \varphi( \ft h, \ft S)  \tilde{\nu }_{\f y} (\de \ft h , \de \ft S) \, .
\end{align*}
This different definitions taken together imply
\begin{align*}
\int_{\R^{d\times d} } f (\f y,\f d(\f y),\f S){\nu}^o_{\f y}( \de \f S)p(\f y)\mu &+ \int_{\overline{B}_d}\int_{\Se^{ d^2 -1} } \tilde{f}(\f y,\ft h, \ft S) \nu^\infty_{\f y} (\de \ft h ,\de \ft S)m \\
%& \qquad+\int_{\Se^{d-1}}\int_{B_{d\times d} } \tilde{f}(\ft h, \ft S) \ov{\nu}^\infty_{\f y} (\de \ft h, \de \ft S)\ov{m}
&= \int_{\overline{B}_{d\times d } } \int_{\overline{B}_d } \tilde{f}( \f y , \ft h , \ft S )  \tilde{\nu}_{\f y }(\de \ft h\de \ft S) \tilde{m}
%
%
%\int_{\R^{d\times d} } f(S) \nu_{\f y}( \de S) + \int_{\Se^{d^2-1}} \tilde{ f }(S) \nu_{\f y}^\infty (\de S ) = \int_{\overline{B}_{d\times d} } \tilde{f}(\tilde{S})\hat{\nu}_{\f y}( \de \tilde{ S})\tilde{m}\, 
\end{align*}
for all $f \in\mathcal{R}$.
Inserting the new defined measures into the convergence result~\eqref{convmeasures} gives the asserted result~\eqref{convinS}.
The weak convergence of $\nabla \fd  d$ and~\eqref{convinS} imply the asserted equation~\eqref{gradmeas}.

\end{proof}

\begin{rem}
The biting limit of a sequence as given in Proposition~\ref{lem:meas} is given by the classical Young measure generated by this sequence. For functions $f\in \mathcal{R}$,
% with $\tilde{f} ( \f x , \ft h, \ft S)= 0  $ for all $ ( \f x ,\ft h , \ft S)\in Q \times B_d \times \Se^{d^2-1}$
 we can deduce
\begin{align}
f(\f y, \fd d , \nabla \fd d  )  \stackrel{b}{\rightharpoonup}  \int_{\R^{d \times d}} f ( \f y, \f d( \f y),\f  S) \nu_{ \f y} ( \text{\emph{d}} \f S) \, .
\end{align}
It also holds that for $f \in\mathcal{R}$, the sequence $ f(\cdot , \fd d , \nabla \fd d) $ is weakly convergent in $\f L^1( Q)$ if and only if $$
\int_{\overline{B}_d}\int_{\Se^{ d^2 -1} } \tilde{f}(\f y, \ft h, \ft S) \nu^\infty_{\f y} (\de \ft h \de \ft S)m=0\, .$$ 
Moreover, $|\nabla \fd d |^2| \fd d |^2 $ is weakly convergent in $L^1( Q)$ if and only if the measure $m$ vanishes.

The proof of this result is obtained by adapting all the steps in the proof of Theorem~9 in~\cite{alibert} to the case of Proposition~\ref{lem:meas}.
\end{rem}

\subsection{Additional properties of generalized gradient Young measures}
%%%%%%%%%%%%%%%%%%%%%%%%%%%%%%%%%%%%%%%%%%%%%
%% Proposition zu den Massen in der Zeit!
%%%%%%%%%%%%%%%%%%%%%%%%%%%%%%%%%%%%%%%%%%%%%%
The previous proposition (Proposition~\ref{lem:meas}) only uses the $L^2(0,T;\Le)$ boundedness of the sequence $\{\fd d \}$. The following proposition is an adaptation of the considerations in~\cite[section 3]{weakstrongeuler} to our case and indicates the additional properties of the generalized Young measure due to the $L^\infty(0,T;\Le)$ bound which holds for the considered sequence.

\begin{proposition}\label{lem:timemeas}
Let $\fd d :\Omega \times [0,T] \ra \R^d$ be a family of functions fulfilling the assumptions of Theorem~\ref{thm:meas} and let this sequence  generate a generalized Young measure $(\nu^o, m , \nu^\infty)$. Then 
\begin{align}
\esssup_t \left ( \int_{\Omega} \langle| \f d(\f x, t)|^2 | \cdot |^2 , \nu^o_{\f x,t} \rangle \de \f x \right ) < \infty \,, \quad \esssup_t \left ( \int_{\Omega} \langle| \cdot |^2 , \nu^o_{\f x,t} \rangle \de \f x \right ) < \infty \,, \label{esssup}
\end{align} 
and the concentration measure $m$ admits a desintegration of the form 
\begin{align}
m ( \de \f x, \de t ) = m_t ( \de \f x ) \otimes \de t \,, \label{desin}
\end{align}
 where $t \mapsto  m_t$ is a bounded measurable map from $[0,T]$ into $\M^+(\ov{\Omega})$.
\end{proposition}
\begin{proof}
The application of Proposition~\ref{lem:meas} with $f ( \f h, \f S): = (1+| \f h |^2)(1+ | \f S|^2) $ and the recession function $\tilde{f}\equiv 1$ yields
\begin{align*}
0\leq (1+|\fd d |^2)(1+ | \nabla \fd d|^2) \stackrel{*}{\rightharpoonup}  \langle \nu^o, (1+| \f d|^2 )(1+| \f S  |^2)    \rangle  + \langle  \nu^{\infty }, 1  \rangle m \, 
\end{align*}
in $ \M ( \Omega \times [0,T])$. 

The canonical projection of the measure $m$ onto $[0,T]$ is defined by $\bar{m}( E) : = m( \ov \Omega \times E)$ for every Borel subset $E \subset[0,T]$. By the standard desintegration theorem of measures (see Evans and Gariepy~\cite[section 1.6.2]{evans2}), there exists a probability measure $\tilde{m}_t$ such that $m ( \de \f x , \de t ) = \tilde{m}_t( \de x) \otimes \bar{m}(\de t)$.

For  $ \varphi \in \C_c( [0,T] )$ with $ \varphi(t)\geq 0$ for all $t \in [0,T]$  we get
\begin{multline}
\intte{\intet{ \varphi ( t)(1+ | \fd d( \f x ,t)|^2)(1+|\nabla  \fd d( \f x ,t)|^2) }} \\ \longrightarrow \intte{\intet{  \varphi( t) \langle \nu^o_{( \f x, t)},(1+ | \f d(\f x, t)|^2)(1+ | \cdot|^2) \rangle   }}+  \int_0^T \varphi(t) \bar{m} ( \de t)\, .  \label{convgenYou}
\end{multline}
Remark that $ \nu_{(\f x, t)}^\infty $ and $\tilde{m}_t$ are probability measures and thus
\begin{align*}
\int_{\ov \Omega} \langle 1 , \nu_{(\f x, t)} ^ \infty \rangle \tilde{m}_t ( \de \f x) =1\, .
\end{align*}
Due to the a priori estimates holding for~$\fd d$ (see~\eqref{apri3}), we get
\begin{align}
\begin{split}
&\left | \intte{  \intet{  \varphi( t) \langle \nu_{( \f x, t)}^o , (1+| \f d(\f x, t)|^2)(1+ | \cdot|^2) \rangle   } } \right |\\&\qquad \leq \sup_{\delta \in (0,1)} \left | \intte{\varphi(t)\left (\left \| |\fd d(t)|\nabla \fd d(t) \right \|_{\Le}^2 + \| \nabla \fd d(t)\|_{\Le}^2 + \| \fd d(t) \|_{\Le}^2 +1 \right )} \right |\\&\qquad\leq \sup_{\delta\in(0,1)}  \left (\left \| |\fd d(t)|\nabla \fd d(t) \right \|_{L^\infty(\Le)}^2 + \| \nabla \fd d(t)\|_{L^\infty(\Le)}^2 + \| \fd d(t) \|_{L^\infty(\Le)}^2 +1 \right ) \| \varphi\|_{L^1(0,T)}\, 
\end{split}
\label{testedphi}
\end{align}
and hence the assertion of~\eqref{esssup}. 

The convergence~\eqref{convgenYou} together with the estimate~\eqref{testedphi} additionally implies
\begin{align*}
\left | \intte{\varphi (t) \bar{m}(\de t)}  \right |  \leq \sup_{\delta\in(0,1)} \esssup_{ t \in [0,T]} \left (\left \| |\fd d(t)|\nabla \fd d(t) \right \|_{\Le}^2 + \| \nabla \fd d(t)\|_{\Le}^2 + \| \fd d(t) \|_{\Le}^2 +1 \right )\| \varphi\|_{L^1(0,T)}\, .
\end{align*}
This shows that $\bar{m}$ is absolutely continuous with respect to the Lebesgue measure on $(0,T)$.
 By the Radon-Nikod\'{y}m theorem (see Evans and Gariepy~\cite[section 1.6.2]{evans2}), there exists a function $g\in L^1(0,T)$ with 
\begin{align*}
\int_0^T\varphi (t) \bar{m}(\de t) = \intte{\varphi(t) g(t) }\,  \quad \text{for all } \varphi \in \C([0,T])\, .
\end{align*}
Setting $m_t = g(t) \tilde{m}_t$, we find the desintegration property~\eqref{desin}.

\end{proof}

%%%%%%%%%%%%%%%%%%%%%%%%%%55
%%%%%%%%%%%%%%%%%%%%%%%%%%%%%
%%%%%%%%%%%%%%%%%%%%%%%%%%%%%

\subsection{Defect measure}
A similar statement as in Theorem~\ref{thm:young} is valid for families of functions which are bounded in the  sense of the following theorem.%$\sup_{\delta\in(0,1)}\delta \|\nabla ^2 \fn d \|_{\Le}^2 <\infty$.
\begin{theorem}\label{thm:defectmeas}
Let $\{ \fd d\}$ be a family of functions 
fulfilling
\begin{align}
\sup _{\delta\in (0,1)} \left (\delta \| \Lap \fd d  \|_{L^\infty(\Le)}^2 + \| \fd d\|_{L^\infty(\He)}^2\right ) < \infty\label{deltacoer}
 \,.
\end{align}
Then there exists a subsequence $\{ \fdk d \}$, a defect measure $\mu_t \in \M^+(\ov \Omega)$ for a.e.~$t\in(0,T)$ and a family of probability measure $\{\nu^\mu \}\subset \mathcal{P}(\Se ^{d^3-1}) $ for $\mu_t$ a.e.~$\f x \in \ov\Omega$ such that 
\begin{align*}
\int_0^T \int_{\Omega} f\left (\f x ,t,  \frac{\nabla^2\fdk d (\f x ,t) }{| \nabla^2 \fdk d(\f x ,t)|}\right ) \delta_k  | \nabla ^2 \fdk d(\f x ,t)|^2 \de \f x \de t 
\longrightarrow \int_0^T \int_{\Omega} \int_{\Se^{d^3-1}} f\left (  \f x,t, \f \Gamma \right )  \nu^\mu_{(\f x, t)} ( \de \f \Gamma ) \mu_t ( \de \f x ) \de t \,
\end{align*}
holds for all $f\in \C(\ov \Omega\times [0,T]   \times \Se^{d^3-1})$ and for~$\delta_k \ra 0$.

Additionally, $ \esssup_{t\in(0,T)} \ll{ \mu_t , 1} <\infty$ and 
\begin{align*}
\lim_{k\ra\infty } \int_0^T \int_{\ov \Omega }\phi(t) \varphi(\f x ) \delta _k| \nabla ^2 \fdk d(\f x ,t) |^2 \de\f x \de t = {}& \int_0^T \int_{\ov \Omega }\phi(t) \varphi(\f x ) \mu_t( \de\f x) \de t \\  ={}& \lim_{k\ra\infty} \int_0^T \int_{\ov \Omega }\phi(t) \varphi(\f x ) \delta_k |\Lap  \fdk d(\f x ,t) |^2 \de\f x \de t\,
\end{align*}
for all $\varphi \in \C^\infty_c(\Omega)$ and $\phi\in\C([0,T])$.
 
\end{theorem}

\begin{proof}
The Radon measures $\M ( \ov \Omega \times [0,T]\times \Se^{d^3-1})$ are identified with the dual space of the continuous functions $\C( \ov\Omega \times [0,T]\times \Se^{d^3-1})$ (see {Edwards}~\cite[Theorem 4.10.1]{edwards}).
The family of measures 
$ \{ L_{\delta} \} \subset \M ( \ov\Omega \times [0,T]\times \Se^{d^3-1})$ is given by  
\begin{align*}
\langle L_{\delta} , g \rangle : = \int_0^T \int_{\ov \Omega} g\left (\f x, t ,  \frac{ \nabla ^2 \fd d (\f x, t) }{| \nabla ^2 \fd d(\f x, t) | } \right ) \delta | \nabla ^2 \fd d(\f x, t) |^2 \de \f x \de t  \,
\end{align*}
for all  $g \in \C(\ov  \Omega \times [0,T]\times \Se^{d^3-1})$.
The boundedness~\eqref{deltacoer} yields
 $$\| \nabla ^2 \fd d(t) \|_{\Le}\leq c \| \Lap \fd d(t) \|_{\Le}+ c \| \fd d (t)\|_{\He}\leq c \,,$$
such that the Banach--Alaoglu--Bourbaki theorem  provides the existence of a weakly$^*$ converging subsequence $\{ L_{\delta_k}\}\subset \{ L_{\nicefrac{1}{n}}\}$ with $n\in\N$, i.\,e.
\begin{align*}
L_{\delta_k} \stackrel{*}{\rightharpoonup} L  \quad \text{in } \M (\ov \Omega \times [0,T]\times \Se^{d^3-1})
\end{align*}
for $\delta_k \ra 0$.
The classical desintegration argument (see~{Evans}~\cite[Theorem 10.]{evans}
%, \textsc{Fonseca}~\cite[Proposition 3.2.]{Fonseca} 
or {Ambrosio, Fusco and Pallara}~\cite[Theorem 2.28]{BoundVar})
shows the existence of a probability measure $\nu^\mu\in \mathcal{P}(\Se^{d^3-1})$ and a measure $\bar{\mu}\in \M(\ov\Omega \times [0,T])$ such that 
\begin{align*}
\langle L , g \rangle = \int_0^T \int_{\Omega} \int_{\Se^{d^3-1}} f\left ( t, \f x, \f \Gamma \right )  \nu^\mu_{(\f x, t)}( \de \f \Gamma ) \bar\mu ( \de \f x ,\de  t) \,.
\end{align*} 
Hence, for the test function $f \equiv 1$ we get $$ \delta | \nabla ^2 \fdk d(\f x,t)  |^2   \stackrel{*}{\rightharpoonup} \bar \mu   \quad \text{in } \M (\ov \Omega \times [0,T])$$ and thus $\bar \mu \in \M^+(\ov \Omega \times [0,T])$. 
 Like in Proposition~\ref{lem:timemeas}, the desintegration argument is again applied to $\bar{\mu}$ such that
 $ \bar \mu = \mu^1 \otimes \mu ^2$, where $ \mu^1 \in \mathcal{P}(\ov \Omega)$ and $ \mu^2 \in \M^+([0,T])$.
 Additionally, for the function $f \equiv 1$ it holds
 \begin{align}
\int_0^T  \phi(t) \mu^2 (\de t )  \leq \displaystyle \sup_{\delta\in(0,1) } \esssup_{t\in(0,T)} \delta  \| \nabla^2 \fd d \|_{\Le}^2  \| \phi \|_{L^1(0,T)}\leq c \| \phi \|_{L^1(0,T)}\,\label{muzwei}
\end{align}
for all  $\phi\in\C_c^\infty(0,T)$ with $\phi(t) \geq 0$. As a consequence $\mu^2$ is absolutely continuous with respect to the Lebesgue measure
(see~{Brenier, De Lellis \& Sz{\'e}kelyhidi}~\cite{weakstrongeuler} or {Elstrodt}~\cite[Kapitel VIII, Satz 2.5]{elstrodt}). Thus, the Radon--Nikod\'{y}m derivative of $\mu^2$ with respect to the Lebesgue measure exists (see~{Halmos} \cite[Section 32, Theorem A]{halmos}). 
There is a function $g\in L^1(0,T)$
 such that $\mu^2(\de t) = g(t) \de t $.
The first assertion of Theorem~\ref{thm:defectmeas} is reached by setting  $\mu_t = g(t) \mu^1_t$.   
The estimate for $\mu_t$  is a direct consequence of inequality~\eqref{muzwei}.

Using a partial integration, we see 
\begin{align}
\begin{split}
\delta \left (   | \Lap \fd d |^2 ,\varphi \right ) = {}& - \delta \left(      \nabla  \fd d :\nabla \Lap \fd d , \varphi\right ) 
%+\delta \left \langle   \f \gamma_{\f n} (\f \Lambda : \nabla  \fd d) , \f \gamma_0( \Lap \fd d  \varphi) \right \rangle _{\partial\Omega} 
%\\
%&
 -\delta \left(  \nabla \fd d ;  \Lap \fd d \otimes \nabla  \varphi \right)  \\
={}& \delta\left (  | \nabla ^2  \fd d |^2 ,\varphi \right )
%+ \delta\left \langle   \f \gamma_{\f n} (\f \Lambda : \nabla  \fd d) , \f \gamma_0( \Lap \fd d  \varphi) \right \rangle _{\partial\Omega}  
%\\
%&
%+ \delta\left ( \f \Lambda : \nabla \fd d ;  \nabla \f \Lambda : \nabla \fd d \cdot \nabla  \varphi \right )
%- \delta\left \langle   \f \gamma_{\f n} (\f \Lambda : \nabla  \fd d) , \f \gamma_0( \nabla \f \Lambda : \nabla  \fd d  \varphi) \right \rangle _{\partial\Omega} 
%\\
%& 
+ \delta \left ( \nabla \fd d : \nabla ^2 \fd d  ,\nabla  \varphi\right )
-\delta \left(   \nabla \fd d ;  \Lap \fd d \otimes \nabla   \varphi \right)\,
\end{split}\label{massiden}
   \end{align}
for all $\varphi \in \C^\infty_c(\Omega)$. 

The terms on the right-hand side of~\eqref{massiden} can be estimated by 
\begin{align*}
\delta \left ( \nabla \fd d : \nabla^2 \fd d , \nabla  \varphi \right )
 - \delta \left(   \nabla \fd d ;  \Lap \fd d \otimes\nabla  \varphi \right) &\leq 
 c\delta  \| \nabla \fd d \|_{\Le} \|  \nabla^2 \fd d \|_{\Le} \| \nabla \varphi\|_{\f L^\infty}\\& \leq c\sqrt{\delta} \left ( \| \nabla \fd d \|_{\Le} ( \delta \| \Lap \fd d  \|_{\Le}^2+ \| \fd d \|_{\He}^2)^{\nicefrac{1}{2}}  \| \nabla \varphi\|_{\f L^\infty}\right )\,.
\end{align*}
Hence, this terms vanishes for $\delta \ra 0$. 
\end{proof}

%%%%%%%%%%%%%%%%%%%%%%%%%%%%%%%%%%%%%%%%
%%%%%%%%%%%%%%%%%%%%%%%%%%%%%%%%%%%%%
%%%%%%%%%%%%%%%%%%%%%%%%%%%%%%%%%%%%%

\section{Existence of weak solutions to the regularised system\label{sec:weak}}
\subsection{Galerkin basis and solvability of the approximate problem}
In this section, we argue in the same way as in~\cite{unsere} and therefore, we refer to this previous work.
The approximation scheme is similar to the one in~\cite{unsere}.
To approximate  the Navier--Stokes-like equation we use again the eigenfunctions of the Stokes operator $( \f w_i)_{i\in \N}$
with the associated sequence of Galerkin spaces  $W_n : = \spa\{\f  w_1, \f w_2 , \f w_3, \ldots \}$ and sequence of $\Le$-orthogonal projections $P_n : \Ha \ra W_n$. Remark that $\Omega$ is of class $\C^{3,1}$ such that the family of projections $P_n$ is continuous as a mapping of $\f H^2 \cap \V$ to itself (see~\cite{malek}).

For the regularized director equation, we choose eigenfunctions of the differential operator corresponding to the boundary value problem
\begin{align}
\begin{split}
- \Lap \f z  &= \f h \quad\text{in } \Omega \, , \\
\f z &= 0\quad\text{on } \partial\Omega \,.
\end{split}\label{bvp}
\end{align}
Since $\f \Lambda$ is strongly elliptic~\eqref{ellip} and symmetric, i.e.~$\f \Lambda _{ijkl}=\f \Lambda_{klij}$, 
 the above problem  is a symmetric strongly elliptic system that possesses a unique weak solution $\f z \in \f H^1_0$ for any $\f h \in \f H^{-1}$ (see e.g.~Chipot~\cite[Theorem~13.3]{chipot}). Its solution operator is thus a compact operator in $\f L^2$. Hence there exists an orthogonal basis of eigenfunctions $(\f z_n )$.
Moreover, the problem is $\f H^2$-regular (see e.g.~Morrey~\cite[Theorem~6.5.6]{morrey} and recall that $\Omega$ is of class $\mathcal{C}^{3,1}$), i.e.~for any $\f h \in \f L^2$ the solution $\f z$ is in $\f H^2 \cap \f H^1_0$ and there exists a constant $c>0$ such that
\begin{equation}\label{H2-Lambda}
\|\f z\|_{\f H^2} \le  \eta \, \|\Lap \f z\|_{\f L^2}
\end{equation}
for any $\f z \in \f H^2 \cap \f H^1_0$. With a standard bootstrap argument we get, that for every $\f h \in \Hc$, the solution $\f z$ of~\eqref{bvp} is in $\Hg$ and for another constant $c>0$, we have
\begin{equation}\label{H4-Lambda}
\|\f z\|_{\f H^4} \le c \, (\|\Lapp \f z \|_{\f L^2} + \| \f z\|_{\Hc})\,.
\end{equation}

Again, the eigenfunctions form an orthogonal basis in $\Le$. Let $ Z_n : = \spa \left \{  \f z_1, \dots , \f z_n \right \} $ ($n\in \N$) and assume $\|\f z_i\|_{\f L^2} = 1$ for $i=1, 2, \dots$. Then
\begin{equation*}
R_n : \Le \longrightarrow Z_n \, , \quad
R_n  \f f:= \sum_{i=1}^n ( \f f, \f z_i ) \f z_i
\end{equation*}
 is the
$\Le$-orthogonal projection onto $Z_n$.

We define the inverse of the trace operator in an appropriate way for our system. This is done by using the solution operator to the associated stationary problem.

\begin{theorem}[Extension operator]
\label{spur}

There exists a linear continuous operator $\Sr: \f H^{\nicefrac{7}{2}}(\partial \Omega)  \ra \f H^{4}(\Omega)$, where $\Omega$ is of class $\C^{3,1}$.
This operator is the right-inverse of the trace operator, i.e.~for all $\f g\in \f H^{\nicefrac{7}{2}}(\partial \Omega)  $, it holds  $ \Sr\f g = \f g $  on $\partial\Omega$ in the sense of the trace operator.
Additionally, it holds $ \Lap \Sr\f g = 0  \text{ in } \Omega $ and there exists a constant $c>0$ such that
\begin{align}
\| \Sr\f g \|_{\Hg(\Omega)} \leq{}& c \| \f g \|_{\Hrand{7}}\,  \quad \text{for }\f g \in \Hrand{7} \,.\label{H4}
\end{align}
\end{theorem}
\begin{proof}
Let $\Omega$ be of class $\C^{3,1}$. The extension operator is defined via the solution operator of the problem
\begin{align}
-\Lap \f d= 0 \quad \text{in } \Omega\,,\qquad
\f d = \f g \quad \text{on } \partial \Omega\,.
\label{randop}
\end{align}
This problem is uniquely solvable for a tensor enjoying the strong ellipticity (see~{McLean}~\cite[Theorem 4.10]{mclean}). 
The associated solution operator is linear and continuous and the regularity of this problem asserts 
 (vgl.~{McLean}~\cite[Theorem 4.21]{mclean})
\begin{align*}
\Sr : \f H^{s-\nicefrac{1}{2}}(\partial \Omega)  \ra \f H^{s}(\Omega)   \, \quad \text{for all $s$ with } 0 \leq s \leq 4 \,.
\end{align*}
 
We remark that $\f \Lambda $ as defined in~\eqref{Lambda} is strongly elliptic (see~\eqref{ellip}).

\end{proof}

%%%%%%%%%%%%%%%%%%%%%%%%%%%%%
%%%%%%%%%% hier muss noch das diskrete System hin
%%%%%%%%%%%%%%%%%%%%%%%%%%%%%
The approximate system is similar to the one in~\cite{unsere}.
Let $n\in \N$ be fixed. As usual, we consider the ansatz
\begin{align}
\fn v ( t)  = \sum_{i=1}^n v_n^i(t)\f w_i, \quad \fn d(t) = \Sr \f d_1 + \sum_{i=1}^nd_n^i(t)\f z_i \, \label{dar}
\end{align}
with $( v_n^i, d_n^i) \in \AC([0,T]) $ for all $i=1,\ldots ,n$.
 
Our approximation reads as
\begin{subequations}\label{eq:dis}
\begin{align}
\begin{split}
( \partial_t {\fn v}, \f w  ) +( ( \fn v \cdot \nabla ) \fn v, \f w  ) -(\nabla \fn d ^ T  \fn q , \f w  )+ \left (\f T^L_{n,\delta}: \nabla \f w \right )&= \left \langle  \f g   , \f w\right  \rangle,\\
\fn v(0) &= P_n \f v_0 \,,
\end{split}
\label{vdis}\\
\begin{split}
( \partial_t \fn d + ( \fn v \cdot \nabla ) \fn d 
- \skn{v} \fn d , \f z ) + \lambda ( \syn v \fn d , \f z)
+  (\fn q , \f z ) &=0 \,,
\\ \fn d(0)&= R_n\f d_{0}\,
\end{split}\label{ddis}
\end{align}
for all $ \f w \in W_n$ and $ \f z \in Z_n$,
where  $ \f q_{n,\delta} $ is given by the projection of the variational derivative of the free energy
\begin{align}
{\f q_{n,\delta}} : = 
%R_n \left ( \frac{\delta \mathcal{F}}{\delta \f d }( \fn d )\right ) = 
R_n \left (F_{\f h}( \fn d ,\nabla\fn d ) - \di F_{\f S}( \fn d , \nabla \fn d) + \frac{1}{\varepsilon}(|\fn d |^2 -1) \fn d  \right )+\delta \Lapp \fn d \, ,\label{qn}
\end{align}
and 
\begin{align}
\begin{split}
{\f T}_{n,\delta}^L :={}&  \mu_1 (\fn d \cdot  \syn v \fn d )(\fn d \otimes \fn d)+\mu_4 \syn v  - (\mu_2+\mu_3) \left (\fn d \otimes \fn q     \right )_{\sym} \\&
- \left (\fn d \otimes \fn q   \right)_{\skw} + ((\mu_5+\mu_6) -\lambda(\mu_2+\mu_3))  \left ( \fn d \otimes \syn v \fn d \right)_{\sym}\, 
\end{split}\label{lesliedis}
\end{align}%
is the discrete Leslie stress, where we replaced $\f e_{n,\delta}$ by $-\lambda\syn v \fn d -  {\f q}_{n,\delta}$ in comparison to formulation~\eqref{Leslie}. 
This allows to write this system as an ordinary differential equation in finite dimensions.
The solvability of this discrete system is rather standard and we refer to~\cite{unsere} for more details.

\end{subequations}%

\subsection{A priori estimates}
To get a priori estimates, we use the important dissipative character of the system. The proof of the energy inequality is given in~\cite[Proposition 4.1]{unsere}. The subsequent corollary works in the same way for our present case. We thus get the following energy equality for the discrete system:
%%%%%%%%%%%%%%%%%%%%%%%%%%%%%%%%
%%% Hier muss die Energie Ungleichung hin...
%%%%%%%%%%%%%%%%%%%%%%%%%%%%%%%%
\begin{align} 
\begin{split}
&  \frac{1}{2}\|\fn v(t)\|_{\Le}^2 
%+ \frac{\delta}{2} \| \Lap \f d(t)\|_{\Le}^2
+  \F_{\delta}(\fn d(t))  +  \int_0^t\left (\mu_1 \left\|\fn d\cdot \syn v \fn d\right\|_{L^2}^2 + \mu_4 \|\syn v \|_{\Le}^2\right ) \de s  \\
&\quad +   \int_0^t\left ((\mu_5+\mu_6- \lambda (\mu_2+\mu_3))\|\syn{v}\fn d\|_{\Le}^2
+  \|\fn q \|_{\Le}^2\right )\de s  \\
&= \frac{1}{2}\|P_n \f v _0\|_{\Le}^2 + \F_{\delta}(R_n\f d_{0})+   \int_0^t \left (\langle \f g, \fn v\rangle + \left (  (\mu_2+\mu_3) -\lambda \right ) ({\f q}_{n,\delta} , \syn v \fn d )\right )\de s \, .
\end{split}\label{entro1}
\end{align}
\begin{proposition}[A priori estimates I]\label{prop:apri}
The solutions $( \fn v , \fn d)$ to the approximate system~\eqref{eq:dis} admit the following a priori estimate. There exists $\alpha, \beta >0$ and a constant $c>0$ independent of $n$ such that
\begin{align}
\begin{split}
&\frac{1}{2}\| \fn v \|_{L^\infty(\f L^2)}^2 
%+  \frac{\delta}{2} \| \Lap \fn d\|_{L^\infty(\Le)}^2+\sup_{t\in[0,T]}
%\mathcal{F}(\fn d(t)) + \frac{1}{4\varepsilon} \left  \| | \fn d |^2-1\right \|_{L^\infty(L^2)}^2
+ \sup_{t\in[0,T]}
\mathcal{F}_{\delta}(\fn d(t)) + \mu_1\left \Vert \fn d\cdot \syn v \fn d \right \Vert_{L^2(L^2)}^2
 \\
&\quad 
+ \frac{\mu_4}{2} \|\syn v \|_{L^2(\Le)}^2+ \alpha \|\syn v\fn d\|_{L^2(\Le)}^2+ \beta  \|\fn q \|_{L^2(\Le)}^2
\\
& \leq \frac{1}{2}\|\f v_0\|_{\Le}^2+  \F_\delta(\f d_0) +  c \left ( \| R_n \f d_0 \|_{\Hc}^3 + \|  \f d_0 \|_{\Hc}^3 + 1\right ) \left \|R_n \f d_0 - \f d_0\right  \|_{\Hc}
+ c  \| \f g\|_{L^2(\Vd)}^2\leq c\,.
\end{split}
\label{entrodiss}
\end{align}
\end{proposition}
\begin{proof}
This proposition can be shown as in Corollary~\cite[Corollary 4.2]{unsere}, we only need another estimate for the free energy evaluated at the projection of the initial values, i.e.~${\F}_{\delta}(R_n \f d_0)$.
 Due to the higher regularity of the initial value $\f d_0$, we can estimate
\begin{align*}
\delta \| \Lap R_n \f d_0\|_{\Le}^2 - \delta \| \Lap \f d_0 \|_{\Le}^2 ={}&\delta \left ( \Lap R_n \f d_0 ,\Lap R_n \f d_0 -\Lap \f d_0 \right ) + \delta\left  (\Lap \f d_0 ,\Lap R_n \f d_0 -\Lap \f d _0 \right ) \\
\leq{}& \delta \left ( \| \Lap R_n \f d_0 \|_{\Le}+ \|\Lap \f d_0\|_{\Le}\right ) \| \Lap R_n \f d_0 - \Lap \f d_0 \|_{\Le} \,. 
\end{align*} 
 Similarly, we get for the Oseen--Frank free energy
 \begin{align*}
 \left ( \nabla R_n \f d _0 ; \f \Lambda : \nabla R_n  \f d_0 \right ) -  \left ( \nabla   \f d _0 ; \f \Lambda :   \nabla \f d_0 \right ) \leq {}& c \left ( \|R_n \nabla \f d_0 \|_{\Le} + \|  \nabla \f d_0\|_{\Le}  \right ) \|  \nabla R_n \f d_0 - \nabla \f d_0 \|_{\Le}\, 
 \end{align*}
 as well as with Youngs inequality
 \begin{align*}
 \left ( \nabla R_n \f d _0 \otimes R_n \f d_0 \dreidotkom \f \Theta \dreidots \nabla R_n  \f d_0  \otimes R_n\f d_0 \right ) -  \left ( \nabla  \f d _0 \otimes  \f d_0 \dreidotkom \f \Theta \dreidots \nabla   \f d_0  \otimes \f d_0 \right ) \leq c \left ( \| R_n \f d_0 \|_{\f W^{1,4}}^3 + \| \f d_0 \|_{\f W^{1,4}}^3  \right ) \| R_n \f d_0 - \f d_0 \|_{\f W^{1,4}}\,.
 \end{align*}
 For the penalization term, we get
 \begin{align*}
 \frac{1}{4\varepsilon}\left \| | R_n \f d_0 |^2 -1 \right \|_{\Le}^2 - \frac{1}{4\varepsilon} \left \| |  \f d_0 |^2 -1 \right \|_{\Le}^2 \leq  c \left (  \| R_n \f d_0 \|_{\f L^4 } ^3 + \| \f d_0 \|_{\f L^4}^3 +1   \right ) \| R_n \f d_0 - \f d_0 \|_{\f L^4} \,.
 \end{align*}
 Together, we can estimate with the standard Sobolev embeddings and Young's inequality
 \begin{align*}
 \F_{\delta}(R_n \f d_0 ) \leq  \F_\delta(\f d_0) +  c \left ( \| R_n \f d_0 \|_{\Hc}^3 + \|  \f d_0 \|_{\Hc}^3 + 1\right ) \left \|R_n \f d_0 - \f d_0 \right \|_{\Hc}\,.
  \end{align*}
 Since $R_n$ is the orthogonal  projection on $\Hc$ the right-hand side of the above inequality is bounded independently of $n$. 
 \end{proof}

Initially, the a priori estimate~\eqref{entrodiss} only holds for the maximal time interval on which the solutions to the approximate problem~\eqref{eq:dis} exist. With a standard continuation argument as in~\cite{unsere}, this existence interval can be shown to be $[0,T]$.

\begin{proposition}\label{prop:apri2}
Let the assumptions of Theorem~\ref{thm:weak} be fulfilled. Then there exists a constant $c>0$ independent of $n$, but dependent on $\delta$ such that 
\begin{align}
\begin{split}
\| \f v _n \|_{L^\infty(\f L^2)}^2 +  \|\Lap  \f d_n\|_{L^{\infty}(\Le)}^2+  \|\nabla  \f d_n\|_{L^{\infty}(\Le)}^2   &+ \left \Vert \f d_n \cdot  \syn v \f d_n \right \Vert_{L^2(L^2)}^2 + \|\f v\|_{L^2(\Hb)}^2 \\
&+ \|\syn v\f d_n\|_{L^2(\Le)}^2
+  \| \Lapp \f d_n\|_{L^2(\Le)}^2
\leq c
\, 
\end{split}
\label{apri2reg}
\end{align}
holds for all solutions $(\f v_n , \f d_n) $  of \eqref{eq:dis}.
\end{proposition}
\begin{proof}
With the a priori estimate~\eqref{entrodiss} and Proposition~\ref{coerc1} we get
\begin{align}
\| \Lap \fn d \|_{L^\infty(\Le)}^2 + \| \fn d \|_{L^\infty(\He)}^2 \leq c \,.\label{schranken}
\end{align}
The definition of the variational derivative~\eqref{qtilde} of the free energy and Young's inequality provide
\begin{align*}
\| \fn q \|_{\Le}^2 \geq \frac{1}{2}\| R_n \Lapp \fn d \|_{\Le}^2 - \left \| R_n \left ( F_{\f h} ( \fn d , \nabla \fn d) - \di  F_{\f S} ( \fn d , \nabla \fn d)\right ) \right \|_{\Le}^2 \,.
\end{align*}
Since  $\Lap \Sr \f d_1 =0$, we get $\Lapp \fn d\in Z_n$ and thus 
$ R_n \Lapp \fn d = \Lapp \fn d $.
Additionally, $R_n$ is an orthogonal projection and, using the partial derivatives of the Oseen--Frank energy~\eqref{FSFh}, we can estimate the norm of the variational derivative 
\begin{align*}
& \left \| R_n \left ( F_{\f h} ( \fn d , \nabla \fn d) - \di  F_{\f S} ( \fn d , \nabla \fn d)\right ) \right \|_{\Le}^2 \\ &\qquad\leq{}  \left \|   F_{\f h} ( \fn d , \nabla \fn d) \right \|_{\Le}^2+ \left \|  \di  F_{\f S} ( \fn d , \nabla \fn d) \right \|_{\Le}^2 \\ &\qquad\leq{}  c\left ( \| \fn d \|_{\f W^{1,4}}^4\| \fn d\|_{\f L^\infty}^2  +  \|\fn d \|_{\Hc}^2\right ) + c\left ( \|\fn d \|_{\Hc}^2 \| \fn d \|_{\f L^\infty}^4 + \| \fn d \|_{\f W^{1,4}}^4\| \fn d\|_{\f L^\infty}^2 \right )\,.
\end{align*}
Gagliardo--Nirenberg's inequality (see~\cite[Section~21.19]{zeidler2A})  yields  
\begin{align*}
 \left \|  F_{\f h} ( \fn d , \nabla \fn d) - \di  F_{\f S} ( \fn d , \nabla \fn d) \right \|_{\Le}^2  \leq{}& c \left (\|  \fn d \|_{\Hc }^3 \| \fn d \|_{\He} \| \fn d \|_{\Hc}^{\nicefrac{3}{2}} \| \fn d \|_{\Le}^{\nicefrac{1}{2}}+ \| \fn d \|_{\Hc}^2 \right) 
 \\& +c \left (  \|\fn d \|_{\Hc}^2 \| \fn d \|_{\Hc}^{3} \| \fn d \|_{\Le}\right )\\ \leq{}& c ( \| \fn d \|_{\Hc}^6 +1) \,.
\end{align*}
Due to the coercivity~\eqref{nad} and the estimate~\eqref{entrodiss}, we can bound the right-hand side of the above inequality, which implies the assertion.
\end{proof}
\begin{rem}
It should be emphasized that the last a priori estimate depends on $\delta$. This estimate does not hold for $\delta \ra 0$.

\end{rem}

%%%%%%%%%%%%%%%%%%%%%%%%%%%%%%%%%%%%%%%%
%%%%%%%%%%%%%%%%%%%%%%%%%%%%%%%%%%%

We are now going to estimate the time derivatives of the approximate solutions in appropriate norms.
\begin{proposition}\label{prop:time}
Under the assumptions of Theorem~\ref{thm:weak}  there is a constant $C>0$,
depending on the initial values $\f v_0 , \, \f d_0$
and right-hand side $\f g$, such that for all $n\in\N$ and $\delta\in(0,1)$ 
\begin{align}
 \| \partial_t \fn v \|_{L^{2}((\f H^2 \cap \V)^*)} + 
 \| \partial_t \fn d\|_{L^{2}(\Hd)} \le C \, .\label{esttime}
\end{align}
\end{proposition}

\begin{proof}
The bound on $\{ \partial_t \fn v \}$ follows from the same argumentation as in~\cite[Proposition~4.2]{unsere}.

%%%%%%%%%%%%%%%%%%%%%%%%%%%%%%%5
%%%%%%%%Hier ist der Beweis
%%%%%%%%%%%%%%%%%%%%%%%%%%%%%%%%%%

%%%%%%%%%%%%%%%%%%%%%%%%%%%%%%%%%%%%%%%%%%%5
%%%%%%%%%%%%%%%%%%%%%%%%%%%%%%%%%%%%%%%%%%%%%%%%

Recall that $R_n$ is the $\Le$-orthogonal projection onto $Z_n$ and a continuous mapping between $\Hb$ and itself. With the Sobolev embedding $ \Hb \hookrightarrow \f L^3  $  we thus find with \eqref{ddis} for all $t\in [0,T]$
\begin{align}
\begin{split}
\|\partial_t \fn d\|_{\Hd} &={}
\sup_{{\|\f\psi\|_{\Hb}\leq 1}} | ( \partial_t \fn d , \f\psi)|
= \sup_{{\|\f\psi\|_{\Hb}\leq 1}} | ( \partial_t \fn d , R_n \f\psi ) |
\\
\leq&{} \sup_{{\|\f\psi\|_{\Hb}\leq 1}}
\left\|  -(\fn v\cdot \nabla )\fn d +\left (\skn{v} - \lambda \syn v\right ) \fn d -  \fn q
\right\|_{\f L^{\nicefrac{3}{2}}} \| R_n \f\psi\|_{\f L^3}
\\
\leq&{}c \left (
\left\|  (\fn v\cdot \nabla )\fn d \right\|_{\f L^{\nicefrac{3}{2}}}
+ \left\|\skn{v}\fn d \right\|_{\f L^{\nicefrac{3}{2}}}
+ |\lambda| \left\|\syn v \fn d \right\|_{\f L^{2}}
+  \left\|\fn q
\right\|_{\f L^{2}}\right ) 
 .
\end{split}\label{dntabs}
\end{align}
In view of (\ref{entrodiss}), we see that
\begin{equation*}
\left\|\syn v \fn d \right\|_{L^{2}(\f L^{2})} \text{ and }
\left\|\fn q \right\|_{L^{2}(\f L^{2})}
\end{equation*}
are bounded.
It remains to consider the first two terms on the right-hand side of
\eqref{dntabs}.
With H\"older's inequality, we find
\begin{align*}
\left \| ( \fn v \cdot \nabla ) \fn d\right \| _{L^{2}(\f L^{\nicefrac{3}{2}})}+ 
\left \|\skn{v} \fn d \right \| _{L^{2}(\f L^{\nicefrac{3}{2}})} 
\le \left \|\fn v\right \|_{L^2(\f L^6)} \left \| \nabla \fn d\right \| _{L^{\infty}(\f L^2)}
+ \left \|\fn v\right \|_{L^2( \Hb )}\left \| \fn d \right \|_{ L^{\infty}( \f L^{6} )}
%\le c \|\fn v\|_{L^2(\Hb)} \|\fn d\| _{L^\infty(\Hb)}^{\nicefrac{1}{2}}
%\|\fn d\| _{L^2(\Hc)}^{\nicefrac{1}{2}}
\, .
\end{align*}
Note that all terms on the right-hand side are bounded in view of \eqref{entrodiss}.

This proves the assertion.
\end{proof}

\subsection{Convergence of the approximate solutions}
The a priori estimates in the previous sections are crucial to deduce the convergence of a subsequence of solutions to the approximate system~\eqref{eq:dis}. 
%
%The boundedness of the norms of the sequences achieved with the a priori estimate~\eqref{apri} allows us to use the Banach--Alaoglu theorem to deduce the convergence of a not relabled subsequence in the following senses:
\begin{proposition}\label{lem:limits}
There is a (not relabeled) subsequences $\{( \fn v , \fn d)\}$ of the sequence of solutions  to the approximate systems~\eqref{eq:dis} such that 
\begin{subequations}\label{wkonv}
\begin{align}
   \fn v &\stackrel{*}{\rightharpoonup}  \fd v&\quad&\text{ in } L^{\infty} (0,T;\Ha)\,,\label{w:vstern}\\
 \fn v &\rightharpoonup  \fd v&\quad&\text{ in }  L^{2} (0,T;\V)\,,\label{w:v}\\
% \fn d &\stackrel{*}{\rightharpoonup}  \fd d&\quad&\text{ in } L^{\infty} (0,T;\He)\,,\label{w:dstern}\\
\fn q &\rightharpoonup  \fd {\ov{q}}&\quad&\text{ in }  L^{2} (0,T;\Le)\,,\label{w:E}\\
\syn v \fn d &\rightharpoonup   \syd{v} \fd d &\quad& \text{ in }  L^{2} (0,T;\Le)\,,\label{w:Dd}\\
\fn d\cdot \syn v \fn d &\rightharpoonup  \fd d \cdot \syd{v} \fd d &\quad&\text{ in }  L^{2} (0,T;L^2)\,.\label{w:dDd}\\
\t \fn v  &\rightharpoonup \t \fd v &\quad&\text{ in } L^2(0,T; (\f H^2\cap\V)^*)\, ,\label{timev}\\
\t\fn d &\rightharpoonup \t \fd d&\quad&\text{ in } L^{2}(0,T; \Hd) \, ,\label{timed}
\\
\fn d &\stackrel{*}{\rightharpoonup}  \fd d&\quad&\text{ in } L^{\infty} (0,T;\Hc)\,,\label{w:ddstern}\\
\fn d &\rightharpoonup  \fd d&\quad&\text{ in } L^{2} (0,T;\Hg)\,,\label{w:dd}\\
\fn v &\ra \fd v&\quad&\text{ in }L^p (0,T; \Ha)\text{ for any } p \in [1,\infty)\, ,\label{s:v}\\
\fn d &\ra \fd d&\quad&\text{ in } L^q ( 0,T; \Hc )\text{ for any } q
 \in [1,\infty)\, .\label{s:d}
\end{align}
\end{subequations}
\end{proposition}
\begin{proof}
The existence of a weakly and weakly$^*$ converging subsequence follows from standard arguments from the a priori estimates~\eqref{entrodiss} and~\eqref{apri2reg} as well as~\eqref{esttime}.
The strong convergence follows from the Lions--Aubin compactness lemma
(see Lions~\cite[Th\'eor\`eme~1.5.2]{lions}).
Indeed, with respect to $\fn v$, we observe that $\V$ is compactly embedded
in $\Ha$, which implies strong convergence in
$L^2(0,T;\Ha)$ and together with the boundedness in $L^{\infty}(0,T;\Ha)$ also in $L^p(0,T;\Ha)$ for any $p\in [1,\infty)$.
With respect to $\fn d$, we observe that $\Hg$ is compactly embedded in $\Hc$, which implies strong convergence in $L^2(0,T;\Hc)$ and together with the boundedness in $L^{\infty}(0,T;\Hc)$ also in $L^q(0,T;\Hc)$
for any $q\in [1,\infty)$.  
This strong convergence allows to identify the limits in~\eqref{w:Dd} and~\eqref{w:dDd}.
\end{proof}

\begin{corollary}\label{cor:initial}
Under the assumptions of Theorem~\ref{thm:weak},
the limits $\fd v$ and $\fd d$ from Corollary~\ref{lem:limits} satisfy
\begin{equation*}
\fd v(0) = \f v_0 \text{ and } \fd d(0) = \f d_{0} \, .
\end{equation*}
\end{corollary}
The\textit{Proof} can be found in~\cite[Corollary~4.5]{unsere}.

With the following proposition, we identify the limit $\bar{\f q}_\delta$ in \eqref{w:E}.
\begin{proposition}\label{Eweak}
Under the assumptions of Theorem~\ref{thm:weak}, the limit $\bar{\f q}_\delta$  in~\eqref{w:E} is given by $\bar{\f q}_\delta = \fd q$, where $\fd q $ is given by \eqref{qtilde}.
\end{proposition}

\begin{proof}
We already established the weak convergence~\eqref{w:E}, we thus only need to identify the limit~$\bar{\f q}_\delta$.

Recalling that $R_n$ is the $\Le$-orthogonal projection onto $Z_n$ and $\f\gamma_0(\Delta \fn d)=0$,
we find 
\begin{align*}
\int_0^T \langle \fn q(t) -& \fd{{q}}(t) ,  \f \psi(t) \rangle \de t \notag\\
={}&  \intte{ \left \langle \delta(\Lapp \fn d (t)- \Lapp \fd d(t)) 
,  \f \psi(t) \right \rangle} 
\notag\\
&  + \frac{1}{\varepsilon}\intte{\left (
\left (R_n \left (( | \fn d (t)|^2 -1) \fn d (t)\right )  - ( | \fd d (t) |^2 -1) \fd d (t)\right ),  \f \psi(t) \right )} 
\notag\\
 &+\intte{ \left \langle F_{\f h} ( \fn d (t) , \nabla \fn d (t)) - \di F_{\f S}     ( \fn d (t) , \nabla \fn d (t)), R_n \f \psi(t) \right \rangle} \notag\\
 &-\intte{ \left \langle F_{\f h} ( \fd d(t), \nabla \fd d(t)) - \di F_{\f S}(\fd d(t), \nabla \fd d(t)), \f \psi(t)   \right \rangle     }\notag \\
 ={}
&   \intte{ \left \langle F_{\f h} ( \fn d (t) , \nabla \fn d (t))+ \frac{1}{\varepsilon}(| \fn d |^2 -1)\fn d , R_n \f \psi(t) - \f \psi(t) \right \rangle} 
\\&
+\intte{ \left \langle   F_{\f S}      ( \fn d (t) , \nabla \fn d (t));\nabla ( R_n \f \psi(t) - \f \psi(t)) \right \rangle} \\
&
 +\delta \intte{ \left \langle \Lap \fn d(t) - \Lap \fd d (t), \Lap \f \psi(t) \right \rangle} \notag\\ 
 &+\intte{ \left (F_{\f h} ( \fn d (t) , \nabla \fn d (t)) - F_{\f h} ( \fd d(t), \nabla \fd d(t))  , \f \psi(t)   \right ) }
\\
&+ \intte{ \left \langle  F_{\f S}      ( \fn d (t) , \nabla \fn d (t))-  F_{\f S} (\fd d(t), \nabla \fd d(t)), \nabla \f \psi(t) \right \rangle}\\
 &+\frac{1}{\varepsilon}\intte{ \left (  ( | \fn d (t)|^2 -1) \fn d (t) - ( | \fd d(t) |^2 -1) \fd d(t) , \f \psi(t)   \right ) }
 \\
 ={}& I_{1,n} + I_{2,n} + I_{3,n}+ I_{4,n} + I_{5,n} +I_{6,n}\, 
\end{align*}
for all $\f \psi  \in L^2( 0,T;\Hc\cap \Hb)$.
First, we remark that in regard of definition~\eqref{FSFh}, we have 
\begin{align}
\begin{split}
|F_{\f h} ( \f h , \f S)| &
\leq c ( | \f S|^2 +  | \f h|^2 ) | \f h| \leq c ( | \f S|^3 + | \f h|^3 ) \, ,\\
|F_{\f S} ( \f h , \f S) |& \leq c | \f S | ( | \f h |^2 +1)\leq c( | \f S|^3 + | \f h |^3 +1) \, ,\\
|( | \f h|^2 -1)\f h |& \leq c (|\f h|^3+1)\,
\end{split}
\label{Fwachs}
\end{align}
for all $ \f h \in \R^ d $, $\f S \in \R^{d \times d}$.
Due to standard Sobolev embeddings we know $ \Hc \hookrightarrow \f W^{1,6} \hookrightarrow \f L^\infty$. The a priori bound~\eqref{apri2reg} especially the $L^\infty(0,T;\Hc) $ bound on $\fn d$, together with the estimates~\eqref{Fwachs} shows, that $ F_{\f h} ( \fn d  , \nabla \fn d ) $, $ F_{\f S}      ( \f d_n, \nabla \f d_n) $ and the penalization term are bounded in $ L^\infty(0,T; \f L^2)$ independently of $n$. 
Moreover, $R_n$ is the $\Hb$-orthogonal projection onto $Z_n$ if we equip $\Hb$ with the inner product $( \cdot\, ; \f \Lambda:  \cdot )$. Since the norm induced by this inner product is equivalent to the standard norm, we find that for all $\f \psi  \in L^2( 0,T;\Hb)$
\begin{align*}
\lim_{n\ra \infty} \| R_n \f \psi - \f \psi \|_{L^2(\Hb)}= 0\, .
\end{align*}
This shows that $I_{1,n}$ and $I_{2,n}$ converge to $0$ as $n\to\infty$.
 Due to the strong convergence in $L^\infty(0,T; \Hc)$ (see~\eqref{s:d}), the term~$I_{3,n}$ converges to zero.

Let us now consider the terms $I_{4,n}$, $I_{5,n}$, and $I_{6,n}$.
Due to the strong convergence~\eqref{s:d} and standard Sobolev embeddings, we observe that 
(passing to a subsequence if necessary)
\begin{equation*}
\fn d(\f x,t)  \to \fd d(\f x,t)  \, , \quad
\nabla \fn d (\f x,t)  \to \nabla \fd d(\f x,t)
\end{equation*}
for almost all $(\f x,t) \in \Omega \times (0,T)$. Moreover,
$|\fn d(\f x,t)|$ 
%is majorized by a function in
%$L^6 (0,T;L^6)$  
and $|\nabla\f d_n(\f x,t)|$ are majorized by a function in
$L^6(0,T;  L^6)$. The growth conditions~\eqref{Fwachs} then show that $F_{\f h} ( \fn d (t) , \nabla \fn d (t))$, $F_{\f S}      ( \fn d (t) , \nabla \fn d (t))$ and $\nicefrac{1}{\varepsilon}(| \fn d |^2 -1) \fn d $ are majorized by a function in $L^2 (0,T; \f L^2)$.
With the continuity of $F_{\f h}$ and $F_{\f S}$ as well as Lebesgue's theorem on dominated convergence, we  find that
$I_{4,n}$, $I_{5,n}$ and $I_{6,n}$ converge to $0$ as $n\to\infty$.
\end{proof}

%%%%%%%%%%%%%%%%%%%%%%%%%%%%%%%%%%%%%%%%%%%%%%%%%%
%% Jetzt kommt die Konvergenz der approximation gegen das regularisierte System
%%%%%%%%%%%%%%%%%%%%%%%%%%%%%%%%%%%%%%%%%%%%%%%%%%%
We are now ready to prove that the approximate solution $\{( \fn v, \fn d)\}$ converges to a weak solution of the regularized system~\eqref{defi:weak}.

\begin{proof}[Proof of Theorem~\ref{thm:weak}]
It only remains to prove that the limit $(\fd v, \fd d)$
from Corollary~\ref{lem:limits} satisfies the original problem in the sense of Definition~\ref{defi:weak}. This is shown by passing to the limit in the approximate problem \eqref{eq:dis}.

Let us start with the approximation  \eqref{ddis} of the director equation. First, we observe convergence of the term incorporating the time derivative because of \eqref{timed}. The three semilinear terms converge due to the strong convergence of the director~\eqref{s:d} and the weak as well as the strong convergence of the velocity field~\eqref{w:v} and~\eqref{s:v}. 
Thus, we have 
\begin{multline*}
\int_0^T ((\fn v(t)\cdot \nabla) \fn d(t) -(\nabla \fn v(t))_{\skw} \fn d(t)+\lambda (\nabla \fn v(t))_{\sym}\fn d (t) , \f \psi(t)) \de t\\
\to
\int_0^T ((\fd v(t)\cdot \nabla) \fd d(t)-(\nabla \fd v(t))_{\skw} \fd d(t)+\lambda (\nabla \fd v(t))_{\sym} \fd d(t)  , \f \psi(t)) \de t \, 
\end{multline*}
for all $\f\psi \in \mathcal{C}_c^\infty(\Omega \times (0,T);\R^3))$.
The variational derivative of the free energy converges due to Proposition~\ref{Eweak}.

All this shows that the limit $(\f v, \f d)$ of the approximate solutions satisfy the original equations~\eqref{eq:dir}. Moreover, Corollary~\ref{cor:initial} shows that the initial conditions are also fulfilled.
Remark that in view of the a-priori estimate~\eqref{entrodiss}, the equation 
\begin{align}
 \t \fd d + ( \fd v \cdot \nabla ) \fd v - \skd v \fd d = \fd e = - \lambda \syd v \fd d - \fd q \, \label{gleiche}
\end{align}
holds in $L^2(0,T;\Le)$. Not all terms on the left-hand side of~\eqref{gleiche} are known to be bounded in~$L^2(0,T;\Le)$, but their sum, i.e.~the term $\fd e$, is.

In the following, we focus on the limiting procedure in the approximation
\eqref{vdis} of the Navier--Stokes-like equation. In view of~\eqref{timev}, we already know that the term incorporating the time derivative converges. 
Moreover, we find with
\eqref{s:v} the convergence of the convection term such that for all solenoidal $\f\varphi \in \mathcal{C}_c^\infty(\Omega\times (0,T);\R^3))$
\begin{equation*}
\int_0^T ((\fn v(t)\cdot \nabla) \fn v (t), \f \varphi(t)) \de t
\to
\int_0^T ((\fd v(t)\cdot \nabla) \fd v(t), \f \varphi(t)) \de t \, .
\end{equation*}
With Proposition~\ref{Eweak}, the convergences \eqref{w:E}, \eqref{s:d} and calculation~\eqref{Erikseniden}, we find that
\begin{equation*}
\intte{\left ( \nabla \fn d^T(t) \fn q (t)  ,  \f \varphi(t)
\right )}
\to
\intte{\left ( \nabla\fd d^T(t) \fd q (t)  ,  \f \varphi(t)
\right )}= \intte{\left ( {\f T}_{\delta}^E  ; \nabla \f \varphi(t)
\right )} \,.
\end{equation*}
It is essential that calculation~\eqref{Erikseniden} is applied in the limit, since it does not hold for the approximate analogues. 
With respect to the term incorporating the Leslie tensor, we only focus on the first term that is the least regular one. With \eqref{w:v} and \eqref{s:d}, we find that
\begin{multline*}
\intte{\left ( (\fn d(t) \cdot  (\nabla \fn v(t))_{\sym}
\fn d(t) )\fn d(t) \otimes \fn d(t)
; \nabla \f \varphi(t)
\right )}
\\
\to
\intte{\left ( (\fd d(t) \cdot  (\nabla \fd v(t))_{\sym}
\fd d(t) )\fd d(t) \otimes \fd d(t)
; \nabla \f \varphi(t)
\right )}
\,.
\end{multline*}
This, together with similar observations for the other terms,
shows that
\begin{equation*}
\intte{({\f T}_{n,\delta}^L (t): \nabla \f \varphi(t) ) }
\to \intte{(\tilde{\f T}_{\delta}^L (t): \nabla \f \varphi(t) ) } \, ,
\end{equation*}
where $\tilde{\f T}_{\delta}^L$ is given by
\begin{align*}
\begin{split}
\tilde{\f T}_{\delta}^L :={}&  \mu_1 (\fd d \cdot  \syd v \fd d )\fd d \otimes \fd d+\mu_4 \syd v  - (\mu_2+\mu_3) \left (\fd d \otimes \fd q     \right )_{\sym} \\&
- \left (\fd d \otimes \fd q   \right)_{\skw} + ((\mu_5+\mu_6) -\lambda(\mu_2+\mu_3))  \left ( \fd d \otimes \syd v \fd d \right)_{\sym}\, .
\end{split}
\end{align*}
Due to~\eqref{gleiche}, $\tilde{\f T}_{\delta}^L$ is equivalent to $\f T^L_{\delta}$ defined analogously to \eqref{Leslie} by
\begin{align}
\begin{split}
\f T^L_{\delta} ={}&  \mu_1 (\fd d \cdot \syd{v}\fd d )\fd d \otimes \fd d +\mu_4 \syd{v}
 + {(\mu_5+\mu_6)} \left (  \fd d \otimes\syd{v}\fd d \right )_{\sym}
\\
& +{(\mu_2+\mu_3)} \left (\fd d \otimes \fd e  \right )_{\sym}
 +\lambda \left ( \fd d \otimes \sy{v}\fd d  \right )_{\skw} + \left (\fd d \otimes \fd e  \right )_{\skw}\, 
\end{split}\label{Lesliedel}
\end{align}
with
\begin{align}
\fd e := \t \fd d + ( \fd v \cdot \nabla ) \fd d - \skd v \fd d \,.\label{edel}
\end{align}
This proofs Theorem~\ref{thm:weak}.

\end{proof}
\section{Convergence for vanishing regularization\label{sec:convmeas}}
\subsection{A priori estimates independent of the regularization}
%%%%%%%%%%%%%%%%%%%%%%%%%%%%%%%%%%%
%%%%%%%%%%%%%%%%%%%%%%%%%%%%%%%%%%%%%

The next lemma is a coercivity estimate for the free energy.
\begin{proposition}[Coercivity I]
\label{coerc1}
Let $\f d \in \Hc $. Then the following holds:
\begin{align}
\| \f d\|_{\He}^2 \leq \int_\Omega ( ( \di \f d )^2 + | \curl \f d |^2 )   \de \f x + c \| \f d \|_{\Hrand{3}}^2 \, \label{nad}
\end{align} 
and
\begin{align*}
\int_\Omega | \f d |^2 |\nabla \f d|^2 \de \f x \leq 2  \int_{\Omega} (| \f d |^2 ( \di \f d )^2 +  ( \f d\cdot \curl \f d )^2 + | \f d \times \curl \f d |^2)\de \f x + c\| \f d\|_{\Hrand{3}}^4 \, .
\end{align*}
\end{proposition}
\begin{proof}
The following equality can be shown by means of  simple vector calculus,
\begin{align}
| \nabla \f d |^2 = ( \di \f d )^2 + | \curl \f d |^2 + \tr( \nabla \f d ^2 ) - ( \di \f d )^2 \, .\label{iden}
\end{align}
The last two terms can be written as the divergence of a vector field
\begin{align}
\tr( \nabla \f d ^2 ) - ( \di \f d )^2 = \di ( \nabla \f d \f d - (\di \f d) \f d) \label{div1}\,.
\end{align}
Integrating the identity~\eqref{iden} over $\Omega$, using Gau\ss{}' formula,
and estimating the boundary terms yields the desired estimate~\eqref{nad}.

Again, simple vector calculus shows that
\begin{align*}
| \f d |^2 |  \curl \f d| ^2 = ( \f d \cdot \curl \f d )^2 + | \f d \times \curl \f d |^2 \, .
\end{align*}
In the same way as in~\eqref{div1}, we calculate
\begin{align}
\begin{split}
\di ( ( \nabla \f d \f d - ( \di \f d ) \f d ) | \f d |^2) 
%&= ( \tr ( \nabla \f  d ^2 ) - ( \di \f d)^2 ) | \f d|^2  \\& \quad + 2 (( \nabla \f d \f d ) \cdot \nabla \f d ^T \f d - ( \di \f d) \f d \cdot \nabla \f d ^T \f d ) \\
& = ( \tr ( \nabla \f  d ^2 ) - ( \di \f d)^2 ) | \f d|^2   - | \nabla \f d \f d - \nabla \f d^T \f d |^2 + | \nabla \f d \f d |^2 + | \nabla \f d ^T \f d |^2 \\&\quad - ( \di \f d) \f d \cdot \nabla \f d  \f d- ( \di \f d) \f d \cdot \nabla \f d^T  \f d  \,.
\end{split}\label{Numm}
\end{align}
Another vector identity grants that
\begin{align*}
| \nabla \f d \f d - \nabla \f d^T \f d |^2   = 4  | ( \nabla \f d)_{\skw} \f d |^2  = | \f d \times \curl \f d |^2 \, .
\end{align*}
The term $ | \f d |^2 | \nabla \f d|^2 $ integrated over the domain can be transformed via
\eqref{iden} and~\eqref{Numm} to
\begin{align*}
\int_{\Omega} | \f d|^2 | \nabla \f d|^2 \de \f x & = \int_{\Omega} ( ( \di \f d)^2 | \f d|^2 + |\f d |^2| \curl \f d |^2 + (\tr ( \nabla \f d ^2 ) - ( \di \f d ) ^2 ) | \f d|^2 ) \de \f x \\
& = \int_{\Omega} ( ( \di \f d)^2 | \f d|^2 + ( \f d \cdot \curl \f d )^2 + | \f d \times \curl \f d |^2 \de \f x \\ 
& \quad + \int_{\Omega} \di ( ( \nabla \f d \f d - ( \di \f d ) \f d ) | \f d |^2)\de \f x \\
& \quad  +\int_{\Omega} | \f d \times \curl \f d|^2 - | \nabla \f d \f d|^2 - | \nabla \f d^T \f d  |^2   \de \f x\\ 
& \quad + \int_{\Omega} ( \di \f d) \f d \cdot \nabla \f d  \f d+ ( \di \f d) \f d \cdot \nabla \f d^T  \f d  \de \f x
%\\ 
%& \leq \int_{\Omega} ( ( \di \f d)^2 | \f d|^2 + ( \f d \cdot \curl \f d )^2 + 2| \f d \times \curl \f d |^2 \de \f x + c\|\f d \|_{\Hrand{3}}^4 \ \\ 
%& \quad + \int_{\Omega} - | \nabla \f d \f d|^2 - | \nabla \f d^T \f d  |^2   \de \f x\\ 
%& \quad +\int_{\Omega} \frac{1}{2}( \di \f d)^2  |\f d|^2   + |\nabla \f d  \f d|^2 + |\nabla \f d^T  \f d |^2  \de \f x\\ 
%& = \int_{\Omega} ( \frac{3}{2}( \di \f d)^2 | \f d|^2 + ( \f d \cdot \curl \f d )^2 + 2| \f d \times \curl \f d |^2 \de \f x  + c\|\f d \|_{\Hrand{3}}^4 
 \,.
\end{align*}
  Young's inequality, Gau\ss{}' formula  and appropriate estimates of the boundary terms show
  \begin{align*}
\int_{\Omega} | \f d|^2 | \nabla \f d|^2 \de \f x 
& \leq \int_{\Omega} ( ( \di \f d)^2 | \f d|^2 + ( \f d \cdot \curl \f d )^2 + 2| \f d \times \curl \f d |^2 \de \f x + c\|\f d \|_{\Hrand{3}}^4 \ \\ 
& \quad + \int_{\Omega} - | \nabla \f d \f d|^2 - | \nabla \f d^T \f d  |^2   \de \f x\\ 
& \quad +\int_{\Omega} \frac{1}{2}( \di \f d)^2  |\f d|^2   + |\nabla \f d  \f d|^2 + |\nabla \f d^T  \f d |^2  \de \f x\\ 
& = \int_{\Omega} ( \frac{3}{2}( \di \f d)^2 | \f d|^2 + ( \f d \cdot \curl \f d )^2 + 2| \f d \times \curl \f d |^2 \de \f x  + c\|\f d \|_{\Hrand{3}}^4  \,.
\end{align*}  
   Therewith, both asserted coercivity estimates are proven.
\end{proof}
%%%%%%%%%%%%%%%%%%%%%%%%%%%%%%%%%%%%%%%
%%%%%%%%%%%%%%%%%%%%%%%%%%%%%%%%%%%%%%%%%
%

\begin{corollary}[A priori estimates]
There is a constant $C>0$,
depending on the initial values $\f v_0 , \, \f d_0$
and right-hand side $\f g$, such that for all $\delta\in(0,1)$ the constructed weak solution of the regularized system $\{( \fd v, \fd d)\}$ fulfills the estimate
\begin{align}
\begin{split}
&\| \fd v  \|_{L^\infty(\f L^2)}^2 +  {\delta} \| \Lap \fd d\|_{L^\infty(\Le)}^2+
\| \fd d \|_{L^\infty(\Hb)}^2 + \sup_{t\in[0,T]} \intet{|\fd d(t)|^2|\nabla \fd d(t)|^2}\\
& \quad  +\left \Vert \fd d\cdot \syd v \fd d \right \Vert_{L^2(L^2)}^2
 +  \|\fd v \|_{L^2(\V)}^2+  \|\syd v\fd d\|_{L^2(\Le)}^2
\\& \quad+   \|{\f q_{\delta}}\|_{L^2(\Le)}^2  
% + \delta^2 \|\Lapp\fd d \|_{L^2(\Le)} ^2 + \delta \| \nabla \Lap \fd d\|_{L^2(\Le)}^2 +\delta \| \Lap ( ( \di \fd d ) \fn d) \|_{L^2(\Le)}^2\\ & \quad  + \delta \| \Lap ( \fd d \cdot \curl \fd d ) \| _{ L^2(\Le)}^2 + \delta  \| \Lap ( \skd d \fd d) \|_{L^2(\Le)}^2+ \| \partial_t \fd v \|_{L^{2}((\f H^2 \cap \V)^*)} 
+  \| \partial_t \fd v \|_{L^{2}((\f H^2 \cap \V)^*)}  
 + \| \partial_t \fd d\|_{L^{2}(\f L^{\nicefrac{3}{2}})}
 \leq C \,.
\end{split}
\label{apri3}
\end{align}
\end{corollary}
\begin{proof}
This assertion is obvious by the a priori estimates~\eqref{entrodiss}, \eqref{esttime} and the weakly lower semi-continuity of the appearing norms.
Remark that the right-hand side of~\eqref{entrodiss} is bounded independently of $\varepsilon$, since $\f d_0$ is a unit vector a.e.~in $\Omega$ and $R_n\f d_0 $ converges strongly to $\f d_0$ in $\Hc$. 

 In regard of the time derivative of the director, we observe that the equation~\eqref{eq:dir} holds for all test functions. To estimate the time derivative, the projection $R_n$ and thus the restriction onto a Hilbert space as in Proposition~\ref{prop:time} is no longer needed. With the same argumentation as in Proposition~\ref{prop:time}, we get the asserted $L^2(\f L^{\nicefrac{3}{2}})$ bound.
\end{proof}

\subsection{Convergence of the solutions to the regularized systems}
The energy estimates of the previous corollary allow us to deduce the convergence of a subsequence of the solutions to the regularized system.
\begin{proposition}
Out of the family of solutions $(\fd v, \fd d)$ to the regularized systems~\eqref{weak}, we can extract a (not relabled) subsequence such that 
\begin{subequations}\label{wkonvreg}
\begin{align}
   \fd v &\stackrel{*}{\rightharpoonup}  \f v &\quad& \text{ in } L^{\infty} (0,T;\Ha)\,,\label{wr:vstern}\\
 \fd v &\rightharpoonup  \f v &\quad& \text{ in }  L^{2} (0,T;\V)\,,\label{wr:v}\\
% \fd d &\stackrel{*}{\rightharpoonup}  \f d &\quad& \text{ in } L^{\infty} (0,T;\He)\,,\label{w:dstern}\\
\fd q &\rightharpoonup  \ov{\f q} &\quad& \text{ in }  L^{2} (0,T;\Le)\,,\label{wr:E}\\
\syd v \fd d &\rightharpoonup   \sy v \f d &\quad& \text{ in }  L^{2} (0,T;\Le)\,,\label{wr:Dd}\\
\fd d\cdot \syd v \fd d &\rightharpoonup  \f d \cdot \sy v \f d  &\quad& \text{ in }  L^{2} (0,T;L^2)\,,\label{wr:dDd}\\
\fd e &\rightharpoonup  {\f e} &\quad& \text{ in }  L^{2} (0,T;\Le)\,,\label{wr:e}\\
\t \fd v  &\rightharpoonup \t \f v  &\quad& \text{ in } L^2(0,T; (\Hc\cap \f H^1_{0,\sigma})^*)\, ,\label{rtimev}\\
\t\fd d &\rightharpoonup \t \f d &\quad& \text{ in } L^{2}(0,T; \f L^{\nicefrac{3}{2}}) \, ,\label{rtimed}
\\
\fd d &\stackrel{*}{\rightharpoonup}  \f d &\quad& \text{ in } L^{\infty} (0,T;\He)\,. \label{wr:ddstern}\\
\fd v &\ra \f v &\quad& \text{ in }L^p (0,T; \Ha)\text{ for any } p \in [1,\infty)\, ,\label{sr:v}\\
\fd d &\ra \f d &\quad& \text{ in } L^q ( 0,T; \f L^r )\text{ for any } q
 \in [1,\infty)\, ,  r \in [1,12)\, ,\label{sr:d}
\end{align}
\end{subequations}
for $\delta\ra 0$.
\end{proposition}
\begin{proof}
This assertion is similar to the one of Proposition~\ref{lem:limits} and thus, the proof is also similar.
The existence of the weakly and weakly$^*$ converging subsequences follows from the estimate~\eqref{apri3} and the Banach--Alaoglu theorem as well as the definition of the weak derivative. 
The term $\fd e$, defined in~\eqref{edel}, is bounded due to equation~\eqref{eq:dir} and a priori estimate~\eqref{apri3},
\begin{align*}
\|\fd e \|_{L^2(\Le)} = \| \t \fd d +( \fd v\cdot \nabla ) \fd v - \skd v \fd d \|_{L^2(\Le)} \leq | \lambda | \| \syd v \fd d \| _{ L^2( \Le)} \| \fd q \|_{L^2( \Le)}\, . 
\end{align*}
The weak convergence of this term to some $\ov{\f e} \in L^2(0,T;\Le) $ can again be deduced  by standard arguments. 
For $\fd v$, we make the same observations as in Proposition~\ref{lem:limits} resulting in the strong convergence~\eqref{sr:v}.
For $\fd d$, we have less regularity than before. We note that $\He$ is compactly embedded in $\f L^r$ for $r<6$, which implies strong convergence in $L^q(0,T;\f L^r)$ for any $q\in[1,\infty)$ and $r\in [1,6)$.
Due to the boundedness in $\f L^{12}$, i.\,e.
\begin{align*}
\| \fd d\|_{L^\infty(\f L^{12})}^2 \leq \left  \| | \fd d |^2\right \|_{L^\infty(\f L^6)} \leq \left \| \nabla | \fd d |^2 \right \| _{L^\infty(\Le)} + \left \| | \fd   d | ^2 \right \| _{L^\infty(\Le)} \leq \left  \| \nabla \fd d | \fd d |\right \|_{L^\infty(\Le)} + \| \fd d \|_{L^\infty(\He)}^2\,,
\end{align*}
the strong convergence~\eqref{sr:d} holds due to a standard interpolation argument.
The limits in \eqref{wr:Dd},~\eqref{wr:dDd}, and~\eqref{wr:e} can be identified immediately due to the strong convergences~\eqref{sr:v} and~\eqref{sr:d}. 
\end{proof}

%%%%%%%%%%%%%%%%%%%%%%%%%%%%%%%%%
%% Masse kommen jetzt
%%%%%%%%%%%%%%%%%%%%%%%%%%%%%%%%%%
%%
%%
%%
%%
%%
%%
%%
%%
%%

%%%% Kommentar
%%%%%%%%%%%%%%%
Let $(\f v _{\delta_k}, \f d_{\delta_k})$ be a sequence of solutions  to the regularized system~\eqref{weak} for vanishing regularization, i.e.~$ \delta_k\ra 0$ for $k\ra \infty$.
Then we can identify the sequence of gradients of the directors $\nabla \f d_{\delta_k}(\f x, t) $ with an $(\f x,t)$ dependent family of probability measures~$ \f\delta_{\nabla \f d_{\delta_k}(\f x, t) } $ on the space of gradients of vector valued functions. Here the $\f \delta$ characterizes a point measure.
Instead of studying the weak limits of the functions $\nabla \f d_{\delta_k}(\f x, t) $, we can study the weak$^*$ limit of the probability distributions $\f \delta_{\nabla \f d_{\delta_k}(\f x, t)}$. The right sense for this turns out to be the generalized gradient Young measures introduced in Section~\ref{sec:young}.

%%%%%%%%%%%%%%%%%%%%%%%%%%%%%%%%%%%%%%%%
%% Konvergenz der variationaellen Ableitung
%%%%%%%%%%%%%%%%%%%%%%%%%%%%%%%%%%%%%%%%%
Since we want to go to the limit of the equation~\eqref{dir}, we have to take every term of 
equation~\eqref{eq:dir} in the cross product with the director.
Therefore, we are interested in the limit of the term $\fd d \times \fd q$.

\begin{proposition}\label{prop:qtilde}
The limit of $\{ \fd d \times \fd q\} $ is given by $\f d \times \f q$, where $\f d \times \f q$ can be expressed for every test function $\f \psi \in \C^\infty_c(\Omega \times (0,T))$ via
\begin{align*}
\int_0^T(\f d(t) \times \f q (t), \f \psi(t))\de t ={}& \int_0^T
\ll{\nu_t, \left (\f \Upsilon :\left (\f  S    (F_{\f S}(\f h, \f S))^T\right )+ \f h \times F_{\f h}(\f h ,\f S)\right ) \cdot\f \psi(t)   }
 \de t \\
&+ \int_0^T \left (
\rot{\f d(t)}  F_{\f S}( \f d(t) , \nabla \f d(t) ); \nabla \f \psi(t) 
\right ) \de t \, . 
\end{align*}
\end{proposition}
\begin{proof}
We already established the weak convergence~\eqref{wr:E}. It remains to identify the limit of $\fd d \times \fd q$.
First we observe that $\fd d \times ( | \fd d |^2 -1 ) \fd d = 0$ and the term due to the penalization, the last term in~\eqref{qtilde}, vanishes. 

Recalling the definition of $\fd q$ (see~\eqref{qtilde}), we find with an integration by parts for every $\f \psi \in \C^\infty_c(\Omega\times (0,T))$
\begin{align*}
\intte{\langle \rot{\fd d(t)}  \fd q (t), \f \psi(t) \rangle } ={}&  \delta  \intte{( \fd d(t)\times  \Lapp \fd d(t) ,  \f \psi(t) ) }
\\& + \intte{\left \langle \fd d (t) \times F_{\f h } ( \fd d(t), \nabla \fd d (t)) - \fd d (t) \times \di F_{\f S}( \fd d (t) , \nabla \fd d (t)) , \f \psi (t)  \right \rangle  }
\\
 ={}& 
 \delta  \intte{(  \Lap \fd d (t)\times \Lap \fd d(t) ,   \f \psi(t) )} + 2\delta \intte{
\left (  \Lap \fd d(t) , \nabla \rot{\fd d(t)}^T : \nabla \f \psi(t)   \right )
 }
\\&
+ \delta \intte{\left ( \fd d (t) \times  \Lap \fd d(t) , \Lap \f \psi(t) \right ) 
}
+ \intte{\left (  \fd d(t) \times F_{\f h } ( \fd d(t), \nabla \fd d (t)), \f \psi (t)  \right )}\\&
+ \intte{\left ( \f \Upsilon :\left (\nabla \fd d(t) \cdot   (F_{\f S}(\fd d(t), \nabla \fd d(t) )\right )^T ,\f \psi (t)  \right )} \\& +\intte{\left ( \rot{\f d(t)}  F_{\f S}( \fd d (t) , \nabla \fd d (t)) ; \nabla  \f \psi (t)  \right )  }\\
={}&  J_{1,\delta}+J_{2,\delta} + J_{3,\delta} + J_{4,\delta}+J_{5,\delta}+J_{6,\delta}\, .
\end{align*}
The first term vanishes, since it incorporates the cross product of two equal terms. The second and the third term can be estimated by
\begin{align*}
J_{2,\delta}   \leq \sqrt{\delta} c \left (\sqrt{\delta} \| \Lap \fd d \|_{L^\infty(\Le)} \| \nabla \fd d \|_{L^\infty(\Le)} \| \nabla \f \psi \|_{L^2(\f L^\infty)}\right )
\intertext{and}
J_{3,\delta} \leq  \sqrt{\delta} c \left (  \sqrt{\delta} \| \Lap \fd d \|_{L^\infty(\Le)} \| \fd d \|_{L^\infty(\f L^6)}  \| \Lap \f \psi \|_{L^2(\f L^3)}\right )\,,
\end{align*}
respectively. Remark that $ \delta \| \Lap \fd d \|_{L^\infty(\Le)}^2$ is bounded. 
The terms thus converge to zero for $\delta\ra 0$. 
The terms $J_{4,\delta}$ and $J_{5,\delta}$ converge in regard of Theorem~\ref{thm:young}
\begin{multline*}
 \intte{\left (\left \langle  \fd d(t) \times F_{\f h } ( \fd d(t), \nabla \fd d (t)), \f \psi (t)  \right \rangle+ \left \langle \f \Upsilon :\left (\nabla \fd d(t) \cdot   (F_{\f S}(\fd d(t), \nabla \fd d(t) )\right )^T ,\f \psi (t)  \right \rangle \right ) } 
\\
\longrightarrow
\int_0^T \left (\ll{\nu_t, \f \Upsilon :\left (\f  S \cdot   (F_{\f S}(\f h, \f S))^T\right ) \cdot\f \psi(t)   }+ 
 \ll{ \nu_t , \f h \times F_{\f h}(\f h, \f S) \cdot \f \psi(t) } \right ) \de t\,.
\end{multline*}

Finally, the term $J_{6,\delta}$ converges weakly due to~\eqref{wr:ddstern} and~\eqref{sr:d} and since the gradient of the director occurs only linearly (see definition~\eqref{FSFh}),
\begin{align*}
\intte{\left ( \rot{\fd d(t)}  F_{\f S}( \fd d (t) , \nabla \fd d (t)) ; \nabla  \f \psi (t) \right ) } \ra \intte{\left (  \rot{\f d (t )}  F_{\f S}( \f d (t) , \nabla \f d (t)) ; \nabla  \f \psi (t) \right ) }\, .
\end{align*}

\end{proof}
\begin{proposition}\label{prop:Erik}
The Ericksen stress $\f T^E_\delta $ converges in the following sense:
\begin{align*}
\int_0^T \left ( \f T^E _{\delta }(t) ; \nabla\f  \varphi(t)\right ) \de t \longrightarrow \int_0^T \left (2 \ll{\mu_t , \f \Gamma  \dreidots (\f \Gamma \cdot \nabla\f  \varphi(t))   } + \ll{ \nu_t , \f S ^T F_{\f S}( \f h , \f S) : \nabla\f  \varphi(t)  } \right ) \de t \,
%\f T^E _ \delta( \f x ,t) \stackrel{*}{\rightharpoonup} \f T^E( \f x ,t) = \langle G( \f d( \f x ,t), \cdot ) , \nu_{( \f x ,t)}\rangle+ \langle   G , \nu_{( \f x ,t)}^\infty \rangle m_t\, , \quad \text{ in  } \M(\Omega) \quad \text{for a.e.~}t\in[0,T]\, 
\end{align*}
for $\delta\ra 0$ and  for all $\f \varphi \in \C_c^\infty( \Omega \times (0,T))$ with $\di \f \varphi=0$. 
\end{proposition}
\begin{proof}
Recall the definition of the Ericksen stress~\eqref{Erikreg}.  An integration by parts in the second term yields for every $\f \varphi\in L^2(0,T;\C^\infty_{0,\sigma}(\Omega) )$
\begin{align*}
\intte{ \left ( {\f T}^E_{\delta  }; \nabla \f \varphi \right ) } &= \intte{\left (\nabla \fd d ^T F_{\f S}(\fd d , \nabla \fd d )  + \delta \Lap \fd d\cdot \nabla ^2 \fd d   - \delta \nabla \fd d ^T  \nabla \Lap \fd d ; \nabla \f \varphi \right ) } \\
& = \intte{(\nabla \fd d ^T F_{\f S}(\fd d , \nabla \fd d ); \nabla \f \varphi  )  } + 2\delta \intte{( \Lap \fd d \cdot \nabla ^2 \fd d  ; \nabla \f \varphi  )  }\\& \quad  
%- \delta \intte{( \nabla \fd d^T \cdot \f \Lambda : \nabla \Lap \fd d ; \nabla \f \varphi  )  } 
- \delta \intte{(  \nabla \fd d ^T \Lap \fd d   ,\Lap  \f \varphi  )  }= K_{1,\delta}+ K_{2,\delta}+K_{3,\delta}\, .
\end{align*}
Regarding the term $K_{1,\delta}$, we can go to the limit due to Proposition~\ref{lem:meas},
\begin{align*}
K_{1,\delta} = \intte{( \nabla \fd d ^T F_{\f S} (\fd d , \nabla \fd d ); \nabla \f \varphi  )  } \ra  \intte{\ll{\nu_t, \f S ^T F_{\f S} ( \f h ,\f S)  :\nabla \f \varphi(t) }  }\, .
\end{align*}
For the term $K_{2,\delta}$, we get after two integrations by parts
\begin{align*}
\frac{1}{2}K_{2,\delta} ={}& \intte{( \Lap \fd d \cdot \nabla^2 \fd d   ; \nabla \f \varphi  )  } \\
={}& -\intte{ \left ( \nabla \fd  d: \nabla^3  \fd d   ; \nabla \f \varphi\right ) + \left (  \nabla \f d^T \cdot \nabla^2\fd d  \dreidotkom \nabla ^2 \f \varphi \right )  
}\\
={}& \intte{\left ( \nabla ^2 \fd d  \dreidotkom \nabla ^2  \fd d \cdot \nabla \f \varphi \right ) }+\intte{ \left ( \nabla ^2\fd d : \nabla \fd d  ,  \nabla (\di  \f \varphi) \right )  
}
%\\
%& 
- \intte{  \left (  \nabla \f d^T \cdot \nabla^2\fd d  \dreidotkom \nabla ^2 \f \varphi \right )    } 
\\
={}& L_{1,\delta} + L_{2,\delta} + L_{3,\delta} \,.
%
%- 2\delta \intte{( \nabla \fd d^T \cdot \f \Lambda : \nabla \Lap \fd d ; \nabla \f \varphi  )  } \leq \sqrt{\delta} c \| \fd d \|_{L^\infty(\Hb)} \sqrt{\delta}\| \nabla \Lap \fd d \|_{L^2(\Le)} \| \f \varphi \| _{L^2(\f W^{1,\infty})}\\&\leq  \sqrt{\delta} c\left (  {\delta}\| \nabla \Lap \fd d \|_{L^2(\Le)}^2 + \| \fd d \|_{L^\infty(\Hb)} ^2 \| \f \varphi \| _{L^2(\f W^{1,\infty})}^2\right )\, .
\end{align*}
For $L_{1,\delta}$ holds with Theorem~\ref{thm:defectmeas}
\begin{align*}
\intte{\left ( \nabla ^2 \fd d  \dreidotkom \nabla ^2 \fd d \cdot\nabla \f \varphi\right )} \ra \intte{ \ll{\mu_t , \f \Gamma\dreidots ( \f \Gamma \cdot \nabla \f \varphi (t)) } }\,.
\end{align*}

The term $L_{2,\delta}$ vanishes since $\f \varphi$ is divergence free.
Due to a priori estimate~\eqref{apri3}, the coercivity of the Laplace 
operator and the regularity of the test function, the remaining terms can be estimated by a constant times $\sqrt \delta$ and  go to zero for $\delta\ra 0$,
\begin{align*}
K_{3,\delta}+  L_{3,\delta}  \leq{}& c \delta (\| \Lap  \fd d  \|_{\Le} + \| \nabla ^2 \fd d\|_{\Le} )\| \nabla \fd d  \|_{\Le} \| \nabla^2 \f \varphi \|_{\f L^\infty}\\ \leq{}& \sqrt{\delta} (  \delta \| \Lap  \fd d  \|_{\Le}^2 + \| \f d \|_{\He}^2
)^{\nicefrac{1}{2}} \| \nabla \fd d  \|_{\Le} \| \nabla^2 \f \varphi \|_{\f L^\infty}  \ra 0 \,.
\end{align*}

\end{proof}
%%%%%%%%%%%%%%%%%%%%%%%%%%%%%%%%%%%%%%%%%%%%%%%
%%%%%%%%%%%%%%%%%%%%%%%%%%%%%%%%%%%%%%%%%%%%%%%
%%%%%%%%%%%%%%%%%%%%%%%%%%%%%%%%%%%%%%%%%%%%%%%

%%%%%%%%%%%%%%%%%%%%%%%%%%%%%%%%%%%%%%%%%%%
%%%%%%%%%%%%%%%%%%%%%%%%%%%%%%%%%%%%%%%%%
%%%%%%%%%%%%%%%%%%%%%%%%%%%%%%%%%%%%%%%

\begin{proof}[Proof of Theorem~\ref{thm:meas}]
It only remains to prove that the limit $(\f v, \f d)$ of Proposition~\ref{wkonvreg} satisfies the definition of a measure-valued solution of the system (see Definition~\ref{def:meas}). This is shown by passing to the limit in the regularized problem (see Definition~\ref{defi:weak}).

Let us start with the regularized director equation~\eqref{eq:dir}. We consider equation~\eqref{eq:dir} in the cross product with the director and get for the term incorporating the time derivative that it converges due to~\eqref{rtimed} and~\eqref{sr:d}. The semilinear terms 
converge weakly due to the strong convergence of $\fd v$ and $\fd d$ (see~\eqref{sr:v}, \eqref{sr:d}) and the weak convergence of its gradients (see~\eqref{wr:v}, \eqref{wr:ddstern}). Thus, ee obtain for all  $\f \psi  \in \mathcal{C}_c^\infty( \Omega\times(0,T);\R^3)$
\begin{multline*}
\intte{\left ( \fd d  \times \left ( \partial_t \fd d  + (\fd v \cdot \nabla ) \fd d  - \left (\left ( \nabla \fd v \right) _{\skw}  - \lambda  \left ( \nabla \fd v \right) _{\sym} \right ) \fd d \right ), \f \psi \right )}\\ \longrightarrow \intte{\left ( \f d \times \left ( \t \f d  + (\f v \cdot \nabla ) \f d  -  \left (\nabla  \f v \right) _{\skw}  \f d  + \lambda  \left ( \nabla\f v \right) _{\sym}  \f d \right ), \f \psi  \right )}
\, ,
\end{multline*}
where we omitted the time dependence for brevity.
We observe the convergence of the term $ \fd q$ due to~\eqref{wr:E} and Proposition~\ref{prop:qtilde}. 
Since all terms of the regularized  director equation converge, we can go to the limit in equation~\eqref{eq:dir}  and attain the measure-valued formulation~\eqref{eq:mdir}.

The next step is to go to the limit in the fluid-flow equation. We already established the convergence of the time derivative in~\eqref{rtimev}. The convection term converges due to the strong convergence of the velocity fields~\eqref{s:v} and the  weak convergence of its gradients~\eqref{w:v}, such that we have for all solenoidal $\f \varphi \in \mathcal{C}_c^\infty( \Omega\times(0,T);\R^3))$ 
\begin{equation*}
\int_0^T ((\fd v \cdot \nabla) \fd v  , \f \varphi ) \de t
\to
\int_0^T ((\f v \cdot \nabla) \f v , \f \varphi ) \de t \, .
\end{equation*}
With the strong convergence of the director~(see~\eqref{sr:d}) and the weak convergences~\eqref{wr:dDd}, \eqref{wr:v}, \eqref{wr:e}, and~\eqref{wr:Dd}, we get the convergence of the Leslie stress, i.e.
\begin{align}
\begin{split}
\intte{\left (  \mu_1 (\fd d \cdot \syd{v}\fd d )\fd d \otimes \fd d +\mu_4 \syd{v}
 + {(\mu_5+\mu_6)} \left (  \fd d \otimes\syd{v}\fd d \right )_{\sym}; \nabla \f\varphi\right ) }
\\
 +
\intte{\left (
{(\mu_2+\mu_3)} \left (\fd d \otimes \fd e  \right )_{\sym}
 +\lambda \left ( \fd d \otimes \sy{v}\fd d  \right )_{\skw} + \left (\fd d \otimes \fd e  \right )_{\skw}; \nabla \f \varphi \right )}\rightarrow\\
\intte{\left (  \mu_1 (\f d \cdot \sy{v}\f d )\f d \otimes \f d +\mu_4 \sy{v}
 + {(\mu_5+\mu_6)} \left (  \f d \otimes\sy{v}\f d \right )_{\sym}; \nabla \f \varphi\right )}
\\
 +\intte{ \left ({(\mu_2+\mu_3)} \left (\f d \otimes \f e  \right )_{\sym}
 +\lambda \left ( \f d \otimes \sy{v}\f d  \right )_{\skw} + \left (\f d \otimes \f e  \right )_{\skw}; \nabla \f \varphi \right ) }\, .\end{split}
\end{align}
The convergence of the Ericksen stress~$\f T^E$ was already established in Proposition~\eqref{prop:Erik}.
This shows that the limit~$(\f v, \f d)$ of solutions $\{(\fd v , \fd d)\}$ to the regularized system~\eqref{weak} for vanishing regularization satisfies the system~\eqref{meas}. 

The solution~$(\f v, \f d )$ already satisfies the initial values $\f v(0)= \f v_0$ and $\f d(0)=\f d_0$ due to corollary~\eqref{cor:initial}.

\end{proof}

\section{Additional properties of the measure-valued solutions~\label{sec:add}}
\subsection{Additional estimates}
%%%%%%%%%%%%%%%%%%%%%%%%%%%%%%
%%%%%%%%%%%%%%%%%%%%%%%%%%%%%%
%%%%%%%%%%%%%%%%%%%%%%%%%%%%%
This section is devoted to the proof of an additional estimate for the system, i.\,e.~an $\f L^\infty$-estimate in space for the director.  
Later on, this allows to characterize the support of the defect angle measure $\nu^\infty$ and additionally, to give a remark concerning the existence theory despite the lack of coercivity.

\begin{proposition}\label{prop:apri2an}
Let the assumptions of Theorem~\ref{thm:meas} be fulfilled with the additional assumption on the constants appearing in the Oseen--Frank energy $k:=k_1=k_2$. Let additionally be $\varepsilon = \delta^{\nicefrac{7}{3}}$. For the solutions to the approximate regularized system, we find
\begin{align*}
\left \| |\fn d |^2 -1\right \|_{L^{\nicefrac{8}{3}}(L^\infty)} + \left \| \nabla | \fn d |^2 \right \|_{L^{\nicefrac{8}{3}}(L^3)} \leq  c \delta^{\nicefrac{1}{3}}	\,.
\end{align*}
\end{proposition}
\begin{proof}
To prove this identity, we investigate the variational derivative $\fn q$.
Recall the identity $$ \Delta \fn d = \nabla\di \fn d - \curl\curl \fn d\,. $$
\begin{rem}
The result also holds for $k_1\neq k_2$, but then the proof gets more technical.
\end{rem}
The Definition of $\fn q$~\eqref{qn} gives
\begin{align}
\begin{split}
\| {\f q}_{n,\delta} \|_{\Le}^2 ={}& \delta ^2 \|  \Delta^2 \fn d \|_{\Le}^2  
\\
&+ 2\delta\left (\Delta^2\fn d ,R_n\left ( F_{\f h}( \fn d , \nabla \fn d) - \di F_{ \f S} ( \fn d , \nabla \fn d)+ \frac{1}{\varepsilon } ( |\fn d|^2 -1 ) \fn d  \right ) \right )\\
& +  \left \| R_n\left (F_{\f h}( \fn d , \nabla \fn d) - \di F_{ \f S} ( \fn d , \nabla \fn d)+ \frac{1}{\varepsilon } ( |\fn d|^2 -1 ) \fn d \right )\right \|_{\Le}^2 
\, .
\end{split}
\label{normqt}
\end{align}
We consider the second term on the right-hand side of~\eqref{normqt} further on. The projection $R_n$ can be ignored since $\Delta^2 \fn d \in Z_n$.
The definition of the variational derivative now gives
\begin{align}
\begin{split}
 ( &\f q_{n,\delta} , \Delta^2 \fn d ) \\&= \frac{\delta}{4} \| \Delta\fn d  \|_{\Le}^2+k  ( \Delta^2 \fn d, - \Delta\fn d  ) \\& \quad + k_3 \left ( ( \Delta^2  \fn d  , -\nabla  (( \di \fn d) |\fn d|^2) ) + ( \Delta^2  \fn d , \fn d ( \di \fn d)^2 ) \right )\\
 & \quad + 
 k_4 \left (( \Delta^2\fn d ,-\di ( \rot{\fn d  } ( \fn d \cdot \curl \fn d ) ) + ( \Delta^2 \fn d , \curl \fn d ( \fn d \cdot \curl \fn d)) \right )\\& \quad 
 + 4 k_5 ( ( \Delta^2 \fn d , -\di  ( \skn {d} \fn d \otimes \fn d )_{\skw}) + ( \Delta^2 \fn d, \skn d ^T \skn d \fn d ) ) \\
 & \quad + \frac{1}{\varepsilon } (  \Delta^2 \fn d , ( | \fn d|^2 -1 ) \fn d )  
 \\ & =  I_1+k I_2  + k_3 I_3 + k_4 I_4 + 4 k_5 I_5 + \frac{1}{\varepsilon} I_6 \,.
 \end{split}\label{prev}
\end{align}
The appearing terms are going to be estimated individually. Since $\Delta \Sr \f d_1 =0$, the definition of~\eqref{dar} grants that $\f \gamma_0(\Delta \fn d ) \equiv 0 $. Hence, the boundary terms vanish in the following integration by parts  
\begin{align*}
k  I_2= 
-k( \Delta^2 \fn d , \Delta\fn d ) =  k\| \nabla \Delta \fn d\|_{\Le}^2 \, .
\end{align*}
For the upcoming integration by parts, we transform the functions~$\fn d$ onto homogeneous Dirichlet boundary conditions. Due to definition~\eqref{dar}, $\fn d$ can be transformed via  $\fnt d:= \fn d - \Sr \f d_1$, where $\fnt d$ takes values in $Z_n$. The terms $I_3$, $I_4$ and $I_5$ in~\eqref{prev} can be written as 
\begin{align*}
k_3 I_3 + k_4 I_4 + 4 k_5 I_5  =   \left ( \Delta^2 \ftn d , - \di \left ( \fn d \cdot \f \Theta \dreidots \nabla \fn d \otimes \fn d \right ) + \nabla \fn d : \f \Theta \dreidots \nabla \fn d \otimes \fn d \right )  \,. 
\end{align*} 
With some vector calculus, we see 
\begin{align}
\begin{split}
   \delta &\left ( \Delta^2 \ftn d , - \di \left ( \fn d \cdot \f \Theta \dreidots \nabla \fn d \otimes \fn d \right ) + \nabla \fn d : \f \Theta \dreidots \nabla \fn d \otimes \fn d \right )   \\
={}&  \delta\left ( \Delta^2 \ftn d , - \di \left ( \ftn d \cdot \f \Theta \dreidots \nabla \ftn d \otimes \ftn d \right ) + \nabla \ftn d : \f \Theta \dreidots \nabla \ftn d \otimes \ftn d \right )  \\
  &  +\delta \left ( \Delta^2 \ftn d , - \di \left ( \Sr \f d_1 \cdot \f \Theta \dreidots \nabla  \fn d \otimes \fn d \right ) + \nabla  \Sr \f d_1 : \f \Theta \dreidots \nabla \fn d \otimes \fn d \right )  \\
  &  +\delta \left( \Delta^2 \ftn d , - \di \left (  \ftn d \cdot \f \Theta \dreidots \nabla \Sr \f d_1 \otimes \fn d \right ) + \nabla \ftn d : \f \Theta \dreidots \nabla \Sr \f d_1 \otimes \fn d\right )  \\
  &  +\delta \left( \Delta^2 \ftn d , - \di \left (  \ftn d \cdot \f \Theta \dreidots \nabla\ftn d \otimes \Sr \f d_1  \right ) + \nabla \ftn d : \f \Theta \dreidots \nabla \ftn d  \otimes \Sr \f d_1\right ) \,,
  \end{split}\label{umformu}
\end{align}  
which can be estimated by the Gagliardo--Nirenberg and Young inequality,
 \begin{align}
\begin{split}  
  k_3 I_3 +& k_4 I_4 + 4 k_5 I_5 \\
    \geq {}& \delta \left ( \Delta^2 \ftn d , - \di \left ( \ftn d \cdot \f \Theta \dreidots \nabla \ftn d \otimes \ftn d \right ) + \nabla \ftn d : \f \Theta \dreidots \nabla \ftn d \otimes \ftn d \right ) \\ 
  & - \delta\| \Delta^2\fn d \| _{\Le} \| \Sr \f d _1 \|_{\f W^{1,\infty }} \| \fn d \|_{\f W^{1,4}} \| \fn d \|_{\f L^4}+
  \\
  & - \delta\| \Delta^2\fn d \| _{\Le}
 \| \Sr \f d_1 \|_{\f L^\infty} \left (\| \fn d \|_{\Hc
 }\| \fn d\|_{\f L^\infty} + \|\fn d \|_{\f W^{1,4}}^2\right )
\\
&- \delta\| \Delta^2 \fn d\|_{\Le}    \| \ftn d \|_{\f W^{1,4}} \| \Sr \f d_1 \|_{\f W^{1,\infty}} \| \fn d \|_{\f L^4}\\& 
- \delta\| \Delta^2 \fn d\|_{\Le}    
 \|\ftn d  \|_{\f L^4} (\| \Sr \f d_1 \|_{\f W^{2,\infty}}\| \fn d\|_{\f L^4} + \|\Sr \f d_1  \|_{\f W^{1,\infty}}\|\fn d \|_{\f W^{1,4}}) \\
&-\delta \| \Delta^2 \fn d\|_{\Le}  \| \ftn d \|_{\f W^{1,4}} Therefore, w^2  \| \Sr \f d_1  \|_{\f L^\infty}
\\
&-\delta \| \Delta^2 \fn d\|_{\Le}
 \|\ftn d  \|_{\f L^\infty} (\| \ftn d \|_{\Hc}\| \Sr \f d_1\|_{\f L^\infty} + \|\ftn d   \|_{\He}\|\Sr \f d_1 \|_{\f W^{1,\infty}}) 
   \\
  \geq {}&\delta\left ( \Delta^2 \ftn d , - \di \left ( \ftn d \cdot \f \Theta \dreidots \nabla \ftn d \otimes \ftn d \right ) + \nabla \ftn d : \f \Theta \dreidots \nabla \ftn d \otimes \ftn d \right ) - \frac{\delta^2}{4} \| \Delta^2 \fn d\|_{\Le}^2\\ 
    & - c \| \Sr \f d_1 \|_{\Hg}^2\left (\| \fn d \|_{\Hc}^2\| \fn d \|_{\f L^\infty}^2 + \| \fn d \|_{\f W^{1,4}}^4    +\| \ftn d \|_{\Hc}^2\| \ftn d \|_{\f L^\infty}^2 + \| \ftn d  \|_{\f W^{1,4}}^4    +1  \right )   \\
   \geq {}&\delta\left ( \Delta^2 \ftn d , - \di \left ( \ftn d \cdot \f \Theta \dreidots \nabla \ftn d \otimes \ftn d \right ) + \nabla \ftn d : \f \Theta \dreidots \nabla \ftn d \otimes \ftn d \right ) \\ 
    &- \frac{\delta^2}{4} \| \Delta^2 \fn d\|_{\Le}^2 - c \|  \f d_1 \|_{\Hrand{7}}^2\left (\| \fn d \|_{\Hc}^{\nicefrac{8}{3}}\| \fn d \|_{\f L^{12}}^{\nicefrac{4}{3}} + \| \ftn d \|_{\Hc}^{\nicefrac{8}{3}}\| \ftn d \|_{\f L^{12}}^{\nicefrac{4}{3}}  +1  \right )\,.
    \end{split}\label{abschaetz}
      \end{align}
      It should be recognized that the norms of the transformed variable ${\ftn d}$ can still be estimated by the original variable  $\fn d$
      \begin{align*}
\|{\ftn d} \|_{\Hc} \leq \| {\fn d }\| _{ \Hc} +\| \Sr \f d_1 \|_{\Hc} \leq \| \fn d \|_{\Hc} + c \| \f d _1\|_{\Hrand{3}} \,.
\end{align*}

      In the following, the Laplace operator is going to be applied to the mixed terms. Therefore, we recall the product rule for the Laplace operator 
      $$ \Delta ( \f a \cdot\f b ) = \Delta \f a\cdot  \f b + 2 \nabla \f a : \nabla \f b + \f a\cdot  \Delta \f b\,  \quad \text{for all } \f a, \f b \in \C^1(\ov \Omega;\R^3).$$
We are going to perform the appropriate estimates for the term $I_4$ in detail, the other terms are bounded analogously. 
An integration by parts shows
\begin{align}
\begin{split}
&\left  ( \Delta^2 \fnt d , -\di (\rot{\fnt d} ( \fnt d \cdot \curl \fnt d ))\right ) + \left ( \Delta^2 \fnt d , \curl \fnt d ( \fnt d \cdot \curl \fnt d )  \right ) \\
%={}&\left  (\nabla  \Delta \fnt d , \nabla \di  (\rot{\fnt d} ( \fnt d \cdot \curl \fnt d ))\right ) + \left ( \nabla \Delta \fnt d , -\nabla \left (\curl \fnt d ( \fnt d \cdot \curl \fnt d ) \right )\right ) 
%\\
%&- \left \langle \f \gamma_{\f n} ( \nabla  \Delta \fnt d) , \f \gamma_0 ( \di  (\rot{\fnt d} ( \fnt d \cdot \curl \fnt d )))\right \rangle_{\partial\Omega} + \left \langle \f \gamma_{\f n} ( \nabla  \Delta \fnt d) , \f \gamma_0 (\curl \fnt d ( \fnt d \cdot \curl \fnt d ) ) \right \rangle_{\partial\Omega}
\,.
\end{split}\label{rechdis}
\end{align}
The boundary terms vanish, since the transformed variable $\fnt d$ fulfils homogeneous Dirichlet boundary conditions. 
Another integration by parts shows
\begin{align*}
&\left  (\nabla  \Delta \fnt d ; \nabla \di  (\rot{\fnt d} ( \fnt d \cdot \curl \fnt d ))\right ) + \left ( \nabla \Delta \fnt d ; -\nabla \left (\curl \fnt d ( \fnt d \cdot \curl \fnt d ) \right )\right ) \\
&\quad={} \left  (  \Delta \fnt d , -\di \Delta (\rot{\fnt d} ( \fnt d \cdot \curl \fnt d ))\right ) + \left (  \Delta \fnt d , \Delta \left (\curl \fnt d ( \fnt d \cdot \curl \fnt d ) \right )\right ) \\
&\quad ={} \left  (  \Delta (\nabla \fnt d ),  \Delta (\rot{\fnt d} ( \fnt d \cdot \curl \fnt d ))\right ) + \left (  \Delta \fnt d , \Delta \left (\curl \fnt d ( \fnt d \cdot \curl \fnt d ) \right )\right )\,.
\end{align*}
Here, the boundary terms vanish since $\f \gamma_0( \Delta\fnt d)=0$.
Using the product rule for the Laplace operator, we get
 \begin{align}
\begin{split}
\big  ( \Delta (\nabla \fnt d ),&  \Delta (\rot{\fnt d} ( \fnt d \cdot \curl \fnt d ))\big ) + \left ( \Delta \fnt d , \Delta(\curl \fnt d ( \fnt d \cdot \curl \fnt d ) )\right ) 
\\
={}& \left  (\Delta (\nabla \fnt d  ) : \rot{\fnt d }  + \curl \fnt d \cdot  \Delta \fnt d , \Delta ( \fnt d \cdot \curl \fnt d ) \right ) \\
&+ 2 \left ( (\nabla \Delta \fnt d)_{\skw}: \rot{\Delta \fnt d} ,  \fnt d \cdot \curl \fnt d \right ) 
\\&
+ 2 \left ( \Delta ( \nabla \fnt d)_{\skw}:\nabla\rot{  \fnt d}  +  (\nabla( \curl \fnt d))^T \Delta \fnt d , \nabla (\fnt d \cdot \curl \fnt d ) \right ) \\
={}& \left  \|\Delta (\fnt d \cdot \curl \fnt d  )\right \|_{\Le}^2 - 2 \left ( \nabla (\nabla \fnt d)_{\skw} \dreidots \nabla \rot{\fnt d}   , \Delta (\fnt d \cdot \curl \fnt d  ) \right ) \\
&+ 2 \left ( (\nabla \Delta \fnt d)_{\skw}: \rot{\Delta \fnt d} ,  \fnt d \cdot \curl \fnt d \right ) \\&+ 2 \left ( \Delta ( \nabla \fnt d)_{\skw}:\nabla\rot{  \fnt d}  +  (\nabla( \curl \fnt d))^T \Delta \fnt d , \nabla (\fnt d \cdot \curl \fnt d ) \right ) \,.
\end{split}\label{umformung}
\end{align}
The H\"older, Gagliardo--Nirenberg and Young inequality allow to estimate the non-positive terms on the right hand side of the previous estimate,
\begin{align*}
\begin{split}
 2  k_4\delta  &    \left ( \nabla (\nabla \fnt d)_{\skw} \dreidots \nabla \rot{\fnt d}   , \Delta (\fnt d \cdot \curl \fnt d  ) \right )  \\
&\leq c \delta  \| \fnt d\|_{\f W^{1,6}}\| \fnt d\|_{\f W^{2,6}} \| \Delta ( \fnt d \cdot \curl \fnt d)\|_{\f L^{\nicefrac{3}{2}}} 
 \\
 &\leq c \delta  \| \fnt d\|_{\Hc}^{\nicefrac{3}{2}} \| \fnt d\|_{\Hg}^{\nicefrac{1}{2}}\|\Delta (\fnt d \cdot \curl \fnt d)\|_{\Le } ^{\nicefrac{3}{4}}\|\fnt d \cdot \curl \fnt d\|_{\Le} ^{\nicefrac{1}{4}}  \\
% & \leq c\| \fnt d\|_{\Hc}^{\nicefrac{3}{2}}\|\fnt d\|_{\He}^{\nicefrac{1}{2}} \| \fnt d \|_{\Hf} ^{\nicefrac{3}{2}}\|\fnt d\| _{\Hc}^{\nicefrac{1}{2}} + \frac{k_4}{8}\|\Delta(\fnt d \cdot \curl \fnt d)\|_{\Le} ^2\\
% & \leq \frac{\delta^2}{32} \| \Delta^2 \fnt d\|_{\Le}^{2} + \frac{k_4\delta }{8}\|\Delta(\fnt d \cdot \curl \fnt d)\|_{\Le}^2+ c \delta^{\nicefrac{2}{3}} \| \fnt d \|_{\Hc} ^{\nicefrac{14}{3}}\,,
&\leq \frac{\delta^2}{32} \| \Delta^2 \fnt d \|_{\Le}^2 +\frac{k_4\delta }{8}\| \Delta(\fnt d \cdot \curl \fnt d)\|^2_{\Le} 
+  c \delta^{\nicefrac{1}{3}}\| \fnt d \|_{\Hc}^{4} \|\fnt d \cdot \curl \fnt d\|_{\Le}^{\nicefrac{2}{3}}\,,
 \end{split}%\label{abschk3}
 \\
\begin{split}
 2k_4\delta&\left ( (\nabla \Delta \fnt d)_{\skw}: \rot{\Delta \fnt d} ,  \fnt d \cdot \curl \fnt d \right )  \\
 &\leq c \delta\| \fnt d\|_{\Hf} \| \fnt d\|_{\Hc} \| \fnt d \cdot \curl \fnt d\|_{\f L^\infty}
 \\
 &\leq c\delta \| \fnt d \|_{\Hc}^{\nicefrac{3}{2}}  \|\fnt d\|_{\Hg}^{\nicefrac{1}{2}} \| \Delta (\fnt d \cdot \curl \fnt d)\|_{\Le } ^{\nicefrac{3}{4}}\|\fnt d \cdot \curl \fnt d\|_{\Le} ^{\nicefrac{1}{4}} 
\\ & 
% \leq \frac{1}{32} \|\nabla \Delta \fnt d \|_{\Le}^2 + c \| \fnt d \|_{\Hc}^2
% \|\Delta(\fnt d \cdot \curl \fnt d)\|_{\Le} ^{\nicefrac{3}{2}}\|\fnt d \cdot \curl \fnt d\|_{\Le} ^{\nicefrac{1}{2}} 
%\\ & 
\leq \frac{\delta^2}{32} \| \Delta^2 \fnt d \|_{\Le}^2 +\frac{k_4\delta }{8}\| \Delta(\fnt d \cdot \curl \fnt d)\|^2_{\Le} 
+  c \delta^{\nicefrac{1}{3}}\| \fnt d \|_{\Hc}^{4} \|\fnt d \cdot \curl \fnt d\|_{\Le}^{\nicefrac{2}{3}}\,,
\end{split}
\\
\begin{split}
 2k_4\delta&\left ( \Delta ( \nabla \fnt d)_{\skw}:\nabla\rot{  \fnt d} , \nabla (\fnt d \cdot \curl \fnt d ) \right ) \\
 &\leq c \delta   \| \fnt d\|_{\Hf}\| \fnt d \|_{\f W^{1,6}} \| \fnt d \cdot \curl \fnt d \|_{\f W^{1,3}}\\
%&\leq c \| \fnt d\|_{\Hf} \| \fnt d \|_{\f W^{1,4}}  \| \fnt d \cdot \curl \fnt d \|_{\f W^{1,4}}  \\
&\leq c  \delta \| \fnt d\|_{\Hg}^{\nicefrac{1}{2}}  \| \fnt d \|_{\Hc}^{\nicefrac{3}{2}}
\| \fnt d \cdot \curl \fnt d \|_{\Hc}^{\nicefrac{3}{4}} \| \fnt d \cdot \curl \fnt d \|_{\Le}^{\nicefrac{1}{4}} 
\\
&\leq \frac{\delta^2}{32}\| \Delta^2 \fnt d \|_{\Le}^2 + \frac{k_4 \delta }{8 }\|\Delta( \fnt d \cdot \curl \fnt d) \|_{\Le}^2  + c  \delta^{\nicefrac{1}{3}} \| \fnt d \|_{\Hc}^{4}  \| \fnt d \cdot \curl \fnt d\|_{\Le}^{\nicefrac{2}{3}}  \,,
\end{split}
\\
\begin{split}
2k_4 \delta & \left ( (\nabla( \curl \fnt d))^T  \Delta \fnt d , \nabla (\fnt d \cdot \curl \fnt d ) \right ) \\ 
  &\leq c\delta \| \fnt d \|_{\Hc} \| \fnt d\|_{\f W^{2,3}} \| \fnt d \cdot \curl \fnt d \|_{\f W^{1,3}}\\
%&\leq c \| \fnt d \|_{\Hc} \| \fnt d\|_{\f W^{2,4}} \| \fnt d \cdot \curl \fnt d \|_{\f W^{1,4}} \\
&\leq c \delta  \| \fnt d\|_{\Hc}^{\nicefrac{7}{4}}  \| \fnt d\|_{\Hg}^{\nicefrac{1}{4}}
\|( \fnt d \cdot \curl \fnt d) \|_{\Hc}^{\nicefrac{3}{4}} \| \fnt d \cdot \curl \fnt d \|_{\Le}^{\nicefrac{1}{4}} 
\\
&\leq \frac{\delta^2 }{32}\|  \Delta^2\fnt d \|_{\Le}^2 + \frac{k_4\delta }{8 }\|\Delta( \fnt d \cdot \curl \fnt d) \|_{\Le}^2 + c \delta^{\nicefrac{3}{4}} \| \fnt d \|_{\Hc}^{\nicefrac{7}{2}}  \| \fnt d \cdot \curl \fnt d\|_{\Le}^{\nicefrac{1}{2}} \,.
\end{split}
\end{align*}
Together, we get
\begin{align*}
k_4 \delta  &\left (\left  ( \Delta^2 \fnt d , -\di (\rot {\fnt d} ( \fnt d \cdot \curl \fnt d))\right ) + \left ( \Delta^2 \fnt d , \curl \fnt d ( \fnt d \cdot \curl \fnt d) \right )\right )\\ & \geq \frac{k_4\delta }{2} \|\Delta(  \fnt d \cdot \curl \fnt d) \|_{\Le}^2   - \frac{\delta^2}{8}\| \Delta^2 \fnt d \|_{\Le}^2 \\&\quad  - 
c  \delta^{\nicefrac{1}{3}} \| \fnt d \|_{\Hc}^{4}  \| \fnt d \cdot \curl \fnt d\|_{\Le}^{\nicefrac{2}{3}} -c \delta^{\nicefrac{3}{4}} \| \fnt d \|_{\Hc}^{\nicefrac{7}{2}}  \| \fnt d \cdot \curl \fnt d\|_{\Le}^{\nicefrac{1}{2}}\,.
\end{align*}

Similarly, we get for the terms $I_3$ and $I_5$
\begin{align*}
%&k_3\left (\left  ( \Delta^2 \fnt d , -\nabla ( \di \fnt d | \fnt d|^2)\right ) +  \left ( \Delta^2 \fnt d ,  (\di \fnt d)^2\fnt d\right ) \right )  \\
k_3 I_3 + 4 k_5 I_5
% &+ 4 k_5 ( ( \Delta^2 \fnt d , -\di  ( \skn {d} \fnt d \otimes \fn d )_{\skw}) + ( \Delta^2 \fnt d, \skn d ^T \sknt d \fnt d ) ) \\ 
 \geq{}& \frac{k_3 \delta }{2} \|\Delta((\di  \fnt d ) \fnt d)\|_{\Le}^2 + 2k_5 \delta \| \Delta(\skn d \fnt d)\|_{\Le}^2 - \frac{\delta^2 }{4}\| \Delta ^2 \fnt d \|_{\Le}^2\\  & 
 - 
c  \delta^{\nicefrac{1}{3}} \| \fnt d \|_{\Hc}^{4}  \left \|  F(\fnt d , \nabla \fnt d)\right \|_{\Le}^{\nicefrac{2}{3}} -c \delta^{\nicefrac{3}{4}} \| \fnt d \|_{\Hc}^{\nicefrac{7}{2}} \left  \|  F(\fnt d , \nabla \fnt d)\right \|_{\Le}^{\nicefrac{1}{2}}\,. 
% - c  \left (\| \fn d\|_{\He}^{8}+1)\left\| F(\fnt d , \nabla \fnt d)\right \|_{L^\infty(\f L^1)}
%% + \| \fn d\|_{\He}^{10}+\left\| F(\fnt d , \nabla \fnt d)\right \|_{\f L^1}
%+ \| \f d _1\|_{\Hrand{3}}^{10} +1 \right ) \, 
\end{align*}
Remark that the nonlinear terms can be transformed with similar calculations as in~\eqref{umformu} and estimates as in~\eqref{umformu} to estimates for the variable $\fn d$ with inhomogeneous boundary values.
Therefore, one has to employ as beforehand $\Lap \fn d = \Lap \fnt d $.

For the term $I_6$, there is no transformation onto homogeneous boundary values necessary since the given boundary data has norm one, i.e. $| \f d_1|=1 $ on $\partial \Omega$. Additionally, $\Delta \fn d = 0 $ on $\partial \Omega$ such that the boundary term of the following integration by parts vanishes
\begin{align*}
 (  \Delta^2 \fn d , ( | \fn d|^2 -1 ) \fn d ) 
 ={}& ( \Delta \fn d ,  \Delta( | \fn d|^2 -1 ) \fn d  ) + ( \Delta \fn d , \Delta \fn d ( | \fn d |^2 -1)) \\
 &+ 2 ( \Delta \fn d , \nabla \fn d \nabla (| \fn d|^2 -1)) \\
 ={}& \frac{1}{2 } \left \| \Delta(| \fn d |^2 -1) \right \| _ {\Le}^2  - ( |\nabla  \fn d|^2 ,  \Delta( | \fn d|^2 -1 )   ) 
\\& 
 + ( \Delta \fn d , \Delta \fn d ( | \fn d |^2 -1)) + 2 ( \Delta \fn d , \nabla \fn d \nabla (| \fn d|^2 -1))
 \,.
 \end{align*}
 Estimating again the right-hand side with  H\"older, Gagliardo--Nirenberg and Young inequality, we get
 \begin{align*}
 I_6 \geq{}&  \frac{1}{2}\left \| \Delta(| \fn d |^2 -1) \right \| _ {\Le}^2 - \| \nabla \fn d \|_{\f L^6}^2 \| \Delta ( | \fn d |^2-1)\|_{L^{\nicefrac{3}{2}}} - \| \Delta \fn d\|_{\Le}^2 \left \| |\fn d |^2-1\right \|_{L^\infty} \\&- 2 \| \Delta \fn d \|_{\Le} \|\nabla \fn d \|_{\f L^6} \left \| \nabla ( | \fn d|^2 -1) \right \|_{L^3} \\
 \geq {}& \frac{1}{2 } \left \| \Delta(| \fn d |^2 -1) \right \| _ {\Le}^2 - c
 \|  \fn d\|_{\Hc}^2 \left \| \Delta(|\fn d |^2-1)\right \|_{L^2 }^{\nicefrac{3}{4}} \left \| |\fn d |^2-1\right \|_{L^2}^{\nicefrac{1}{4} }
 \\
 \geq{}&\frac{1}{4 } \left \| \Delta(| \fn d |^2 -1) \right \| _ {\Le}^2 
- c \| \fn d \|_{\Hc}^{\nicefrac{16}{5}}  \left \| |\fn d |^2-1\right \|_{L^2}^{\nicefrac{2}{5} }
% -  c
% \|  \Delta^2 \fn d\|_{\Le}^{\nicefrac{48}{55}} \|   \fn d\|_{\f L^{12}}^{\nicefrac{128}{55}} \left \| |\fn d |^2-1\right \|_{L^2}^{\nicefrac{2}{5} }\,.
% 
% \| \nabla \fn d \|_{\f L^4}^4 -  \frac{1}{12 \varepsilon }\left \| \Delta ( | \fn d |^2-1)\right \|_{L^2}^2 \\&- \| \Delta \fn d\|_{\Le}^2 \left \| \Delta(|\fn d |^2-1)\right \|_{L^2 }^{\nicefrac{3}{4}} \left \| |\fn d |^2-1\right \|_{L^2}^{1/4} \\&- 2 \| \Delta \fn d \|_{\Le} \| \nabla \fn d \|_{\f L^4} \left \| \Delta ( | \fn d|^2 -1) \right \|_{L^2}^{3/4} \left \| |\fn d |^2-1\right \|_{L^2}^{2/4}
% \\
%  \geq {}& \frac{1}{4\varepsilon } \left \| \Delta(| \fn d |^2 -1) \right \| _ {\Le}^2 - c\|  \fn d \|_{\Hc}^3\|\fn d \|_{\He} -   c\|  \fn d\|_{\Hi}^{16/5}  \left \| |\fn d |^2-1\right \|_{L^2}^{2/5} 
%\,.
\end{align*}

 Together, we get the coercivity estimate
 \begin{align}
\begin{split}
\| { \f q}_{n}\|_{\Le}^2 &\geq \delta^2 \| \Delta^2 \fn d\|^2_{\Le} + \left \| R_n\left (F_{\f h}( \fn d , \nabla \fn d) - \di F_{ \f S} ( \fn d , \nabla \fn d)\right )\right \|_{\Le}^2 + \frac{\delta  }{2} \| \nabla \Delta \fn d \| _   {\Le}^2\\&\quad + \frac{\delta k_3 }{2} \| \Delta ( ( \di \fn d ) \fn d) \|_{\Le}^2  + \frac{\delta k_4 }{2 } \| \Delta ( \fn d \cdot \curl \fn d ) \| _{ \Le}^2\\ & \quad + \frac{\delta k_5 }{2} \| \Delta ( \skn d \fn d) \|_{\Le}^2 + \frac{\delta}{4\varepsilon}\| \Delta | \fn d|^2\|_{\Le}^2  \\&\quad
- c \|  \f d_1 \|_{\Hrand{7}}^2\left (\| \fn d \|_{\Hc}^{\nicefrac{8}{3}}\| \fn d \|_{\f L^{12}}^{\nicefrac{4}{3}}  +1  \right )- c \frac{ \delta}{\varepsilon} \left \| |\fn d |^2-1\right \|_{L^2}^{\nicefrac{22}{31} }\\
&\quad-c  \delta^{\nicefrac{1}{3}} \| \fn d \|_{\Hc}^{4}  \left \|  F(\fn d , \nabla \fn d)\right \|_{\Le}^{\nicefrac{2}{3}} -c \delta^{\nicefrac{3}{4}} \| \fn d \|_{\Hc}^{\nicefrac{7}{2}} \left  \|  F(\fn d , \nabla \fn d)\right \|_{\Le}^{\nicefrac{1}{2}}
\,.
%- c \left  ( \delta\| \fn d\|_{\He}^{8}+1\right )\left\| F(\fn d , \nabla \fn d)\right \|_{L^\infty(\f L^1)}
%%+ \| \fn d\|_{\He}^{10}+\left\| F(\fn d , \nabla \fn d)\right \|_{\f L^1}^5 
%%+ \left \| |\fn d|^2-1\right \|_{\Le}^{10}
%- \| \f d _1\|_{\Hrand{3}}^{10}+1 \\& \quad 
%-  c\frac{\delta}{\varepsilon}\|  \fn d \|_{\Hc}^3\|\fn d \|_{\He} -   c\frac{\delta}{\varepsilon}\|  \fn d\|_{\Hi}^{16/5}  \left \| |\fn d |^2-1\right \|_{L^2}^{2/5} \,.
\end{split}\label{coerc}
\end{align}

 This estimate~\eqref{coerc} reinserted in~\eqref{entrodiss} gives another a priori estimate,
 \begin{align*}
\begin{split}
&\frac{1}{2}\| \fn v \|_{L^\infty(\f L^2)}^2 +  \frac{\delta}{2} \| \Delta \fn d\|_{L^\infty(\Le)}^2+ \frac{k}{2}\| \nabla \fn d\|_{L^\infty(\Le)}^2+ \frac{k_3}2 \| (\di \fn d )\fn d\|_{L^\infty(\Le)}^2   \\ &+\frac{k_4}2  \| \fn d \cdot \curl \fn d\|_{L^\infty(L^2)}^2   + \frac{k_5}2 \| \skn d \fn d\|_{L^\infty(\Le)}^2\notag
  + \frac{1}{4\varepsilon} \left \| | \fn d |^2-1\right \|_{L^\infty (L^2)}^2
\\
& + \mu_1\left \Vert \fn d\cdot \syn v \fn d \right \Vert_{L^2(L^2)}^2  + \frac{\mu_4}{2} \|\syn v \|_{L^2(\Le)}^2+ \alpha \|\syn v\fn d\|_{L^2(\Le)}^2
\\& + \beta  \left (\delta^2 \| \Delta^2 \fn d\|^2_{L^2(\Le)}+ \left \| R_n\left (\partial_{\f h} F( \fn d , \nabla \fn d) - \di \partial_{\f S} F  ( \fn d , \nabla \fn d)\right )\right \|_{\Le}^2\right )    \\ &   +\beta\left (  \frac{\delta  }{2} \| \nabla \Delta \fn d \| _   {L^2(\Le)}^2 + \frac{\delta k_3 }{2} \| \Delta ( ( \di \fn d ) \fn d) \|_{L^2(\Le)}^2+ \frac{\delta}{2\varepsilon}\left \| \Delta | \fn d |^2 \right \|_{L^2}^2\right )
\\ & 
 +\beta \left (  \frac{\delta k_5 }{2} \| \Delta ( \skn d \fn d) \|_{L^2(\Le)}^2   +\frac{\delta k_4 }{2 } \| \Delta ( \fn d \cdot \curl \fn d ) \| _{ L^2(\Le)}^2  \right )
\\
&\leq 2 K
+\beta  c \left (\| \fn d \|_{\Hc}^{\nicefrac{8}{3}}+ \frac{ \delta}{\varepsilon}  \| \fn d \|_{\Hc}^{\nicefrac{16}{5}}\left \| |\fn d |^2-1\right \|_{L^2}^{\nicefrac{2}{5} }+ \delta^{\nicefrac{1}{3}} \| \fn d \|_{\Hc}^{4}  + \delta^{\nicefrac{3}{4}} \| \fn d \|_{\Hc}^{\nicefrac{7}{2}} \right )
\\
& \leq c \left ( 1+ \frac{1}{\delta^{\nicefrac{4}{3}}} +  \frac{\delta^{\nicefrac{3}{5}}}{\varepsilon^{\nicefrac{4}{5}}} + \frac{1 }{\delta^{\nicefrac{5}{3}}}+ \frac{1}{\delta}\right )
\,.
\end{split}
\end{align*}
 Here, we explicitly used the estimates $ \| \fn d \|^2_{\Hi}\leq \delta^{-1}$ and $ \left \| |\fn d |^2-1\right \|_{L^2}^2 \leq \varepsilon$. 
 By the choice $\varepsilon = \delta^{\nicefrac{7}{3}}$ we see
 \begin{align*}
 \| \Delta ( | \fn d |^2 -1)\|_{L^2(\Le)}^2 \leq  c\left  ( \frac{\varepsilon}{\delta} +\frac{\varepsilon}{\delta^{\nicefrac{1}{3}}}+ \frac{\varepsilon^{\nicefrac{1}{5}}}{\delta^{\nicefrac{2}{5}}} +   \frac{\varepsilon}{\delta^{\nicefrac{8}{3}}} + \frac{\varepsilon}{\delta^2} \right ) \leq c \left ( 1+ \frac{1}{\delta^{\nicefrac{1}{3}}}\right ) \,.
 \end{align*}
The assertion follows with the Gagliardo--Nirenberg inequality,
\begin{align*}
\left \| | \fn d|^2 -1\right \|_{L^{\nicefrac{8}{3}}(L^\infty)} + \left \| | \fn d |^2-1\right \|_{L^{\nicefrac{8}{3}}(W^{1,3})} &\leq c  \| \Delta ( | \fn d |^2 -1)\|_{L^2(\Le)}^{\nicefrac{3}{4}} \|  | \fn d |^2 -1\|_{L^\infty(\Le)}^{\nicefrac{1}{4}}\\ & \leq c (1 +  \delta^{\nicefrac{-1}{4}}) \delta^{\nicefrac{7}{12}} \,.
\end{align*}

\end{proof}
\begin{rem}
If we choose $\varepsilon = \delta^{\nicefrac{4}{3}}$, it can be shown that $ \delta^{\nicefrac{11}{3}} \| \fn d \|_{L^2(\Hg)}^2 $ is bounded. Together with the global boundedness of $\| \f d \|_{L^\infty(\f L^{12})}$, we can derive global boundedness of the term $\delta \| \Lap
 \fn d \|_{L^{\nicefrac{11}{3}}\Le}^2$ by the Gagliardo--Nirenberg estimates
 \begin{align*}
 \delta \left (\int_0^T\left (\| \Lap
 \fn d(t) \|_{\Le}^2\right )^{\nicefrac{11}{3}} \de t \right )^{\nicefrac{3}{11}}  \leq{}& \delta c \left (\int_0^T\left ( \| \f d (t)\|_{\Hg}^{\nicefrac{6}{11}} \| \f d (t)\|_{\f L^{12}}^{\nicefrac{16}{11}}\right )^{\nicefrac{11}{3}} \de t \right )^{\nicefrac{3}{11}}\\
  \leq{}& 
  c \left (\int_0^T\left ( \left (\delta^{\nicefrac{11}{3}} \| \fn d(t) \|_{\Hg}^2\right )^{\nicefrac{3}{11}}  \right )^{\nicefrac{11}{3}} \de t \right )^{\nicefrac{3}{11}}\| \f d \|_{L^\infty(\f L^{12})}^{\nicefrac{16}{11}}
 \\\leq{}& c\left (\delta^{\nicefrac{11}{3}} \| \fn d(t) \|_{L^2(\Hg)}^2 \right )^{\nicefrac{3}{11}}\| \f d \|_{L^\infty(\f L^{12})}^{\nicefrac{16}{11}}
\end{align*}
  Since locally one would expect an $\f L^\infty$-bound on the director (compare~\cite{bethuel,Dim3}), this will hopefully lead to additional local bounds on the defect measure $\mu_t$. 
\end{rem}
\subsection{Support of the defect angle measure}
For the defect angle measure $\nu^\infty$ in Proposition~\ref{lem:meas} we see that under the additional assumptions of Proposition~\ref{prop:apri2an}, the support is $\Se^{d^2-1}\times \Se^{d-1}_{\nicefrac{1}{2}}$ instead of $\Se^{d^2-1}\times \ov B_d$. Here, $\Se^{d-1}_{\nicefrac{1}{2}}$ is the sphere with radius $\frac{1}{2}$ in $B_d$, which corresponds to the unit sphere in untransformed coordinates.
\begin{proposition}
Under the assumptions of Proposition~\ref{prop:apri2an}, the defect measure $m$ is supported on $\Se^{d^2-1}\times \Se^{d-1}_{\nicefrac{1}{2}}$.

\end{proposition}
\begin{proof}
In convergence result~\eqref{convinS}, we take the test function $ f ( \f h , \f S) : = ( | \f h|^2-1) (1+| \f S|^2) = \frac{ | \f h|^2-1}{ | \f h|^2+1}(1+|\f	h|^2)(1+| \f S|^2)$. 
First we observe
\begin{align*}
\int_0^T\left (  | \fd d(t) |^2 -1 , 1+ | \nabla \fd d (t)| ^2  \right ) \de t \leq  c \|   | \fd d |^2 -1 \|_{L^{\nicefrac{8}{3}}(\f L^\infty)} ^2 \left (\| \nabla \fd d \|_{L^\infty(\Le)}^2+1\right )   \,.
\end{align*}
Thus, the term goes to zero for $\delta \ra 0$ due to Proposition~\ref{prop:apri2an}. 
On the other hand, we get
\begin{align*}
\int_0^T\left ( ( | \fd d(t) |^2 -1 ), 1+ | \nabla \fd d (t)| ^2  \right ) \de t  \longrightarrow \int_0^T \left (\int_\Omega \left \langle \nu^0_{(\f x ,t)} ,\frac{|\f d|^2-1}{|\f d|^2+1} \right \rangle \de \f x +
\int_{\ov \Omega} \left \langle \nu^\infty_{(\f x ,t)} , 2 | \f h|^2 -1 \right \rangle m_t(\de \f x) \right ) \de t \,.  
\end{align*}
Since $\f d$ has norm one a.\,e.~with respect to the Lebesgue measure, the first term on the right-hand side vanishes. This implies that the second term has to be zero as well. Consequently, the measure $\nu^\infty_{(\f x ,t)}
$ is  supported on the sphere with radius one-half, which corresponds to the unit sphere in $\R^d$. 
Thus, the measure $\nu^\infty_{(\f x ,t)}$ must be supported on $\Se^{d^2-1}\times \Se^{d-1}_{\nicefrac{1}{2}}$ for $m_t$ a.\,e.~$(\f x, t) \in \Omega\times (0,T)$. 

\end{proof}

\begin{remark}[Vanishing constants in the non-quadratic part of the Oseen--Frank energy]
Due to the additional $L^\infty$-estimate in space for the director, the existence of measure-valued solutions can also be granted, in the non coercive case, when the constants $k_3$, $k_4$ or $k_5$ vanish. 
The terms of the form $|\nabla \fd d |^2 | \fd d|^2 $ can be bounded by
\begin{align*}
\left \| \nabla \fd d | \fd d |\right \|_{L^{\nicefrac{8}{3}}(\Le)} \leq c 
%\int^T_0 \int_{\Omega}|\nabla  \fd d(\f x, t)|^2 | \fd d (\f x ,t)|^2 \de \f x \de t 
%\leq 
\| \nabla \fd d\|_{L^\infty(\Le)}^2 \| \f d \|_{L^{\nicefrac{8}{3}}(\f L^\infty )}^2 \,. 
\end{align*}
The convergence result of Proposition~\ref{lem:meas} still holds true. But due to the lack of $L^\infty$ regularity in time, the result 	of Proposition~\ref{lem:timemeas} is not valid any more.
The associated energy-inequality~\eqref{energyin} fails to hold and consequently, the associated weak strong uniqueness is not valid any more.

\end{remark}
\subsection{Energy inequality}

\begin{proposition}[Energy inequality]\label{thm:geeignet}
Let the assumptions of Theorem~\ref{thm:meas} and additionally Parodi's relation $ (\mu_2+\mu_3)=\lambda$ (see~\eqref{parodi}) be fulfilled. Then there exists a measure-valued solution 
to the Ericksen--Leslie equations (see Definition~\ref{def:meas}), which satisfies the energy inequality
 \begin{align}
\begin{split}
 &\frac{1}{2}\|\f v (t)\|_{\Le}^2 + \ll{\nu_t, F} + \ll{\mu_t , 1 }   + \intt{(\mu_1+\lambda^2)\|\f d\cdot \sy{v}\f d\|_{L^2}^2 +
  \mu_4 \|\sy{v}\|_{\Le}^2 }  \\
& +\intt{( \mu_5+\mu_6-\lambda^2)\|\sy{v}\f d\|_{\Le}^2  +  \|\f d \times \f q\|_{\Le}^2}
\\
& \qquad\qquad\qquad \leq  \left ( \frac{1}{2}\|\f v_0 \|_{\Le}^2 + \F( \f d_0)\right )
 + \intte{\langle \f g , \f v \rangle 
% +( \gamma ( \mu_2+ \mu_3) - \lambda ) \left ( \f q , \sy{v} \f d \right ) 
 }\, .
\end{split}
\label{energyin}
\end{align}
The time derivatives of the measure-valued solution possess the  regularity 
 $$\t \f v \in L^2(0,T; ( \f W^{1,3}_{0,\sigma})^*)\qquad\text{and}\qquad \t \f d \in L^2(0,T;\f L^{\nicefrac{3}{2}})\,.$$ 
\end{proposition}
\begin{proof}
The existence of measure-valued solutions follows from Theorem~\ref{thm:meas}. It is sufficient to show the energy inequality. 
Consider the inequality~\eqref{entro1}. Due to Parodi's relation, the last term on the right-hand side vanishes.
Passing to the limit in the approximate Galerkin space and using the weak lower semi-continuity of the appearing norms gives
\begin{align}
\begin{split}
&\frac{1}{2 }\| \fd v(t) \| _{\Le}^2   
%+\frac{k_1}{2} \|  \di \fd d(t)\|_{L^2}^2 +
%\frac{k_2}{2} \|  \curl \fd d(t)\|_{\Le}^2
+ \frac{\delta}{2} \| \Lap \fd d(t)\|_{\Le}^2
 %  + \frac{k_3}2 \| (\di \fd d(t) )\fd d(t)\|_{\Le}^2  
% \\ &
 %+\frac{k_4}2  \| \fd d(t) \cdot \curl \fd d(t)\|_{L^2}^2    + \frac{k_5}2 \|  {(\nabla \fd d(t))_{\skw}} \fd d(t)\|_{\Le}^2
 + \F(\fd d (t)) 
  + \frac{1}{4\varepsilon}\left  \| | \fd d(t) |^2-1\right \|_{L^2}^2+  \int_0^t \mu_4 \|\syd {v}\|_{\Le}^2\de s \\&+\int_0^t\mu_1\|\fd d\cdot \syd{v}\fd d\|_{L^2}^2
 +( \mu_5+\mu_6-\lambda(\mu_2+\mu_3))\|\syd{v}\fd d\|_{\Le}^2+ \|\fd q\|_{\Le}^2 \de s \\  &\qquad\qquad
\leq   \frac{1}{2}\| \f v_0\|_{\Le}^2 + \frac{\delta}2 \| \Lap \f d_{0}\|_{\Le}^2 + \int_{\Omega} F(\f d_{0} , \nabla \f d_0)  \de s  + \frac{1}{4\varepsilon} \left \| | \f d_0 |^2-1\right \|_{L^2}^2 + \int_0^t  \langle \f g , \fd v \rangle\de s  \,.
\end{split}
\label{ungleichung}
\end{align}
On the right-hand side of the above inequality, the initial values $(\f v_0,\f d_0)$ are inserted. This can be done due to the strong convergences
\begin{align*}
P_n \f v_0 \ra \f v_0 \quad\text{in }\Ha \quad \text{and} \quad R_n \f d_o \ra \f d_0 \quad\text{in }\Hc\,.
\end{align*}
For the limiting process in the nonlinear energy, we refer to the calculations in Proposition~\ref{prop:apri}.

The aim is now to pass to the limit for vanishing regularization in the above inequality~\eqref{ungleichung}.
The penalisation-term on the right hand side of~\eqref{ungleichung} vanish since $\f d_0$ has norm one a.\,e.~and the penalization term on the left-hand side of~\eqref{ungleichung} can be estimated from below by zero. Since $\|\Lap \f d_0\|_{\Le}\leq c$, we get $\delta \| \Lap \f d_0 \|_{\Le}^2 \ra 0$ for $\delta \ra 0$.  

   For positive smooth functions  $ \phi \in \C_c^\infty ( 0,T)$ with $\phi (t) \geq 0$ for all $t\in [0,T]$
it follows from Theorem~\ref{thm:young} that
\begin{align*}
\lim _{\delta\ra 0} \int_0^T \phi(t) \left ( \frac{\delta}{2} \| \Lap \fd d \|_{\Le}^2+ \F(  \fdk d (t)) \right ) \de t  &= 
% \int_0^T \phi(t) \liminf _{k\ra \infty} (\F(  \fd d (t))+ \delta \| \Lap \fd d \|_{\Le}^2)  \de t  \\ &\geq
   \int_0^T \phi(t)\left (\frac{1}{2}\ll{\mu_t, 1}+  \ll{\nu_t , F }  \right ) \de t\,.
\end{align*}
The fundamental lemma of variational calculus gives
 \begin{align*}
\lim _{\delta\ra 0}  \left ( \frac{\delta}{2}\| \Lap \fdk d (t) \|_{\Le}^2+ \F(  \fdk d (t))\right ) =   \frac{1}{2}\ll{\mu_t, 1} +  \ll{\nu_t , F }\, \quad\text{a.e.~in $(0,T)$.} 
\end{align*} 

With the weak convergence of the appearing sequences and the  weak-lower semi-con\-tin\-u\-ity of the appearing norms, we can pass to the limit in the regularisation parameter and attain
 \begin{align}
\begin{split}
 &\frac{1}{2}\|\f v (t)\|_{\Le}^2 + \ll{\nu_t, F} + \ll{\mu_t , 1 }   + \int_0^t\left ({\mu_1\|\f d\cdot \sy{v}\f d\|_{L^2}^2 +
  \mu_4 \|\sy{v}\|_{\Le}^2 } \right ) \de s \\
& +\int_0^t\left ({( \mu_5+\mu_6-\lambda(\mu_2+\mu_3))\|\sy{v}\f d\|_{\Le}^2  +  \| \f q\|_{\Le}^2}\right )\de s
\\
& \qquad\qquad\qquad \leq  \left ( \frac{1}{2}\|\f v_0 \|_{\Le}^2 + \F( \f d_0)\right )
 + \int_0^t{\langle \f g , \f v \rangle 
% +( \gamma ( \mu_2+ \mu_3) - \lambda ) \left ( \f q , \sy{v} \f d \right ) 
 }\de s\, .
\end{split}
\label{energyin2}
\end{align}

Testing the director equation of the regularized system with $\f d \phi$, where $\phi\in \C^\infty_c(\Omega\times (0,T))$, gives
\begin{align*}
\int_0^t \frac{1}{2}( \t | \fd d (t) |^2 + \fd v(t) \cdot  \nabla | \fd d(t) |^2 , \phi(t) ) \de t  + \int_0^t  ( \lambda\fd d  (t) \cdot ( \nabla \fd v (t))_{\sym} \fd d(t)+ \f q(t) \cdot \fd d(t) , \phi(t))  \de t = 0 \,.  
\end{align*} 
Using two integrations by parts and due to the fact that the weak derivative of a constant function is zero, we get
\begin{align*}
- \int_0^t \frac{1}{2}\left ((  | \fd d (t) |^2-1, \t \phi(t) )   +  ( | \fd d(t) |^2-1  , \di ( \fd v (t) \phi(t))  )\right ) \de t  \ra 0 
%\\ + \int_0^t  ( \lambda\fd d  (t) \cdot ( \nabla \fd v (t))_{\sym} \fd d(t)+ \f q(t) \cdot \fd d(t) , \phi(t))  \de t = 0 \,.  
\end{align*} 
for $\delta\ra0$ since 
for vanishing regularization, we already established that $| \f d|=1$ a.\,e.~in $\Omega\times (0,T)$. 
%Hence, the first two terms of the forgoing equation vanish.
 Thus, it holds
\begin{align*}
\int_0^t  ( \lambda\f d (t) \cdot ( \nabla \f v (t))_{\sym} \f d(t)+ \f q(t) \cdot \f d(t) , \phi(t))  \de t = 0 \,
\end{align*}
for all $\phi \in \C^\infty_c(\Omega\times (0,T))$. Since the above terms are in $L^1(\Omega\times (0,T))$
the equality holds a.\,e.~in $\Omega\times (0,T)$. The a priori estimate~\eqref{entrodiss} implies that both terms are bounded in $L^2(0,T;\Le)$ and their norms must  coincide,
\begin{align}
\|\lambda  \f d \cdot \sy v \f d \|_{L^2(\Le)} = \| \f q \cdot \f d \|_{L^2(\Le)} \,.\label{eq:1}
\end{align}
Since $|\f d |=1 $ a.\,e.~in $\Omega\times (0,T)$, we conclude 
\begin{align}
\| \f q \|_{\Le} ^2 = ( \f q , \f q ) = ( \f q , | \f d |^2 \f q) = \left ( \f q , \left (| \f d|^2 I - \f d \otimes \f d \right ) \f q \right ) + \left ( \f q \cdot \f d , \f q \cdot \f d \right ) = \| \f d \times \f q \|_{\Le}^2  + \| \f q \cdot \f d \|_{L^2}^2\,.\label{eq:2}
\end{align}
Inserting \eqref{eq:1} and~\eqref{eq:2} into~\eqref{energyin2} gives the asserted energy inequality~\eqref{energyin}. 

The estimate~\eqref{esttime}, the weak convergences~\eqref{rtimev} and \eqref{rtimed} and the weak-lower semi-continuity of the norms give the asserted regularity of the time derivatives.

\end{proof}

\addcontentsline{toc}{section}{References}

\small
%\addcontentsline{toc}{chapter}{Literatur}
\bibliographystyle{abbrv}

\end{document}